\documentclass[11pt]{amsart}
\usepackage[margin=2.5cm]{geometry}

\RequirePackage{amsmath,amssymb,amsthm, amscd, comment,mathtools}

\usepackage{kantlipsum}
\usepackage[all]{xy}
\usepackage{mathabx}
\usepackage{graphicx}
\usepackage{tikz-cd}
\usepackage[colorinlistoftodos]{todonotes}
\usepackage[pagebackref,breaklinks,colorlinks,linkcolor=red,anchorcolor=red,citecolor=blue]{hyperref}

\usepackage{rotating}

\calclayout

\newtheorem{theorem}[subsection]{Theorem}
\newtheorem{proposition}[subsection]{Proposition}
\newtheorem{corollary}[subsection]{Corollary}
\newtheorem{lemma}[subsection]{Lemma}
\newtheorem{conjecture}[subsection]{Conjecture}

\theoremstyle{definition}

\newtheorem{definition}[subsection]{Definition}
\newtheorem{remark}[subsection]{Remark}

\numberwithin{equation}{subsection}

\makeatletter

\begin{document}

\title[Cohomological Milnor formula  and Saito's conjecture]{Cohomological Milnor formula and Saito's conjecture on characteristic classes }

\author[Enlin Yang]{Enlin Yang${}^{\dag}$}
\thanks{${}^{\dag}$Enlin Yang is the corresponding author.}
\thanks{${}^{\dag}$yangenlin@math.pku.edu.cn}
\thanks{${}^\dag$School of Mathematical Sciences, Peking University, No.5 Yiheyuan Road Haidian District.,
	Beijing, 100871, P.R. China.}

\author[Yigeng Zhao]{Yigeng Zhao${}^\ddag$}
\thanks{${}^\ddag$zhaoyigeng@westlake.edu.cn}
\thanks{${}^\ddag$Institute for Theoretical Sciences, Westlake University, No.600 Dunyu Road, Xihu District, Hangzhou, Zhejiang, 310040, P.R. China.}

\date{\today}
\begin{abstract}
We confirm the quasi-projective case of Saito's conjecture [Invent.Math.207,597–695 (2017)], namely that the cohomological characteristic classes defined by Abbes and Saito can be computed in terms of the characteristic cycles.
	
We construct a  cohomological characteristic class supported on the  
non-acyclicity locus of a separated morphism relatively to a constructible sheaf.
As applications of the functorial properties of this class, we prove cohomological analogs of the Milnor formula and the conductor formula for constructible sheaves on 
(not necessarily smooth) varieties. 
\end{abstract}

\maketitle
\tableofcontents

\section{Introduction}
\subsection{}
Let $h:X\to {\rm Spec}k$ be a separated morphism of finite type over a perfect field $k$. Let $\Lambda$ be a finite local ring such that the characteristic of the residue field is invertible in $k$.
Let $\mathcal K_{X/k}=Rh^!\Lambda$.
For a constructible complex $\mathcal F$ of $\Lambda$-modules of finite tor-dimension on $X$, the cohomological characteristic class $C_{X/k}(\mathcal F)\in H^0(X,\mathcal K_{X/k})$ is introduced by Abbes and Saito in \cite{AS07} by using Verdier pairing (cf. \cite{Gro77})\footnote{See also \cite{KS90} for the transcendental setting.}.
Using ramification theory, Abbes and Saito calculate the cohomological characteristic classes for rank 1 sheaves under certain ramification conditions in \cite{AS07}. However, the calculation for general constructible \'etale sheaves remains an outstanding question in ramification theory.
Later, using the singular support defined by Beilinson \cite{Bei16}, 
Saito \cite{Sai17a} constructs the characteristic cycles of constructible \'etale sheaves and further proposes the following conjecture.
\begin{conjecture}[Saito, {\cite[Conjecture 6.8.1]{Sai17a}}]\label{introconj:Scc}
	Let $X$ be a closed sub-scheme of a smooth scheme over a perfect field $k$. 
	Let $\mathcal F$ be a constructible complex of $\Lambda$-modules of finite tor-dimension on $X$.
	Let	 $cc_{X}(\mathcal F):=cc_{X,0}(\mathcal F)\in {\rm CH}_0(X)$ be the characteristic class defined in \cite[Definition 6.7.2]{Sai17a} via the characteristic cycle of $\mathcal F$. 
	Then we have
	\begin{align}\label{introSJeq}
		C_{X/k}(\mathcal F)={\rm cl}(cc_{X}(\mathcal F))\quad{\rm in}\quad H^0(X,\mathcal K_{X/k}),
	\end{align}
	where ${\rm cl}: {\rm CH}_0(X)\to H^0(X,\mathcal K_{X/k})$ is the cycle class map.
\end{conjecture}
Saito's conjecture says that the cohomological characteristic class can be computed in terms of the characteristic cycle.
Note that the two involved ramification invariants in Conjecture \ref{introconj:Scc} are defined in quite different ways. The characteristic cycle is characterized by the Milnor formula  \cite[Theorem 5.9]{Sai17a}, while the cohomological characteristic class in some sense is defined via the categorical trace (cf. Section \ref{sec:CCC}).
In the characteristic zero case, the equality \eqref{introSJeq} on a complex manifold is the microlocal index formula proved by Kashiwara and Schapira \cite[9.5.1]{KS90}. However, we don't know such a microlocal description for characteristic cycles in positive characteristic\footnote{Abe \cite{Abe21} obtains certain microlocal descriptions for characteristic cycles using $\infty$-categories.}.
One of our main goals of this paper is to show the quasi-projective case of Saito's conjecture.
\begin{theorem}[Theorem \ref{thm:sConj}]\label{thm:sConj-inIntro}
	Conjecture \ref{introconj:Scc} holds for  any smooth and quasi-projective scheme $X$ over a perfect field $k$ of characteristic $p>0$.
\end{theorem}

When $X$ is projective and smooth over a finite field $k$ of characteristic $p$, the cohomology group $H^0(X,\mathcal K_{X/k})$ is highly non-trivial. For example, if $\Lambda=\mathbb Z/\ell^m$ with $\ell\neq p$, then we have
$H^0(X,\mathcal K_{X/k})\simeq H^1(X,\mathbb Z/\ell^m)^\vee \simeq \pi_1^{\rm ab}(X)/\ell^m$ (cf. \cite[Theorem 1.12]{Ras95}).

\subsection{}
Our approach to Saito's conjecture is the fibration method, which leans on the construction of the relative version of cohomological characteristic classes over a general base scheme. 
To achieve this, it is necessary to impose additional transversality conditions on the structure morphism. Let $S$ be a Noetherian scheme and ${\rm Sch}_{S}$ the category of separated schemes of finite type over $S$. 
Let $\Lambda$ be a Noetherian ring such that $m\Lambda=0$ for some integer $m$ invertible on $S$. 
Let $h: X\to S$ be a separated morphism of finite type, $\mathcal K_{X/S}=Rh^!\Lambda$ and $\mathcal F\in D_{\rm ctf}(X,\Lambda)$.
In fact, under certain smooth and transversality conditions on $h$, we  introduce the relative (cohomological) characteristic class $C_{X/S}(\mathcal F)\in H^0(X,\mathcal K_{X/S})$ in \cite[Definition 3.6]{YZ18}. 
It is further generalized
to any separated morphism $h:X\to S$ which is (universally) locally acyclic relatively to $\mathcal F$ by using categorical traces \cite[2.20]{LZ22}.
We also define relative characteristic classes in a more general case that only assumes local acyclicity away from small closed subsets. Indeed, if $Z\subseteq X$ is a closed subscheme 
such that $H^0(Z,\mathcal K_{Z/{S}})=H^1(Z,\mathcal K_{Z/{S}})=0$, and 
if $X\setminus Z\to S$ is (universally) locally acyclic relatively to $\mathcal F|_{X\setminus Z}$, then the relative characteristic class $C_{X/S}(\mathcal F)\in H^0(X,\mathcal K_{X/S})$ remains well-defined (cf. Definition \ref{def:rcccCodBig}). 
We prove the following fibration formula for cohomological characteristic classes.
\begin{theorem}[Theorem \ref{thm:SJ:MFforTclass}]\label{thm:SJ:MFforTclass-intro}
	Let $Y$ be a smooth connected curve over a perfect field $k$ of  characteristic $p>0$.  
	Let $\Lambda$ be a finite local ring such that the characteristic of the residue field is invertible in $k$. 
	Let $f:X\to Y$ be a separated morphism of finite type and $Z\subseteq |X|$ be a finite set of closed points. 
	Let $\mathcal F\in D_{\rm ctf}(X,\Lambda)$ such that $f|_{X\setminus Z}$ is universally locally acyclic relatively to $\mathcal F|_{X\setminus Z}$. 
	Then we have 
	\begin{align}
		\label{eq:SJ:milclassepsilon-intro}
		C_{X/k}(\mathcal F)&=c_1(f^\ast\Omega_{Y/k}^{1,\vee})\cap C_{X/Y}(\mathcal F)-\sum_{x\in Z}{\rm dimtot} R\Phi_{\bar{x}}(\mathcal F,f)\cdot [x]\quad{\rm in}\quad H^0(X,\mathcal K_{X/k}),
	\end{align}
	where $R\Phi(\mathcal F,f)$ is  the complex of vanishing cycles and ${\rm dimtot}={\rm dim}+{\rm Sw}$ is the total dimension. 
\end{theorem}
As a consequence, we have the following cohomological version of Grothendieck-Ogg-Shafarevich formula:
\begin{corollary}[{Corollary \ref{cor:GOSend}}]\label{introCorGOS}
	Let $X$ be a smooth connected curve over  a perfect field $k$ of  characteristic $p>0$. 
	Let $\mathcal F\in D_{\rm ctf}(X,\Lambda)$ and $Z\subseteq X$  a finite set of closed points such that the cohomology sheaves of $\mathcal F|_{X\setminus Z}$ are locally constant. Then we have
	\begin{align}
		C_{X/k}(\mathcal F)={\rm rank}\mathcal F\cdot c_1(\Omega_{X/k}^{1,\vee})-\sum_{x\in Z}{ a}_x(\mathcal F)\cdot[x]\quad {\rm in}\quad H^0(X,\mathcal K_{X/k}),
	\end{align}
	where $a_x(\mathcal F)={\rm rank}\mathcal F|_{\bar\eta}-{\rm rank}\mathcal F_{\bar x}+{\rm Sw}_x\mathcal F$ is the Artin conductor and $\eta$ is the generic point of $X$.
\end{corollary}
Corollary \ref{introCorGOS} confirms Saito's conjecture \ref{introconj:Scc} in the curve case.
The quasi-projective case of Saito's conjecture \ref{introconj:Scc} follows from Theorem \ref{thm:SJ:MFforTclass-intro} by using the fibration method.
\subsection{}For the proof of Theorem \ref{thm:SJ:MFforTclass-intro}, the key observation is that  
$C_{X/k}(\mathcal F)-c_1(f^\ast\Omega_{Y/k}^{1,\vee})\cap C_{X/Y}(\mathcal F)$ is zero on $X\setminus Z$ (cf. Theorem \ref{thm:changebaseHomo}). Hence the difference comes from a cohomology class in $H^0(Z,\mathcal K_{Z/k})$. In this paper we construct a canonical lifting of $C_{X/k}(\mathcal F)-c_1(f^\ast\Omega_{Y/k}^{1,\vee})\cap C_{X/Y}(\mathcal F)$  in $H^0(Z,\mathcal K_{Z/k})$, where $Z$ contains the non-acyclicity locus of $f:X\to Y$ relatively to $\mathcal F$. 
 Later for proving a cohomological version of the Milnor formula \cite[Theorem 5.9]{Sai17a}, we need this lifting in a relative version as follows.
\subsection{}
Consider a commutative diagram in ${\rm Sch}_S$:
\begin{align}\label{intro:deltadiag}
	\begin{gathered}
		\xymatrix{
			Z\ar@{^(->}[r]^-\tau&X\ar[rr]^-f\ar[rd]_-h&&Y\ar[ld]^-g\\
			&&S&
		},
	\end{gathered}
\end{align}
where $\tau: Z\to X$ is a closed immersion and $g$ is a smooth morphism of relative dimension $r$.
We define an object $\mathcal K_{X/Y/S}$ on $X$ sitting in a distinguished triangle (cf. \eqref{eq:distriK})
\begin{align}
	\mathcal K_{X/Y}\to\mathcal K_{X/S}\to \mathcal K_{X/Y/S}\xrightarrow{+1}.
\end{align}
Let $\mathcal F\in D_{\rm ctf}(X,\Lambda)$ such that $X\setminus Z\to Y$ is universally locally acyclic relatively 
to $\mathcal F|_{X\setminus Z}$ and that $X\to S$ is universally locally acyclic relatively to $\mathcal F$.
We introduce a new cohomological characteristic class  $\widetilde{C}_{X/Y/S}^{Z}(\mathcal F)\in H_Z^0(X,\mathcal K_{X/Y/S})$ supported on $Z$  in Definition \ref{def:milclass}, 
which measures the difference 
$C_{X/S}(\mathcal F)- c_{r}(f^*\Omega_{Y/S}^{1,\vee})\cap C_{X/Y}(\mathcal F)$. 
If $H^0(Z,\mathcal K_{Z/{Y}})=H^1(Z,\mathcal K_{Z/{Y}})=0$, then 
$H^0(Z,\mathcal K_{Z/S})\simeq H^0_Z(X,\mathcal K_{X/Y/S})$ and
$\widetilde{C}_{X/Y/S}^{Z}(\mathcal F)$ defines a class in $H^0(Z,\mathcal K_{Z/S})$, which is denoted by $C_{X/Y/S}^{Z}(\mathcal F)$. 
We have the following formula:
\begin{theorem}[Fibration formula, Theorem \ref{conj:milclass}]\label{introconj1}
	If $H^0(Z,\mathcal K_{Z/{Y}})=H^1(Z,\mathcal K_{Z/{Y}})=0$,
	then we have
	\begin{align}\label{eq:intro2}
		C_{X/S}(\mathcal F)= c_{r}(f^*\Omega_{Y/S}^{1,\vee})\cap C_{X/Y}(\mathcal F)+
		C_{X/Y/S}^{Z}(\mathcal F) \quad {\rm in }\quad H^{0}(X,\mathcal K_{X/S}).
	\end{align}
\end{theorem}
We call $C_{X/Y/S}^{Z}(\mathcal F) $ (resp. $\widetilde{C}_{X/Y/S}^{Z}(\mathcal F)$)  the non-acyclicity class of $\mathcal F$. We expect that this class on (not necessarily smooth) varieties could play a role as the characteristic cycles on smooth varieties. 

If $Z$ is empty, we prove \eqref{eq:intro2} in  Theorem \ref{thm:changebaseHomo}.
In the case where $f={\rm id}$ and $S={\rm Spec}k$,  the cohomology sheaves of $\mathcal F|_{X\setminus Z}$ are locally constant on $X\setminus Z$ since ${\rm id}:X\setminus Z\to X\setminus Z$ is universally locally acyclic  relatively to $\mathcal F|_{X\setminus Z}$. Then \eqref{eq:intro2} is Abbes-Saito's localization formula  \cite[Proposition 5.2.3]{AS07}.

\subsection{} 
Now we summarize the functorial properties for the non-acyclicity classes.
For simplicity, we denote such a diagram \eqref{intro:deltadiag} by $\Delta=\Delta^Z_{X/Y/S}$ 
and  $\widetilde{C}_{X/Y/S}^{Z}(\mathcal F)$ by $C_{\Delta}(\mathcal F)$.

\begin{proposition}[Proposition \ref{prop:LocBC} and Proposition \ref{thm:ppForNTclass}]\label{introthmforfuncna}
	Let $\mathcal F\in D_{\rm ctf}(X,\Lambda)$. 
	Assume that $Y\to S$ is smooth, $X\setminus Z\to Y$ is universally locally acyclic relatively to $\mathcal F|_{X\setminus Z}$ and that $X\to S$ is universally locally acyclic relatively to $\mathcal F$.
	\begin{itemize}
		\item[(1)]
		Let $b: S'\to S$ be a morphism of Noetherian schemes. Let $\Delta'=\Delta_{X'/Y'/S'}^{Z'}$ be the base change of $\Delta=\Delta^Z_{X/Y/S}$ by $b: S'\to S$. Let $b_X:X'=X\times_SS'\to X$  be the base change of $b$ by $X\to S$. Then
		we have
		\begin{align}
			b_{X}^\ast C_{\Delta}(\mathcal F)=C_{\Delta'}(b_X^\ast\mathcal F)\quad{\rm in}\quad H^0_{Z'}(X',\mathcal K_{X'/Y'/S'}),
		\end{align}
		where $b_{X}^\ast:H^0_{Z}(X,\mathcal K_{X/Y/S})\to H^0_{Z'}(X',\mathcal K_{X'/Y'/S'})$ is the induced pull-back morphism.
		\item[(2)]
		Consider a  diagram $\Delta'=\Delta^{Z'}_{X'/Y/S}$. Let $s:X\to X'$ be a proper morphism over $Y$ such that $Z\subseteq s^{-1}(Z')$.
		Then
		we have
		\begin{align}\label{introeqpushforward}
			s_\ast (C_{\Delta}(\mathcal F))=C_{\Delta'}(Rs_\ast\mathcal F)\quad{\rm in}\quad H^0_{Z'}(X',\mathcal K_{X'/Y/S}),
		\end{align}
		where $s_\ast:H^0_{Z}(X,\mathcal K_{X/Y/S})\to H^0_{Z'}(X',\mathcal K_{X'/Y/S})$ is the induced  push-forward morphism.
	\end{itemize}
\end{proposition}

If $Z$ is a finite set of closed points, the non-acyclicity class can be calculated as follows.
\begin{theorem}[Theorem \ref{thm:MF}]\label{introconj2}
	Let $Y$ be a smooth curve over a perfect field $k$ of  characteristic $p>0$.
	Let $\Lambda$ be a finite local ring such that the characteristic of the residue field is invertible in $k$.   Let    $f: X\to Y$ be a separated morphism of finite type and $x\in|X|$ a closed point. Let $\mathcal F\in D_{\rm ctf}(X,\Lambda)$ such that $f|_{X\setminus\{x\}}$ is universally locally acyclic relatively to $\mathcal F|_{X\setminus\{x\}}$. 
	Then we have 
	\begin{align}\label{introMFeq}
		C_{X/Y/k}^{\{x\}}(\mathcal F)=-{\rm dimtot} R\Phi_{\bar{x}}(\mathcal F,f)\quad {\rm in}\quad  \Lambda=H^0_{x}(X,\mathcal K_{X/k}).
	\end{align}
\end{theorem}
The above theorem gives a cohomological analog of the Milnor formula proved by Saito in \cite[Theorem 5.9]{Sai17a}, and does not involve the smoothness assumption on $X$.
Its proof is based on the pull-back property of the non-acyclicity classes (cf. Proposition \ref{introthmforfuncna}.(1)) together with the method in  \cite{Abe22}.  
 Theorem \ref{thm:SJ:MFforTclass-intro} follows from \eqref{eq:intro2} together with the cohomological Milnor formula  \eqref{introMFeq}.
The construction of non-acyclicity classes in this paper can be applied to motivic categories 
with six-functor formalism.
If there is a generalization of the argument in \cite{Abe22} to the $p$-adic cohomology theory (cf. \cite[3.2]{Abe22}), we could expect that the construction will  lead to a solution of Deligne's conjecture on the Milnor formula in the mixed characteristic case \cite[Conjecture 1.9]{Del72}.

\subsection{}
As an application of Proposition \ref{introthmforfuncna}.(2), we prove a cohomological version of the conductor formula \cite[Theorem 2.2.3]{Sai21} (also of Bloch's conjecture on the conductor formula for constant sheaves \cite{Blo87}).
\begin{theorem}[Theorem \ref{cor:conductorF}]\label{introCor:conductorF}
	Let $Y$ be a smooth connected curve over a perfect field $k$ of  characteristic $p>0$.
	Let $\Lambda$ be a finite local ring such that the characteristic of the residue field is invertible in $k$. 
	Let    $f: X\to Y$ be a proper morphism over $k$ and $y\in|Y|$ a closed point. Let $\mathcal F\in D_{\rm ctf}(X,\Lambda)$. Assume that $f|_{X\setminus f^{-1}(y)}$ is  universally locally acyclic relatively to $\mathcal F|_{X\setminus f^{-1}(y)}$.
	Then we have 
	\begin{align}\label{introeq:thm:MF1}
		f_\ast \widetilde{C}_{X/Y/k}^{f^{-1}(y)}(\mathcal F)=-{ a}_y(R f_\ast\mathcal F)\quad {\rm in}\quad  \Lambda=H^0_{y}(Y,\mathcal K_{Y/k}).
	\end{align}
\end{theorem}
In \cite[Theorem 2.2.3]{Sai21}, Saito assumes that $f$ is projective since the global index formula for the characteristic cycle is only known  for projective schemes\footnote{Recently, Abe \cite{Abe21} proves the index formula for proper schemes by using $\infty$-categories.}. However, using the proper push-forward property for cohomological characteristic classes, we could drop this assumption in Theorem \ref{introCor:conductorF}. This is one of the advantages by using cohomological arguments.

\subsection{}We derive the cohomological Milnor formula  from the pull-back property and a deformation property of constructible \'etale sheaves together with a simple fact that $C_{X/X/k}^{\{x\}}(\mathcal G_{x})={\rm rank}\mathcal G_{x}$ for any smooth scheme $X$ of dimension $\geq 1$ and any constructible sheaf $\mathcal G_x$ supported on a closed point $x\in|X|$.  The cohomological conductor formula follows from the proper push-forward property. It would be interesting to see whether the non-acyclicity classes are characterized by Proposition \ref{introthmforfuncna} and the above formula for skyscraper constructible sheaves.

\subsection{}This article is organized as follows. 
In Section \ref{sec:TC}, we introduce the definition of transversality conditions, which is crucial for defining the non-acyclicity classes.  
In Section \ref{sec:CCC}, we review the construction of the relative characteristic class and its functorial properties. Then we  prove \eqref{eq:intro2} under the condition $Z=\emptyset$ in Theorem \ref{thm:changebaseHomo}.
In Section \ref{sec:CCC2}, we  construct the non-acyclicity classes and 
 study their functorial properties. Then we prove the fibration formula \eqref{eq:intro2} 
 in Section \ref{sec:fib}.
In Section \ref{sec:applications}, we show the cohomological Milnor formula \eqref{introMFeq} and deduce the cohomological conductor formula \eqref{introeq:thm:MF1}. As an application of the cohomological Milnor formula, we establish an induction formula for the cohomological characteristic classes and then prove the quasi-projective case of Saito's conjecture in Theorem \ref{thm:sConj}.

\subsection*{Acknowledgments} 
The authors would like to express their sincere gratitude to Professor Takeshi Saito for his careful reading and many improvement suggestions and thank him for suggesting the terminology ``non-acyclicity class".  The authors would like to thank Professor Tomoyuki Abe for the helpful discussion on the cohomological Milnor formula \eqref{introMFeq}, which was stated as a conjecture in a draft version of this paper. 
He kindly suggests the authors to study the pull-back formula for non-acyclicity classes and then use a similar argument \cite{Abe22} to conclude. 
The authors would like to thank Professor Georg Tamme for the helpful discussion and suggestions on $\infty$-categories.
The authors thank Denis-Charles Cisinski, Haoyu Hu, Fangzhou Jin, Moritz Kerz,  Ruochuang Liu, Guozhen Wang, Jiangnan Xiong and Weizhe Zheng for their comments. The authors appreciate the anonymous referees for careful reading and numerous valuable comments and suggestions, which have greatly improved the manuscript.
This work was partially supported by the National Key R{\&}D Program of China (Grant No.2021YFA1001400), NSFC Grant No.11901008, NSFC Grant  No.12271006 and the New Cornerstone Science Foundation.

\subsection*{Notation and Conventions}
\begin{enumerate}
\item Let $S$ be a Noetherian scheme and ${\rm Sch}_{S}$ the category of separated schemes of finite type over $S$. 
 Let $\Lambda$ be a Noetherian ring such that $m\Lambda=0$ for some integer $m$ invertible on $S$ unless otherwise stated explicitly.
\item For any scheme $X\in {\rm Sch}_S$, 
we denote by $D(X,\Lambda)$ the derived category of \'etale sheaves of $\Lambda$-modules on $X$ and by
$D_{\rm ctf}(X,\Lambda)$ the derived category of complexes of \'etale sheaves of $\Lambda$-modules of finite tor-dimension with constructible cohomology groups on $X$.

\item For any morphism $f:X\to Y$ in ${\rm Sch}_S$, we denote by $D_{\rm ctf}(X/Y,\Lambda)\subseteq D_{\rm ctf}(X,\Lambda)$  the full subcategory consisting of objects $\mathcal F$ such that $f$ is universally locally acyclic relatively to $\mathcal F$ (also say $\mathcal F$ is universally locally acyclic over $Y$). We define $K_0(X/Y,\Lambda)$ to be the Grothendieck group of $D_{\rm ctf}(X/Y,\Lambda)$.
\item  For any separated morphism $f:X\to Y$ in ${\rm Sch}_S$, we use the following notation
\begin{align}\nonumber
\mathcal K_{X/Y}=Rf^!\Lambda,\quad D_{X/Y}(-)=R\mathcal Hom(-, \mathcal K_{X/Y}).
\end{align}

\item For  $\mathcal F\in D_{\rm ctf}(X,\Lambda)$ and $\mathcal G\in D_{\rm ctf}(Y,\Lambda)$ on $S$-schemes $X$ and $Y$  respectively, $\mathcal F\boxtimes_S^L\mathcal G$ denotes ${\rm pr}_1^\ast\mathcal F\otimes^L{\rm pr}_2^\ast\mathcal G$ on $X\times_SY$.
\item To simplify our notation, we omit to write $R$ or $L$ to denote the derived functors unless otherwise stated explicitly or for $R\mathcal Hom$.
\end{enumerate}

\section{Transversality condition}\label{sec:TC}

\subsection{}\label{def:gtrans}
In this section, we introduce a generalized version of the transversality condition studied 
in \cite{Sai17a} and investigate its relation with universal local acyclicity. This generalized version is crucial for defining the non-acyclicity classes.  
Consider the following cartesian diagram in ${\rm Sch}_S$:
\begin{align}\label{eq:carGtrans}
\begin{gathered}
\xymatrix{
X\ar@{}[rd]|\Box\ar[d]_p\ar[r]^-i&Y\ar[d]^-f\\
W\ar[r]^-{\delta}&T.
}
\end{gathered}
\end{align}
Let $\mathcal F\in D_{\rm ctf}(Y,\Lambda)$ and $\mathcal G\in D_{\rm ctf}(T,\Lambda)$. 
Recall that there is  a  base change morphism
\begin{equation}\label{eq:pullbackforanyGdefTrans}
	p^\ast\delta^! \mathcal G\xrightarrow{}i^!f^\ast \mathcal G,
\end{equation}	
which is adjoint to
$	i_! p^\ast \delta^!\mathcal G \xrightarrow[\simeq]{\rm b.c}f^\ast \delta_!\delta^!\mathcal G\xrightarrow{{\rm adj}}f^\ast\mathcal G$.
We define a morphism $c_{\delta,f,\mathcal F,\mathcal G}$
to be the  composition 
\begin{align}\label{eq:carGtrans1}
	\begin{split}
		c_{\delta,f,\mathcal F,\mathcal G}:i^\ast \mathcal F\otimes^L p^\ast \delta^!\mathcal G&\xrightarrow{id\otimes \eqref{eq:pullbackforanyGdefTrans}}i^\ast \mathcal F\otimes^L i^! f^\ast\mathcal G \xrightarrow{ (a)}i^!(\mathcal F\otimes^L f^\ast\mathcal G),
	\end{split}
\end{align}
where $(a)$ is adjoint to the following composition (cf. \cite[(8.8)]{Sai17a})
\begin{equation}\label{eq:defsaitoCFG}
	i_!(i^\ast\mathcal F\otimes^L i^!f^\ast \mathcal G)\xrightarrow[\simeq]{\rm proj.formula}\mathcal F\otimes^L i_!i^! f^\ast \mathcal G\xrightarrow{\rm adj}\mathcal F\otimes^L f^\ast \mathcal G.
\end{equation}
Note that if $i={\rm id}$ and $\delta$ is a closed immersion, then $c_{\delta,f,\mathcal F,\mathcal G}$ is induced by  (cf. definition of \eqref{eq:pullbackforanyGdefTrans})
\begin{align}\label{delta!ast}
	\delta^!\simeq \delta^\ast \delta_! \delta^!\xrightarrow{\rm adj} \delta^\ast.
\end{align}
We also put $c_{\delta,f,\mathcal F}:=c_{\delta,f,\mathcal F,\Lambda}: i^\ast\mathcal F\otimes^L p^\ast\delta^!\Lambda\to i^!\mathcal F$.
If $c_{\delta,f,\mathcal F}$ is an isomorphism, then we say the cartesian diagram \eqref{eq:carGtrans} is {\it $\mathcal F$-transversal} (and also say that the morphism $\delta$ is {\it $\mathcal F$-transversal}).
For the special case where $f={\rm id}$ in \eqref{eq:carGtrans}, it agrees with Saito's definition \cite[Definition 8.5]{Sai17a}.
The basic properties of $c_{\delta,f,\mathcal F,\mathcal G}$ are summarized in the following two lemmas.
\begin{lemma}\label{lem:TransAforC}
Consider the following commutative diagram in ${\rm Sch}_S$:
	\begin{align}\label{eq:TransAforC1}
		\begin{gathered}
			\xymatrix{
				X'\ar[d]_{p'}\ar[r]^s\ar@{}[rd]|\Box &X \ar[d]_p\ar[r]^i \ar@{}[rd]|\Box& Y\ar[d]^f\\
				W'\ar[r]^-r& W\ar[r]^-\delta  & T,
			}
		\end{gathered}
	\end{align}
where squares are cartesian. For any $\mathcal F\in D_{\rm ctf}(Y,\Lambda)$ and $\mathcal G\in D_{\rm ctf}(T,\Lambda)$, we have a commutative diagram
	\begin{align}\label{eq:TransAforC2}
		\begin{gathered}
			\xymatrix@C=2cm{
				(is)^*\mathcal{F}\otimes^L{p'}^*(\delta r)^!\mathcal{G}\ar[rr]^{c_{\delta r,f,\mathcal{F},\mathcal{G}}}&&(is)^!(\mathcal{F}\otimes^Lf^*\mathcal{G})\\
				s^*i^*\mathcal{F}\otimes^L{p'}^*r^!\delta^!\mathcal{G} \ar[u]_{\simeq}\ar[r]^{c_{r,p,i^\ast\mathcal{F},\delta^!\mathcal{G}}}&s^!(i^*\mathcal{F}\otimes^Lp^*\delta^!\mathcal{G})\ar[r]^{s^!(c_{\delta,f,\mathcal{F},\mathcal{G}})} &s^!i^!(\mathcal{F}\otimes^Lf^*\mathcal{G}),\ar[u]_-{\simeq}
			}
		\end{gathered}
	\end{align}
where the vertical arrows are canonical isomorphisms by \cite[(3.1.13.1)]{Del73}.
\end{lemma}
\begin{proof}
By the definition of $c_{\delta,f,\mathcal F,\mathcal G}$, 
it suffices to show the following diagram is commutative:
\begin{align}\label{eq:TransAforC3}
	\begin{gathered}
		\xymatrix@C=2.2cm{
			(is)^*\mathcal{F}\otimes^L{p'}^*(\delta r)^!\mathcal{G} \ar[r]^-{{id}\otimes \eqref{eq:pullbackforanyGdefTrans}} &(is)^*\mathcal{F}\otimes^L(is)^!f^*\mathcal{G} \ar[r]^{c_{is ,id,\mathcal{F},f^*\mathcal{G}}}\ar@{}|{(1)}[ddr]&(is)^!(\mathcal{F}\otimes^Lf^*\mathcal{G})\\
		s^*i^*\mathcal{F}\otimes^L{p'}^*(\delta r)^!\mathcal{G}\ar[u]_{\simeq} \ar[r]^-{{id}\otimes \eqref{eq:pullbackforanyGdefTrans}}	\ar@{}|{(3)}[rdd]&s^*i^*\mathcal{F}\otimes^L (is)^!f^*\mathcal{G}\ar[u]_{\simeq}&s^!i^!(\mathcal{F}\otimes^Lf^*\mathcal{G}) \ar[u]_{\simeq}\\
			&s^*i^*\mathcal{F}\otimes^Ls^!i^!f^*\mathcal{G} \ar[u]_{\simeq}\ar[r]^{c_{s,id,i^*\mathcal{F},i^!f^*\mathcal{G}}}\ar@{}|{(2)}[dr]&s^!(i^*\mathcal{F}\otimes^Li^!f^*\mathcal{G})\ar[u]_-{s^!(c_{i,id,\mathcal{F},f^*\mathcal{G}})}\\
			s^*i^*\mathcal{F}\otimes^L{p'}^*r^!\delta^!\mathcal{G} \ar[r]^-{{id}\otimes \eqref{eq:pullbackforanyGdefTrans}}\ar[uu]_{\simeq}&s^*i^*\mathcal{F}\otimes^Ls^!p^*\delta^!\mathcal{G} \ar[u]^-{{id}\otimes s^!\eqref{eq:pullbackforanyGdefTrans}} \ar[r]^{c_{s,id,i^*\mathcal{F},p^*\delta^!\mathcal{G}}}&s^!(i^*\mathcal{F}\otimes^Lp^*\delta^!\mathcal{G}).\ar[u]_-{s^!({id}\otimes \eqref{eq:pullbackforanyGdefTrans})}
		}
	\end{gathered}
\end{align}
By \cite[Lemma 1.1.2.(2)]{Sai22} (and its proof), the diagram (1)  is commutative. 
Applying the functor $c_{s,id,i^\ast\mathcal F,-}$ to the base change morphism $p^\ast \delta^!\mathcal G\to i^!f^\ast \mathcal G$, we get the commutativity of the diagram (2).
The diagram (3)  is commutative by the functoriality of the base change morphism \eqref{eq:pullbackforanyGdefTrans}.
\end{proof}
\begin{lemma}\label{lem:TransBforC}
	Consider the following commutative diagram in ${\rm Sch}_S$:
	\begin{align}\label{eq:TransBforC1}
		\begin{gathered}
			\xymatrix{
				X'\ar[d]_-h\ar[r]^-{i'}\ar@{}|\Box[rd]&Y'\ar[d]^-g\\
				X\ar[r]^-i\ar[d]_-p\ar@{}|\Box[rd]&Y\ar[d]^-f\\
				W\ar[r]^-\delta&T,
			}
		\end{gathered}
	\end{align}
where squares are cartesian. 
\begin{itemize}
	\item[(1)] 
For any $\mathcal F\in D_{\rm ctf}(Y',\Lambda)$ and any $\mathcal G\in D_{\rm ctf}(T,\Lambda)$, we have a commutative diagram
\begin{align}\label{eq:TransBforC1-2}
	\begin{gathered}
		\xymatrix@C=2.2cm{
			i^{\prime\ast}\mathcal F\otimes^L h^\ast p^\ast \delta^!\mathcal G\ar[r]^-{c_{\delta,fg,\mathcal F,\mathcal G}}\ar[d]_-{id\otimes h^\ast \eqref{eq:pullbackforanyGdefTrans}}&i^{\prime!}(\mathcal F\otimes^L g^\ast f^\ast \mathcal G)\ar@{=}[d]\\
			i^{\prime\ast}\mathcal F\otimes^L h^\ast i^! f^\ast\mathcal G\ar[r]^-{c_{i,g,\mathcal F,f^\ast\mathcal G}}&i^{\prime!}(\mathcal F\otimes^L g^\ast f^\ast \mathcal G).
		}
	\end{gathered}
\end{align}

\item[(2)]For any $\mathcal H\in D_{\rm ctf}(Y,\Lambda)$ and any $\mathcal G\in D_{\rm ctf}(T,\Lambda)$, we have a commutative diagram
\begin{align}\label{eq:pullbackCF}
	\begin{gathered}
		\xymatrix@C=2cm{
			h^\ast(i^\ast\mathcal H\otimes^L p^\ast\delta^!\mathcal G)\ar[d]_-{\simeq}\ar[r]^-{h^\ast c_{\delta,f,\mathcal H,\mathcal G}}&h^\ast i^!(\mathcal H\otimes^L f^\ast\mathcal G)\ar[d]^-{\eqref{eq:pullbackforanyGdefTrans}}\\
			i^{\prime\ast}g^\ast\mathcal H\otimes^L h^\ast p^\ast\delta^!\mathcal G\ar[r]^-{c_{\delta,fg,g^\ast\mathcal H,\mathcal G}}&i^{\prime !}g^\ast(\mathcal H\otimes^L f^\ast\mathcal G).
		}
	\end{gathered}
\end{align}

\item[(3)]
For any $\mathcal F\in D_{\rm ctf}(Y',\Lambda)$ and any $\mathcal G\in D_{\rm ctf}(T,\Lambda)$, we have a commutative diagram
	\begin{align}\label{eq:TransBforC2}
	\begin{gathered}
		\xymatrix@C=2.2cm{
	i^\ast g_\ast\mathcal F\otimes^L p^\ast\delta^!\mathcal G\ar[d]\ar[r]^-{c_{\delta,f,g_\ast\mathcal F,\mathcal G}}&i^!(g_\ast\mathcal F\otimes^L f^\ast\mathcal G)\ar[d]\\
	h_\ast(i^{\prime\ast}\mathcal F\otimes^L h^\ast p^\ast \delta^!\mathcal G)\ar[r]^-{h_\ast(c_{\delta,fg,\mathcal F,\mathcal G})}&h_\ast i^{\prime!}(\mathcal F\otimes^L g^\ast f^\ast\mathcal G),	
	}
	\end{gathered}
	\end{align}
where the left vertical map is the composition of $i^\ast g_\ast\mathcal F\otimes^L p^\ast \delta^!\Lambda\xrightarrow{(b.c)\otimes id}h_\ast i^{\prime \ast}\mathcal F\otimes^L p^\ast\delta^!\Lambda$ with the morphism $($cf. \cite[(1-1)]{Sai22}$)$
	\begin{align}\label{eq:TransBforC3}
		h_\ast i^{\prime \ast}\mathcal F\otimes^L p^\ast\delta^!\Lambda \rightarrow h_\ast(i^{\prime\ast}\mathcal F\otimes^L h^\ast p^\ast \delta^!\Lambda)
	\end{align}
induced by the adjunction $h^\ast h_\ast i^{\prime \ast}\mathcal F\to i^{\prime \ast}\mathcal F$, and 
 the right vertical map of \eqref{eq:TransBforC2} is the composition
\begin{align}\label{eq:TransBforC3-R}
	i^!(g_\ast\mathcal F\otimes^L f^\ast\mathcal G) \xrightarrow{\eqref{eq:TransBforC3}} i^!g_\ast (\mathcal F\otimes^L g^\ast f^\ast\mathcal G)\xrightarrow[\simeq]{\rm b.c}h_\ast i^{\prime!}(\mathcal F\otimes^L g^\ast f^\ast\mathcal G).
\end{align}

\end{itemize}
\end{lemma}
\begin{proof} 
\begin{itemize}
\item[(1)]This follows from the following commutative diagram
\begin{align}\label{eq:TransBforC1-2-pf}
	\begin{gathered}
		\xymatrix@C=2.2cm{
			i^{\prime\ast}\mathcal F\otimes^L h^\ast p^\ast \delta^!\mathcal G\ar[r]^-{id\otimes  \eqref{eq:pullbackforanyGdefTrans}}\ar[d]_-{id\otimes h^\ast \eqref{eq:pullbackforanyGdefTrans}}&i^{\prime\ast}\mathcal F\otimes^L i^{\prime !}g^\ast f^\ast\mathcal G\ar[r]^-{c_{i',id,\mathcal F,g^\ast f^\ast\mathcal G}}\ar@{=}[d]&i^{\prime!}(\mathcal F\otimes^L g^\ast f^\ast \mathcal G)\ar@{=}[d]\\
			i^{\prime\ast}\mathcal F\otimes^L h^\ast i^! f^\ast\mathcal G\ar[r]^-{id\otimes  \eqref{eq:pullbackforanyGdefTrans}}&i^{\prime\ast}\mathcal F\otimes^L i^{\prime !}g^\ast f^\ast\mathcal G\ar[r]^-{c_{i',id,\mathcal F,g^\ast f^\ast\mathcal G}}&i^{\prime!}(\mathcal F\otimes^L g^\ast f^\ast \mathcal G).
		}
	\end{gathered}
\end{align}

\item[(2)] This is a consequence of the following commutative diagram
\begin{align}\label{eq:pullbackCF-pf1}
	\begin{gathered}
		\xymatrix@C=2cm{
			h^\ast i^\ast\mathcal H\otimes^L h^\ast p^\ast\delta^!\mathcal G\ar[d]_-{\simeq}\ar[r]^-{id\otimes  \eqref{eq:pullbackforanyGdefTrans}}&h^\ast i^\ast\mathcal H\otimes^L h^\ast i^{!} f^\ast \mathcal G\ar[r]^-{h^\ast c_{i,id,\mathcal H,f^\ast\mathcal G}}\ar[rd]^-{c_{i,g,g^\ast\mathcal H,f^\ast\mathcal G}}\ar[d]^-{id\otimes  \eqref{eq:pullbackforanyGdefTrans}}&h^\ast i^!(\mathcal H\otimes^L f^\ast\mathcal G)\ar[d]^-{\eqref{eq:pullbackforanyGdefTrans}}\\
			i^{\prime\ast}g^\ast\mathcal H\otimes^L h^\ast p^\ast\delta^!\mathcal G\ar[r]^-{id\otimes  \eqref{eq:pullbackforanyGdefTrans}}&i^{\prime\ast}g^\ast\mathcal H\otimes^L i^{\prime !}g^\ast f^\ast\mathcal G\ar[r]^-{c_{i^\prime,id,g^\ast\mathcal H,g^\ast f^\ast \mathcal G}}&i^{\prime !}g^\ast(\mathcal H\otimes^L f^\ast\mathcal G),
		}
	\end{gathered}
\end{align}
where the left diagram is commutative by the functoriality of the base change morphism \eqref{eq:pullbackforanyGdefTrans}, the right diagram is commutative by \eqref{eq:TransBforC1-2}.

\item[(3)]
By taking adjoint, we obtain the assertion from the following commutative diagram
\begin{align}\label{eq:TransBforC2-pf1}
\begin{gathered}
	\xymatrix@C=2.2cm{
	h^\ast (i^\ast g_\ast\mathcal F\otimes^L p^\ast\delta^!\mathcal G)\ar[d]\ar[r]^-{h^\ast c_{\delta,f,g_\ast\mathcal F,\mathcal G}}&h^\ast i^!(g_\ast\mathcal F\otimes^L f^\ast\mathcal G)\ar[d]\\
	i^{\prime\ast}g^\ast g_\ast \mathcal F\otimes^L h^\ast p^\ast \delta^!\mathcal G\ar[r]^-{c_{\delta,fg,g^\ast g_\ast \mathcal F,\mathcal G}}\ar[d]^-{g^\ast g_\ast\mathcal F\to \mathcal F}& i^{\prime!}g^\ast (g_\ast \mathcal F\otimes^L  f^\ast\mathcal G)\ar[d]\\
	i^{\prime\ast}\mathcal F\otimes^L h^\ast p^\ast \delta^!\mathcal G\ar[r]^-{c_{\delta,fg,\mathcal F,\mathcal G}}& i^{\prime!}(\mathcal F\otimes^L g^\ast f^\ast\mathcal G),	
}
\end{gathered}
\end{align}
where the upper diagram is commutative by applying \eqref{eq:pullbackCF} to $\mathcal H=g_\ast\mathcal F$, the lower diagram is commutative by applying $c_{\delta,fg,-,\mathcal G}$ to 
$g^\ast g_\ast\mathcal F\to \mathcal F$.
\end{itemize}
\end{proof}

The following proposition generalizes \cite[Lemma 8.6.1 and Corollary 8.10]{Sai17a}.
\begin{proposition}\label{prop:closedimmtran}
Consider the cartesian diagram \eqref{eq:carGtrans}. Let $\mathcal F\in D_{\rm ctf}(Y,\Lambda)$ and $\mathcal G\in D_{\rm ctf}(T,\Lambda)$.
\begin{itemize}
	\item[(1)] Assume that $\delta$ is a smooth morphism of relative dimension $d$. 
	Consider the following canonical isomorphism
\begin{align}\label{eq:traceAndPoincareDual}
	 t_{i,\mathcal F}:i^\ast\mathcal F\otimes ^L\Lambda(d)[2d]\to i^!\mathcal F
\end{align}
 	 defined by the Poincar\'e duality \cite[Th\'eor\`eme 3.2.5]{Del73} $($cf. \cite[(8.3)]{Sai17a}$)$. Then we have a commutative diagram
	\begin{align}\label{eq:cAndPoincareDual}
		\begin{gathered}
			\xymatrix@C=2cm{
				i^\ast\mathcal F\otimes ^L i^\ast f^\ast\mathcal G\otimes^L\Lambda(d)[2d]\ar[rr]_-{\simeq}^-{t_{i,\mathcal F\otimes^Lf^\ast\mathcal G}}\ar[d]_-{\simeq}&&i^!(\mathcal F\otimes^L f^\ast\mathcal G)\ar@{=}[d]\\
				i^\ast\mathcal F\otimes ^L p^\ast \delta^\ast\mathcal G\otimes^L\Lambda(d)[2d]\ar[r]_-{\simeq}^-{id\otimes p^\ast t_{\delta,\mathcal G}}&i^\ast\mathcal F\otimes^L p^\ast\delta^!\mathcal G\ar[r]^-{c_{\delta,f,\mathcal F,\mathcal G}}&i^!(\mathcal F\otimes^L f^\ast\mathcal G).
			}
		\end{gathered}
	\end{align}
	It implies that $c_{\delta,f,\mathcal F,\mathcal G}$ is an isomorphism. In particular, $\delta$ is $\mathcal F$-transversal.
	\item[(2)] Assume that $f:Y\to T$ is universally locally acyclic relatively to $\mathcal F$. Then $c_{\delta,f,\mathcal F,\mathcal G}$ is an isomorphism. In particular, $\delta$ is $\mathcal F$-transversal.
\end{itemize}
\end{proposition}
\begin{proof}
\begin{itemize}
	\item[(1)] 
	By the proof of \cite[Lemma 8.6.1]{Sai17a}, we have a commutative diagram
		\begin{align}\label{eq:cAndPoincareDual-1}
		\begin{gathered}
			\xymatrix@C=2.2cm{
				i^\ast\mathcal F\otimes ^L\Lambda(d)[2d]\ar[r]_-{\simeq}^-{t_{i,\mathcal F}}\ar[d]^-{\simeq}_-{id\otimes  t_{i,\Lambda}}&i^!\mathcal F\ar@{=}[d]\\
				i^\ast\mathcal F\otimes^L i^!\Lambda\ar[r]^-{c_{i,id,\mathcal F}}&i^!\mathcal F.
			}
		\end{gathered}
	\end{align}
	By \eqref{eq:cAndPoincareDual-1} and the definition of $c_{\delta,f,\mathcal F,\mathcal G}$ (cf. \eqref{eq:carGtrans1}), the assertion  follows from  the  following commutative diagram 
		\begin{align}\label{eq:cAndPoincareDual-2}
		\begin{gathered}
			\xymatrix@C=1.5cm{
				i^\ast\mathcal F\otimes ^L i^\ast f^\ast\mathcal G\otimes^L\Lambda(d)[2d]\ar@{=}[r]\ar[d]^-{\simeq}_-{id\otimes p^\ast t_{\delta,\Lambda}}\ar@{}|{(a)}[dr]&i^\ast\mathcal F\otimes ^L i^\ast f^\ast\mathcal G\otimes^L\Lambda(d)[2d]\ar[r]_-{\simeq}^-{t_{i,\mathcal F\otimes^Lf^\ast\mathcal G}}\ar[d]^-{\simeq}_-{id\otimes  t_{i,\Lambda}}\ar@{}|{(b)}[dr]&i^!(\mathcal F\otimes^L f^\ast\mathcal G)\ar@{=}[d]\\
				i^\ast\mathcal F\otimes^L p^\ast(\delta^\ast\mathcal G\otimes^L\delta^!\Lambda)\ar[d]_-{id\otimes p^\ast c_{\delta,id,\mathcal G}}\ar[r]^-{id\otimes \eqref{eq:pullbackforanyGdefTrans}}\ar@{}|{(c)}[dr]&i^\ast\mathcal F\otimes^L i^\ast f^\ast\mathcal G\otimes^L i^!\Lambda\ar[d]_-{id\otimes c_{i,id,f^\ast\mathcal G}}\ar[r]^-{c_{i,id,\mathcal F\otimes^L f^\ast\mathcal G}}\ar@{}|{(d)}[dr]&i^!(\mathcal F\otimes^L f^\ast\mathcal G)\ar@{=}[d]\\
				i^\ast\mathcal F\otimes^L p^\ast\delta^!\mathcal G\ar[r]^-{id\otimes\eqref{eq:pullbackforanyGdefTrans}}&i^\ast\mathcal F\otimes^L i^! f^\ast\mathcal G\ar[r]^-{c_{i,id,\mathcal F,f^\ast\mathcal G}}&i^!(\mathcal F\otimes^L f^\ast\mathcal G),
			}
		\end{gathered}
	\end{align}
where $(a)$ is the base change for the trace map, $(b)$ is commutative by \eqref{eq:cAndPoincareDual-1}, $(c)$ is commutative by Lemma \ref{lem:TransBforC}.(3) and the fact that $c_{i,id,f^\ast\mathcal G}\circ\eqref{eq:pullbackforanyGdefTrans}=c_{\delta,f,f^\ast\mathcal G}$,
$(d)$ is commutative by \cite[Lemma 1.1.2.(1)]{Sai22}.

\item[(2)]
By (1) and Lemma \ref{lem:TransAforC},
we may assume that $\delta$ is a closed immersion. We put $V=T\setminus W$. Consider the following cartesian diagram
\begin{align}\label{pfeq:carGtrans}
\begin{gathered}
\xymatrix{
X\ar@{}[rd]|\Box\ar[d]_p\ar[r]^-i&Y\ar[d]^-f\ar@{}[rd]|\Box&U\ar[l]_-{j'}\ar[d]^-{f_V}\\
W\ar[r]^-{\delta}&T&V\ar[l]_-j.
}
\end{gathered}
\end{align}
Since $f$ is universally locally acyclic relatively to $\mathcal F$, $f$ is also strongly  locally acyclic  relatively to $\mathcal F$.
By \cite[Proposition 8.9]{Sai17a}, we obtain an isomorphism $\mathcal F\otimes^L f^\ast j_\ast j^\ast\mathcal G\xrightarrow{\simeq} j'_\ast j'^\ast(\mathcal F\otimes^L f^\ast \mathcal G)$. 
Note that $i_\ast(i^\ast \mathcal F\otimes^L p^\ast\delta^!\mathcal G)\simeq \mathcal F\otimes^L i_\ast p^\ast\delta^!\mathcal G\simeq \mathcal F\otimes^L f^\ast \delta_\ast\delta^!\mathcal G$.
Now we consider the following diagram between distinguished triangles:
\begin{align}\label{pfeq:carGtrans2}
\begin{gathered}
\xymatrix{
i_\ast i^!(\mathcal F\otimes^L f^\ast \mathcal G)\ar[r]&\mathcal F\otimes^L f^\ast \mathcal G\ar[r]&j'_\ast j'^\ast(\mathcal F\otimes^L f^\ast \mathcal G)\ar[r]^-{+1}&\\
i_\ast(i^\ast \mathcal F\otimes^L p^\ast\delta^!\mathcal G)\ar[u]^-{i_\ast(c_{\delta,f,\mathcal F,\mathcal G})}&&\\
\mathcal F\otimes^L f^\ast \delta_\ast\delta^!\mathcal G\ar[r]\ar[u]^-{\simeq}&\mathcal F\otimes^L f^\ast \mathcal G\ar[r]\ar@{=}[uu]&\mathcal F\otimes^L f^\ast j_\ast j^\ast\mathcal G\ar[r]^-{+1}\ar[uu]_-{\simeq}&.
}
\end{gathered}
\end{align}
The commutativity of the above diagram can be verified by using the proof of \cite[Lemma 1.1.3]{Sai22}.
Thus $i_\ast(c_{\delta,f,\mathcal F,\mathcal G})$ is an isomorphism.
Since $i$ is a closed immersion, the morphism $c_{\delta,f,\mathcal F,\mathcal G}$ is also an isomorphism.
\end{itemize}
\end{proof}
The following proposition gives  a converse   of Proposition \ref{prop:closedimmtran}, 
which will not be used in this paper.
\begin{proposition}[{Saito, \cite[Proposition 8.11]{Sai17a}}]\label{prop:TranstoULA}
Let $f: Y\to T$ be a morphism of schemes of finite type over a field $k$ 
of characteristic $p$ and let $\mathcal F\in D_{\rm ctf}(Y,\Lambda)$. The morphism $f$ is locally acyclic relatively to $\mathcal F$ if the following condition is satisfied:

Let $T'$ and $W$ be smooth schemes over a finite extension of $k$, and
\begin{align}\label{eq:prop:TranstoULA}
\begin{gathered}
\xymatrix{
X\ar@{}[rd]|\Box\ar[d]_p\ar[r]^-i&Y'\ar[d]^-{f'}\ar[r]^-{b'}&Y\ar[d]^-f\\
W\ar[r]^-{\delta}&T'\ar[r]^-b&T
}
\end{gathered}
\end{align}
be a cartesian diagram of schemes where $b:T'\to T$ is proper and generically finite and
$\delta: W\to T'$ is a closed immersion. Then $\delta$ is $b'^\ast\mathcal F$-transversal.
\end{proposition}

\subsection{}For any separated morphism $g:Y\to S$ of finite type, we
consider the following diagram
\begin{align}
\begin{gathered}
\xymatrix{
Y\ar@{=}[rd]\ar@{^(->}[r]^-\delta&Y\times_SY\ar[r]^-{\rm pr_1}\ar[d]^-{\rm pr_2}\ar@{}|\Box[rd]&Y\ar[d]^-g\\
&Y\ar[r]^-g&S.
}
\end{gathered}
\end{align}
There is a canonical morphism 
\begin{align}\label{eq:KYSLambda}
\mathcal K_{Y/S}\otimes^L\delta^!\Lambda\to \Lambda
\end{align}
defined by the following composition:
\begin{align}
\mathcal K_{Y/S}\otimes^L\delta^!\Lambda=\delta^\ast {\rm pr}_1^\ast g^!\Lambda\otimes\delta^!\Lambda\xrightarrow{c_{\delta,{\rm id},{\rm pr}_1^\ast g^!\Lambda}}\delta^!({\rm pr}_1^\ast g^!\Lambda)\xrightarrow{\delta^!(\eqref{eq:pullbackforanyGdefTrans})} \delta^!({\rm pr}_2^!g^\ast\Lambda)=\Lambda.
\end{align}
If $g$ is a smooth morphism, then $\mathcal K_{Y/S}\otimes^L\delta^!\Lambda\to \Lambda$ is an isomorphism by the ${\rm pr}_1^\ast \mathcal K_{Y/S}$-transversality of $\delta$ (cf. \cite[Lemma 8.6.2]{Sai17a}).

\begin{proposition}\label{prop:Gysim}
Consider a commutative diagram in ${\rm Sch}_S$:
\begin{align}\label{xyd:prop:Gysim:fgh}
\begin{gathered}
\xymatrix{
X\ar[rr]^-f\ar[rd]_-h&&Y.\ar[ld]^-g\\
&S&
}
\end{gathered}
\end{align}
Let $i: X\times_Y X\to X\times_S X$ be the base change of the diagonal morphism $\delta:Y\to Y\times_SY$:
\begin{align}\label{xyd:prop:Gysim:diagonaldeltai}
\begin{gathered}
\xymatrix{
X\ar@/_2pc/[dd]_-{f}\ar@{_(->}[d]_-{\delta_1}\ar@{}[rd]|\Box\ar@{=}[r]^-{\rm id}&X\ar@{_(->}[d]^-{\delta_0}\\
X\times_Y X\ar[r]^-{i}\ar[d]_-p \ar[r] & X\times_S X  \ar[d]^-{f\times f} \\
Y\ar[r]^-\delta&Y\times_SY,\ar@{}|\Box[ul]
}
\end{gathered}
\end{align}
where $\delta_0$ and $\delta_1$ are the diagonal morphisms.
 Assume that 
$\mathcal K_{Y/S}=g^!\Lambda$ is locally constant and that $\mathcal K_{Y/S}\otimes^L \delta^!\Lambda\simeq \Lambda$ $($cf. \eqref{eq:KYSLambda}$)$. Then we have 
\begin{itemize}
\item[(1)] 
The morphism $f$ is $\mathcal K_{Y/S}$-transversal, i.e., $\mathcal K_{X/Y}\otimes^L f^\ast \mathcal K_{Y/S}=f^!\Lambda\otimes^L f^\ast\mathcal K_{Y/S}\xrightarrow{\simeq}f^! \mathcal K_{Y/S} = \mathcal K_{X/S}$ is an isomorphism.
\item[(2)] 
The composition 
\begin{align}\label{eq:KXYSfdelta}
\mathcal K_{X/S}\otimes^L f^\ast \delta^!\Lambda =f^!\mathcal K_{Y/S} \otimes^L f^\ast \delta^!\Lambda \xrightarrow[\simeq]{c_{f,{\rm id}, \delta^!\Lambda,\mathcal K_{Y/S}}} f^!(\mathcal K_{Y/S}\otimes^L \delta^!\Lambda)\xrightarrow[\simeq]{\eqref{eq:KYSLambda}} f^!\Lambda=\mathcal K_{X/Y}
\end{align}
is an isomorphism.

\item[(3)]If  $\mathcal F\in D_{\rm ctf}(X/Y,\Lambda)$ and  $\mathcal G\in D_{\rm ctf}(X/Y,\Lambda)$, then $\delta$ is 
$\mathcal F\boxtimes_S^LD_{X/S}(\mathcal G)$-transversal, i.e.,
\begin{align}\label{lem:eq:Gysim2}
\mathcal F\boxtimes_Y^LD_{X/Y}(\mathcal G)\overset{(a)}{\simeq}
i^\ast(\mathcal F\boxtimes_S^LD_{X/S}(\mathcal G))\otimes^L p^\ast\delta^!\Lambda
\to i^!(\mathcal F\boxtimes_S^LD_{X/S}(\mathcal G))
\end{align}
is an isomorphism, where $(a)$ is induced by the following composition:
\begin{align}\label{eq:GysimforRHOM}
\begin{split}
D_{X/Y}(\mathcal G)=R\mathcal Hom(\mathcal G, \mathcal K_{X/Y})&\xleftarrow[\simeq]{\eqref{eq:KXYSfdelta}}
R\mathcal Hom(\mathcal G, \mathcal K_{X/S}\otimes f^\ast\delta^!\Lambda)\\
&\simeq R\mathcal Hom(\mathcal G, \mathcal K_{X/S})\otimes f^\ast\delta^!\Lambda
=D_{X/S}(\mathcal G)\otimes^L f^\ast \delta^!\Lambda.
 \end{split}
\end{align}
\end{itemize}
\end{proposition}
\begin{proof}Since  $\mathcal K_{Y/S}$ is locally constant, the assertion (1) follows from the proof of \cite[Lemma 8.6.2]{Sai17a}.
The second claim follows from the fact that $\delta^!\Lambda$ is locally constant.
Now we prove (3). Since $\mathcal G\in D_{\rm ctf}(X/Y,\Lambda)$, we have $D_{X/Y}(\mathcal G)\in D_{\rm ctf}(X/Y,\Lambda)$ by \cite[Corollary 2.18]{LZ22}. Since 
$D_{X/S}(\mathcal G)\simeq D_{X/Y}(\mathcal G)\otimes^L f^\ast\mathcal K_{Y/S}$, the morphism $f:X\to Y$ is universally locally acyclic relatively to $D_{X/S}(\mathcal G)$. By \cite[Corollary 2.5]{Ill17},
the morphism $f\times f: X\times_SX\to Y\times_SY$ is universally locally acyclic relatively to $\mathcal F\boxtimes_S^LD_{X/S}(\mathcal G)$. By Proposition \ref{prop:closedimmtran}.(2), $\delta$ is $\mathcal F\boxtimes_S^LD_{X/S}(\mathcal G)$-transversal.
\end{proof}

\subsection{}Now we compare the morphism induced by $c_{\delta,f, \mathcal F}$ on cohomology groups with the Gysin map. Before that, let us
briefly recall the cycle class defined by a regular immersion.
Let $\eta: P\hookrightarrow Q$ be a  regular immersion of codimension $c$ in ${\rm Sch}_S$. By \cite[Definition 1.1.2]{Fuj02} and \cite[Expos\'e XVI, D\'efinition 2.3.1]{ILO14}, there is a cycle class ${\rm Cl}_\eta\in H_P^{2c}(Q,\Lambda(c))$ which refines the  
$c$-th Chern class $c_c(\mathcal N_{P/Q}^\vee)\in H^{2c}(P,\Lambda(c))$, where $\mathcal N_{P/Q}$ is the conormal sheaf associated to the closed immersion $\eta$.  We view the cycle class ${\rm Cl}_\eta$ as a morphism ${\rm Cl}_\eta:\Lambda\to \eta^!\Lambda(c)[2c]$. 

\subsection{}\label{sub:chernandCdelta}
	Consider a cartesian diagram in ${\rm Sch}_S$
		\begin{align}\label{pfeq:propchangebasehomo:S1}
		\begin{gathered}
			\xymatrix{
				X\ar[d]_-{p}\ar@{}[rd]|\Box\ar@{=}[r]&X\ar[d]^-{f}\\
				W\ar[r]^-\delta&T,
			}
		\end{gathered}
	\end{align}
where $\delta:W\to T$ is a regular immersion of codimension $r$.  
For any $\mathcal F\in D_{\rm ctf}(X,\Lambda)$, the map $c_{\delta,f,\mathcal F}: \mathcal F\otimes^L p^\ast\delta^!\Lambda\to \mathcal F$ induces a morphism
\begin{align}\label{eq:cherncupS1}
	\delta^!: H^0(X,\mathcal F\otimes^L p^\ast\delta^!\Lambda)\to H^0(X,\mathcal F),
\end{align}
which is also induced by $\delta^!\xrightarrow{\eqref{delta!ast}} \delta^\ast$.
Now assume that ${\rm Cl}_\delta:\Lambda\rightarrow\delta^!\Lambda(r)[2r]$ is an isomorphism. 
We view the top Chern class $c_{r}(p^*\mathcal N_{W/T}^{\vee})$ as
an element of $H^{2r}(X,\Lambda(r))$ (cf. \cite[Expos\'e XVI, Th\'eor\`eme 1.3]{ILO14}).
The cup product with this element induces a map
\begin{align}\label{eq:cherncupS2}
	 H^0(X,\mathcal F\otimes^L p^\ast\delta^!\Lambda)\xrightarrow{ c_{r}(p^*\mathcal N_{W/T}^{\vee})} H^0(X,\mathcal F\otimes^L p^\ast\delta^!\Lambda(r)[2r])\xrightarrow[\simeq]{id\otimes p^\ast {\rm Cl}_{\delta}^{-1}} H^0(X,\mathcal F),
\end{align}
which is again denoted by $c_{r}(p^*\mathcal N_{W/T}^{\vee})$.
\begin{lemma}\label{lem:sameChernMapS}
	Consider the cartesian diagram  \eqref{pfeq:propchangebasehomo:S1}. 
\begin{itemize}
\item[(1)]We have a commutative diagram
	\begin{align}\label{eq:ChernEqualCdelta}
		\begin{gathered}
			\xymatrix@C=2cm{
			\Lambda\ar[d]_-{ p^\ast{\rm Cl}_\delta}\ar[r]^-{ c_r(p^\ast\mathcal N^\vee_{W/T})}&\Lambda(r)[2r]\ar@{=}[d]\\
			p^\ast\delta^!\Lambda(r)[2r]\ar[r]^-{ \eqref{eq:pullbackforanyGdefTrans}}& f^\ast\Lambda(r)[2r].
		}
		\end{gathered}
	\end{align}
\item[(2)]	Assume that ${\rm Cl}_\delta:\Lambda\rightarrow\delta^!\Lambda(r)[2r]$ is an isomorphism. Then the two maps \eqref{eq:cherncupS1} and \eqref{eq:cherncupS2} are the same, i.e.,
\begin{align}\label{eq:cherncupS}
	\delta^!=c_{r}(p^*\mathcal N_{W/T}^{\vee}): H^0(X,\mathcal F\otimes^L p^\ast\delta^!\Lambda)\to H^0(X,\mathcal F).
\end{align}
\end{itemize}
\end{lemma}
\begin{proof}
(1)
	Note that the excess normal sheaf of the cartesian diagram \eqref{pfeq:propchangebasehomo:S1}
	equals to $p^*\mathcal N_{W/T}^{\vee}$. 
	By  \cite[Expos\'e XVI, Proposition 2.3.2]{ILO14}, the image of ${\rm Cl}_{\delta}: \Lambda\to \delta^!\Lambda(r)[2r]$ under the map
	\begin{align}\label{eq:fuppS1}
		H^{2r}_{W}(T, \Lambda(r))\xrightarrow{p^\ast}H^{2r}_X(X,\Lambda(r))=H^{2r}(X,\Lambda(r))
	\end{align}
	equals to $c_{r}(p^*\mathcal N_{W/T}^{\vee})\in H^{2r}(X,\Lambda(r))$. By definition of \eqref{eq:fuppS1},
	the class $c_{r}(p^*\mathcal N_{W/T}^{\vee})$ equals to the composition 
	\begin{align}\label{eq:fuppS2}
		\Lambda=p^\ast\Lambda\xrightarrow[\simeq]{p^\ast {\rm Cl}_{\delta}}p^\ast\delta^!\Lambda(r)[2r]\xrightarrow{\eqref{eq:pullbackforanyGdefTrans}}f^\ast\Lambda(r)[2r]=\Lambda(r)[2r].
	\end{align}
This proves the commutativity of \eqref{eq:ChernEqualCdelta}.

(2) For any  $\alpha\in H^0(X,\mathcal F\otimes^L p^\ast\delta^!\Lambda)$, we write it as $\alpha:\Lambda\to \mathcal F\otimes^L p^\ast\delta^!\Lambda$. The image of $\alpha$ under the map \eqref{eq:cherncupS2} equals to
\begin{align}\label{eq:fuppS3}
	\Lambda\xrightarrow{\alpha}\mathcal F\otimes
	p^\ast \delta^!\Lambda\xrightarrow{id\otimes c_{r}(p^*\mathcal N_{W/T}^{\vee})}\mathcal F\otimes
	p^\ast \delta^!\Lambda(r)[2r]\xrightarrow{id\otimes p^\ast {\rm Cl}_\delta^{-1}} \mathcal F.
\end{align}
By definition, the class $\delta^!(\alpha)$ is the composition $\Lambda\xrightarrow{\alpha}\mathcal F\otimes^L p^\ast \delta^!\Lambda\xrightarrow{id\otimes\eqref{eq:pullbackforanyGdefTrans}}\mathcal F$.
We need to show the following diagram
	\begin{align}\label{eq:fuppS4-NPF}
	\begin{gathered}
		\xymatrix@C=2cm{
	\mathcal F\otimes^L p^\ast \delta^!\Lambda\ar[r]^-{id\otimes c_{r}(p^*\mathcal N_{W/T}^{\vee})}\ar[rd]_-{id\otimes \eqref{eq:pullbackforanyGdefTrans}}&\mathcal F\otimes^L p^\ast \delta^!\Lambda(r)[2r]\\
	&\mathcal F\otimes^L \Lambda\ar[u]_-{id\otimes p^\ast{\rm Cl}_\delta}
}
	\end{gathered}
\end{align}
is commutative.
By \eqref{eq:ChernEqualCdelta}, the result follows from the following commutative diagram
\begin{align}\label{eq:fuppS4}
	\begin{gathered}
		\xymatrix@C=2cm{
\mathcal F\otimes^L p^\ast \delta^!\Lambda\otimes^L\Lambda\ar[r]^-{id\otimes id \otimes p^\ast{\rm Cl}_\delta}\ar[d]^-{id\otimes \eqref{eq:pullbackforanyGdefTrans}}&\mathcal F\otimes^L p^\ast\delta^!\Lambda\otimes p^\ast \delta^!\Lambda(r)[2r]\ar[d]^-{id\otimes \eqref{eq:pullbackforanyGdefTrans}\otimes id}\\
\mathcal F\otimes^L\Lambda\ar[r]^-{id\otimes p^\ast{\rm Cl}_\delta}&\mathcal F\otimes^L p^\ast\delta^!\Lambda(r)[2r].
		}
	\end{gathered}
\end{align}
\end{proof}

\subsection{}\label{subsec:defdeltaDelta}
Consider the notation in \ref{def:gtrans}.
If $\delta$ is a closed immersion, we can imitate the constructions of \cite[5.1]{AS07} and \cite[Lemma 3.5]{Tsu11} to define a functor
$\delta^\Delta: D_{\rm ctf}(Y,\Lambda)\to D_{\rm ctf}(X,\Lambda)$ such that for any $\mathcal F\in D_{\rm ctf}(Y,\Lambda)$, we have a distinguished triangle 
\begin{align}\label{eq:Dchf}
i^\ast \mathcal F\otimes^L p^\ast \delta^!\Lambda\xrightarrow{c_{\delta,f,\mathcal F}} i^!\mathcal F\to \delta^\Delta\mathcal F\xrightarrow{+1}.
\end{align} 
Indeed, let $j:T\setminus W\to T$ be the open immersion and we put 
\begin{align}\label{eq:DefineDeltaDelta}
	\delta^\Delta\mathcal F:= i^!(\mathcal F\otimes^L f^\ast j_\ast\Lambda).
\end{align}
Applying $c_{\delta,f,\mathcal F,-}$
to $\delta_!\delta^!\Lambda\xrightarrow{\rm adj} \Lambda$, we get a commutative diagram
\begin{align}\label{eq:distdeltaj2}
\begin{gathered}	
\xymatrix@C=1.5cm{
i^\ast\mathcal F\otimes^L p^\ast \delta^!\delta_!\delta^!\Lambda\ar[r]^-{c_{\delta,f,\mathcal F,\delta_!\delta^!\Lambda}}_-\simeq\ar[d]_-{\rm adj}^-\simeq&i^!(\mathcal F\otimes^L f^\ast \delta_!\delta^!\Lambda)\ar[d]^-{\rm adj}\\
i^\ast\mathcal F\otimes^L p^\ast \delta^!\Lambda\ar[r]^-{c_{\delta,f,\mathcal F}}&i^!\mathcal F,
}
\end{gathered}
\end{align}
where $c_{\delta,f,\mathcal F,\delta_!\delta^!\Lambda}$ is an isomorphism since it is the composition of isomorphisms:
\begin{align}\small
	\begin{gathered}
		\xymatrix{
			i^\ast\mathcal F\otimes^L p^\ast\delta^!\Lambda\ar[r]^-{\simeq}&i^!i_!(i^\ast\mathcal F\otimes^L p^\ast \delta^!\Lambda)\quad\ar[r]^-{\rm proj.formula}_-{\simeq}&\quad
			i^!(\mathcal F\otimes^L i_! p^\ast \delta^!\Lambda)\ar[r]^-{\rm b.c}_-{\simeq}&i^!(\mathcal F\otimes^L f^\ast \delta_!\delta^!\Lambda).
		}
	\end{gathered}
\end{align}
Applying $i^!(\mathcal F\otimes^L f^\ast -)$  to $\delta_!\delta^!\Lambda\xrightarrow{\rm adj} \Lambda\to  j_\ast\Lambda\xrightarrow{+1}$, we get
 the following distinguished triangle
\begin{align}\label{eq:distdeltaj}
	 i^!(\mathcal F\otimes^L f^\ast\delta_!\delta^!\Lambda)\xrightarrow{\rm adj} i^!\mathcal F\to  \delta^\Delta\mathcal F\xrightarrow{+1}.
\end{align} 
Then the  required distinguished triangle \eqref{eq:Dchf} follows from \eqref{eq:distdeltaj2} and \eqref{eq:distdeltaj}.

Consider the cartesian diagram \eqref{pfeq:carGtrans}. There is a canonical morphism
\begin{align}\label{eq:Dchf-twoMap}
\delta^\Delta\mathcal F=i^!(\mathcal F\otimes^L f^\ast j_\ast\Lambda)\xrightarrow{\rm b.c}
i^!(\mathcal F\otimes^L j'_\ast f_V^\ast \Lambda)=i^\Delta\mathcal F,
\end{align}
which induces a commutative diagram
\begin{align}\label{eq:Dchf-two}
\begin{gathered}
\xymatrix{
i^\ast \mathcal F\otimes^L p^\ast \delta^!\Lambda\ar[r]^-{c_{\delta,f,\mathcal F}} \ar[d]_-{\eqref{eq:pullbackforanyGdefTrans}}&i^!\mathcal F\ar[r] \ar@{=}[d]&\delta^\Delta\mathcal F\ar[d]^-{\eqref{eq:Dchf-twoMap}}\ar[r]^-{+1}&\\
i^\ast \mathcal F\otimes^L i^!\Lambda\ar[r]^-{c_{i,{\rm id},\mathcal F}} &i^!\mathcal F\ar[r] &i^\Delta\mathcal F\ar[r]^-{+1}&.
}
\end{gathered}
\end{align} 
For a general $\delta$, one could use $\infty$-enhancement of the \'etale cohomology theory and define $\delta^\Delta$ to be the  cofiber of $i^\ast(-)\otimes^L p^\ast\delta^!\Lambda \to i^!(-)$. Then $\delta^\Delta: D_{\rm ctf}(Y,\Lambda)\to D_{\rm ctf}(X,\Lambda)$ is a functor fitting in the distinguished triangle \eqref{eq:Dchf}. In this paper, we essentially only need the case where $\delta$ is a closed immersion.

\begin{lemma}\label{lem:DeltaFunctor}
Consider the cartesian diagram \eqref{eq:carGtrans}.
For any $\mathcal F\in D_{\rm ctf}(Y,\Lambda)$,  the morphism $\delta$ is $\mathcal F$-transversal if and only if the complex $ \delta^\Delta\mathcal F$ is acyclic on $X$.
\end{lemma}
\begin{proof}
	It follows from the definition of  $\mathcal F$-transversality and
    the distinguished triangle \eqref{eq:Dchf}.
\end{proof}

\section{Cohomological characteristic class}\label{sec:CCC}

\subsection{}\label{subsec:defC}
In this section, 
we summarize the construction and functorial properties of relative cohomological characteristic classes.
Consider the following cartesian diagram in ${\rm Sch}_S$
\begin{align}\label{diagXSY}
	\begin{gathered}
		\xymatrix{
	X\times_SY\ar[d]_-{\rm pr_2}\ar[r]^-{\rm pr_1}\ar@{}|\Box[rd]&X\ar[d]^-h\\
	Y\ar[r]^-g&S,	
	}
	\end{gathered}
\end{align}
where ${\rm pr}_1$ and ${\rm pr}_2$ are the projections.
For any $\mathcal F\in D_{\rm ctf}(X,\Lambda)$ and $\mathcal G\in D_{\rm ctf}(Y,\Lambda)$, 
we have canonical morphisms
\begin{align}
	\label{eq:cForFBoxK}&\mathcal F\boxtimes_S^L \mathcal K_{Y/S}={\rm pr}^\ast_1\mathcal F\otimes^L{\rm pr}_2^\ast g^!\Lambda\xrightarrow{c_{g,h,\mathcal F}}{\rm pr}_1^!\mathcal F,\\
	\label{eq:Kunethiso}&
	\mathcal F\boxtimes_S^{L}D_{Y/S}(\mathcal G)\rightarrow R\mathcal Hom_{X\times_SY}({\rm pr}^*_2\mathcal G, {\rm pr}^!_1\mathcal F), 
\end{align}
where \eqref{eq:Kunethiso} is adjoint to 
\begin{align}\label{Def-eq:Kunethiso}
	\mathcal F\boxtimes_S^{L}(D_{Y/S}(\mathcal G)\otimes^L \mathcal G)\xrightarrow{id\boxtimes{\rm ev}} \mathcal F\boxtimes_S^L \mathcal K_{Y/S}\xrightarrow{\eqref{eq:cForFBoxK}}{\rm pr}_1^!\mathcal F.
\end{align}
Note that \eqref{eq:cForFBoxK} is also a special case of \eqref{eq:Kunethiso} by taking $\mathcal G=\Lambda$. 
If moreover $\mathcal F\in D_{\rm ctf}(X/S,\Lambda)$, then
\eqref{eq:Kunethiso} is an isomorphism by \cite[Proposition 2.5]{LZ22}(see also \cite[Corollary 3.1.5]{YZ18}).

Recall that a {\it correspondence} over $S$ is a pair of morphisms $X\xleftarrow{c_1} C\xrightarrow{c_2} Y$ over $S$ (or equivalently, a morphism $c=(c_1,c_2):C\to X\times_S Y$).
For such a $c$, we have a  canonical isomorphism by \cite[Corollaire 3.1.12.2]{Del73}
\begin{align}
	\label{eq:keyisoCshrinkRHOM}
	R\mathcal Hom(c^*_2\mathcal G, c^!_1\mathcal F) \xrightarrow{\simeq}c^!R\mathcal Hom({\rm pr}^*_2\mathcal G, {\rm pr}^!_1\mathcal F),
\end{align}
which is adjoint to the composition (cf. \cite[1.5]{AS07})
\begin{align}
	\label{Def-eq:keyisoCshrinkRHOM}
	\begin{split}
		{\rm pr}^*_2\mathcal G\otimes^L c_!R\mathcal Hom(c^*_2\mathcal G, c^!_1\mathcal F) &\rightarrow c_!(c^\ast {\rm pr}_2^\ast\mathcal G\otimes^L R\mathcal Hom(c^*_2\mathcal G, c^!_1\mathcal F))\\
		&\simeq c_!(c_2^\ast\mathcal G\otimes^L R\mathcal Hom(c^*_2\mathcal G, c^!_1\mathcal F))\\
		&\xrightarrow{\rm ev}c_!c^!_1\mathcal F\simeq c_!c^!{\rm pr}_1^!\mathcal F\xrightarrow{\rm adj}  {\rm pr}^!_1\mathcal F.
	\end{split}	
\end{align}
Now we take $X=Y$ and define $P$ to be the pull-back of $C$ along the diagonal map $\delta\colon X\to X\times_SX$:
\begin{align}\label{eq:fixpoint.scheme}
\begin{gathered}
\xymatrix{
P\ar[r]^-{i}\ar[d]_-{p} \ar[r] & C  \ar[d]^c \\
X\ar[r]^-{\delta} &X\times_SX.\ar@{}|\Box[ul]
}
\end{gathered}
\end{align}
For any $\mathcal F\in D_{\rm ctf}(X/S,\Lambda)$ and $\mathcal G\in D_{\rm ctf}(X,\Lambda)$, 
we have canonical isomorphisms
\begin{align}
\label{eq:keyisoKu}
c^!(\mathcal F\boxtimes_S^{L}D_{X/S}(\mathcal G)) \xrightarrow[\simeq]{\eqref{eq:Kunethiso}} c^!R\mathcal Hom({\rm pr}^*_2\mathcal G, {\rm pr}^!_1\mathcal F) \xleftarrow[\simeq]{\eqref{eq:keyisoCshrinkRHOM}} R\mathcal Hom(c^*_2\mathcal G, c^!_1\mathcal F).
\end{align}
Combining the inverse of \eqref{eq:keyisoKu} with the canonical map
\begin{equation}\label{eq:evaMap}
c^!(\mathcal F\boxtimes_S^{L}D_{X/S}(\mathcal F))\xrightarrow{{\rm id}\to \delta_*\delta^*} c^!\delta_*(\mathcal F\otimes^LD_{X/S}(\mathcal{F})) \xrightarrow{{\rm ev}} c^!\delta_*\mathcal K_{X/S} \xrightarrow[\simeq]{\rm b.c} i_*\mathcal K_{P/S},
\end{equation}
we obtain
\begin{align}
\label{eq:Tracemap1}&R\mathcal Hom(c_2^*\mathcal F, c_1^!\mathcal F)\to i_*\mathcal K_{P/S},\\
\label{eq:Tracemap2}&{\rm Tr}: {\rm Hom}(c_2^*\mathcal F, c_1^!\mathcal F)\to {\rm Hom}(\Lambda_{C}, i_*\mathcal K_{P/S})=H^0(P, \mathcal K_{P/S}).
\end{align}
In particular, taking the special correspondence $C=X$ and $c=\delta$,  we have  canonical morphisms
\begin{align}
\label{eq:traceHomsheaf}&R\mathcal Hom(\mathcal F, \mathcal F)\to \mathcal K_{X/S},\\
\label{eq:traceHom}
&{\rm Tr}: {\rm Hom}(\mathcal F, \mathcal F)\to H^0(X, \mathcal K_{X/S}).
\end{align}
The {\it relative (cohomological) characteristic class} of $\mathcal F$ is defined to be (cf. \cite[Definition 3.6]{YZ18} and \cite[2.20]{LZ22})
\begin{align}\label{eq:RCCCCDEF}
C_{X/S}(\mathcal F)\coloneq {\rm Tr}({\rm id}_{\mathcal F})\quad{\rm in}\quad H^0(X, \mathcal K_{X/S}).
\end{align}
The map $\mathcal F\mapsto C_{X/S}(\mathcal F)$ defines a group homomorphism
\begin{align}\label{eq:CTrmap}
C_{X/S}: K_0(X/S,\Lambda)\to H^0(X, \mathcal K_{X/S}).
\end{align}
Note that if $S$ is the spectrum of a field $k$, then $C_{X/k}(\mathcal F)$ is the  characteristic class introduced by Abbes and Saito (cf. \cite[Definition 2.1.1]{AS07}).
The following lemma computes the characteristic classes of locally constant sheaves.
\begin{lemma}\label{lem:twsit}
Let $X$ be a connected  scheme in ${\rm Sch}_S$ and $\mathcal L$ a locally constant constructible flat sheaf of $\Lambda$-modules on $X$. Then for any $\mathcal F\in D_{\rm ctf}(X/S,\Lambda)$, we have
\begin{align}\label{eq:twist}
C_{X/S}(\mathcal F\otimes\mathcal L)={\rm rank}\mathcal L\cdot C_{X/S}(\mathcal F).
\end{align}
In particular, $C_{X/S}(\mathcal L)={\rm rank}\mathcal L\cdot C_{X/S}(\Lambda)$.
\end{lemma}
\begin{proof}In the following, we put $\mathcal L^\vee=R\mathcal Hom(\mathcal L,\Lambda)=D_{X/X}(\mathcal L)$. By definition,
the class $C_{X/X}(\mathcal L)={\rm rank}\mathcal L\in H^0(X,\mathcal K_{X/X})=\Lambda$ is the composition 
\begin{align}
\Lambda\xrightarrow{1_{\mathcal L}}R\mathcal Hom(\mathcal L,\mathcal L)\simeq \mathcal L\otimes\mathcal L^\vee\xrightarrow{{\rm ev}}\Lambda,
\end{align}
where ${\rm ev}$ is the evaluation map.
We have a commutative diagram
\begin{align}\small
\begin{gathered}
\xymatrix@C=3.5pc{
\Lambda\ar@{=}[dd]\ar[r]^-{1_{\mathcal F\otimes\mathcal L}}&R\mathcal Hom(\mathcal F\otimes\mathcal L,\mathcal F\otimes\mathcal L)\ar[d]^-{\simeq}
\ar[r]^-{{\eqref{eq:keyisoKu}\&\eqref{eq:evaMap}}}&(\mathcal F\otimes\mathcal L)\otimes^LD_{X/S}(\mathcal F\otimes\mathcal L)\ar[r]^-{\rm ev}\ar[d]^-{\simeq}&\mathcal K_{X/S}\ar@{=}[d]\\
&R\mathcal Hom(\mathcal F,\mathcal F)\otimes\mathcal L\otimes \mathcal L^\vee\ar[r]^-{{\eqref{eq:keyisoKu}\&\eqref{eq:evaMap}}}&\mathcal F\otimes^LD_{X/S}(\mathcal F)\otimes\mathcal L\otimes \mathcal L^\vee\ar[r]^-{{\rm ev}\otimes{\rm ev}}&\mathcal K_{X/S}\otimes\Lambda\ar@{=}[d]\\
\Lambda\ar[r]^-{1_{\mathcal F}}&R\mathcal Hom(\mathcal F,\mathcal F)\ar[u]_-{{\rm id}\otimes 1_{\mathcal L}}\ar[r]^-{{\eqref{eq:keyisoKu}\&\eqref{eq:evaMap}}}&\mathcal F\otimes^L D_{X/S}(\mathcal F)\ar[u]_-{{\rm id}\otimes 1_{\mathcal L}}\ar[r]^-{{\rm ev}\otimes C_{X/X}(\mathcal L)}&\mathcal K_{X/S}\otimes\Lambda.
}
\end{gathered}
\end{align}
The composition of the first row is $C_{X/S}(\mathcal F\otimes\mathcal L)$, while the bottom row gives ${\rm rank}\mathcal L\cdot C_{X/S}(\mathcal F)$. Thus $C_{X/S}(\mathcal F\otimes\mathcal L)={\rm rank}\mathcal L\cdot C_{X/S}(\mathcal F)$.
\end{proof}

We could also define relative characteristic classes in a more general case that only assumes local acyclicity away from small closed subsets. We first note the following lemma.
\begin{lemma}\label{lem:rcccCodBig}
Let $X\in{\rm Sch}_S$ and $Z\subseteq X$ a closed subscheme. Let $U=X\setminus Z$.
Assume that $
H^0(Z,\mathcal K_{Z/{S}})=0$ and $H^1(Z,\mathcal K_{Z/{S}})=0$.
Then we have
\begin{align}
H^0(X,\mathcal K_{X/S})\simeq H^0(U,\mathcal K_{U/S}).
\end{align}
\end{lemma}
\begin{proof}Let $\tau: Z\hookrightarrow X$ be the closed immersion and $j: U\hookrightarrow X$ the open immersion. The assertion follows from the following long exact sequence
\begin{align}
H^0(Z,\mathcal K_{Z/S})\to H^0(X,\mathcal K_{X/S})\to H^0(U,\mathcal K_{U/S})\to H^1(Z,\mathcal K_{Z/S}),
\end{align}
which is induced from the distinguished triangle $\tau_\ast \mathcal K_{Z/S}=\tau_\ast \tau^!\mathcal K_{X/S}\to \mathcal K_{X/S}\to j_\ast \mathcal K_{U/S}\xrightarrow{+1}$.
\end{proof}
\begin{remark}
	Assume that $X$ and $S$ are smooth schemes over a field and that $X\to S$ is of relative dimension $r$.
	If $Z\hookrightarrow X$ is a closed immersion of codimension $>r$, then we have $H^0(Z,\mathcal K_{Z/{S}})=0$ and $H^1(Z,\mathcal K_{Z/{S}})=0$ by semi-purity (cf. \cite[\S 8]{Fuj02}).
\end{remark}
\begin{definition}\label{def:rcccCodBig}
Under the assumptions in Lemma \ref{lem:rcccCodBig}.
Let $\mathcal F\in D_{\rm ctf}(X,\Lambda)$ such that $\mathcal F|_U\in D_{\rm ctf}(U/S,\Lambda)$.
We define the relative (cohomological) characteristic class of $\mathcal F$ to be
\begin{align}\label{eq:RCCCCDEF3}
C_{X/S}(\mathcal F)\coloneqq C_{U/S}(\mathcal F|_U) \quad{\rm in}\quad H^0(X,\mathcal K_{X/S}).
\end{align}
\end{definition}

\subsection{}\label{subsec:CCor}
For any $\mathcal F\in D_{\rm ctf}(X/S,\Lambda)$ and $\mathcal G\in D_{\rm ctf}(X,\Lambda)$,
we call elements of ${\rm Hom}(c_2^*\mathcal G, c_1^!\mathcal F)$   {\it cohomological correspondences} from $\mathcal G$  to $\mathcal F$. By the isomorphism \eqref{eq:keyisoKu}, a cohomological correspondence $u: c_2^*\mathcal G\to c_1^!\mathcal F$ is equivalent to a map $\Lambda\to  c^!R\mathcal Hom({\rm pr}^*_2\mathcal G, {\rm pr}^!_1\mathcal F)$ (resp. $\Lambda\to  c^!(\mathcal F\boxtimes_S^L D_{X/S}(\mathcal F))$). By adjunction, this is also equivalent to a map
$c_!\Lambda\to R\mathcal Hom({\rm pr}^*_2\mathcal G, {\rm pr}^!_1\mathcal F)$ (resp. $c_!\Lambda\to  \mathcal F\boxtimes_S^L D_{X/S}(\mathcal F)$).
In the following, we will identify these descriptions freely. 
\subsection{}\label{subsec:equiDefOfChar} 
We give an equivalent description of the relative (cohomological) characteristic classes.
Let $X\in{\rm Sch}_S$ and $\delta: X\to X\times_SX$ the diagonal morphism.
Let $\mathcal F\in D_{\rm ctf}(X/S, \Lambda)$.
The canonical isomorphism
$R\mathcal Hom_{X\times_SX}({\rm pr}^*_2\mathcal F, {\rm pr}^!_1\mathcal F)\xleftarrow{\simeq}\mathcal F\boxtimes_S^{L}D_{X/S}(\mathcal F)$
induces an isomorphism
\[
\delta^\ast R\mathcal Hom_{X\times_SX}({\rm pr}^*_2\mathcal F, {\rm pr}^!_1\mathcal F)\xleftarrow{\simeq}\mathcal F\otimes^{L}D_{X/S}(\mathcal F). 
\]
Thus the evaluation map $\mathcal F\otimes^{L}D_{X/S}(\mathcal F)\to \mathcal K_{X/S}$ induces a map
\begin{align}\label{eq:keyisoKu3}
v_{\mathcal F}: \delta^\ast R\mathcal Hom_{X\times_SX}({\rm pr}^*_2\mathcal F, {\rm pr}^!_1\mathcal F)
\to \mathcal K_{X/S}.
\end{align}
Let $u: \delta_!\Lambda\to R\mathcal Hom({\rm pr}_2^\ast\mathcal F,{\rm pr}^!_1\mathcal F)$ be a cohomological
correspondence. Then 
${\rm Tr}(u)\in H^0(X,\mathcal K_{X/S})$  equals  the cohomology class 
of the following composition 
\begin{align}\label{eq:traceEquiDef}
\Lambda=\delta^\ast\delta_!\Lambda\xrightarrow{\delta^\ast(u)}\delta^\ast R\mathcal Hom_{X\times_SX}({\rm pr}^*_2\mathcal F, {\rm pr}^!_1\mathcal F)\xrightarrow{v_{\mathcal F}}\mathcal K_{X/S}.
\end{align}
We also  write the class ${\rm Tr}(u)$ in the form
\begin{align}\label{eq:traceEquiDefadj}
	\delta_!\Lambda\xrightarrow{u} R\mathcal Hom_{X\times_SX}({\rm pr}^*_2\mathcal F, {\rm pr}^!_1\mathcal F)\xrightarrow{v_{\mathcal F}}\delta_\ast\mathcal K_{X/S}.
\end{align}
Taking $u={\rm id}_{\mathcal F}$ to be the 
cohomological correspondence defined by the identity of $\mathcal F$, we get that $C_{X/S}(\mathcal F)={\rm Tr}({\rm id}_{\mathcal F})$ equals the cohomology class of  the composition
\begin{align}\label{eq:traceEquiforcharclass}
	\delta_!\Lambda\xrightarrow{{\rm id}_{\mathcal F}} R\mathcal Hom_{X\times_SX}({\rm pr}^*_2\mathcal F, {\rm pr}^!_1\mathcal F)\xrightarrow{v_{\mathcal F}}\delta_\ast\mathcal K_{X/S}.
\end{align}

\subsection{}We first review the pull-back property of the  relative characteristic classes.
Consider a cartesian diagram of schemes
\begin{align}\label{eq:3pullbackFixDia}
\begin{gathered}
 \xymatrix{ X'\ar[d]_-{f'} \ar[r]^-{b_X} & X  \ar[d]^-f \\
S'\ar[r]^-{b} &S,\ar@{}|\Box[ul]
}
\end{gathered}
\end{align}
where $f$ is a separated morphism of finite type, and $b:S'\to S$ is a morphism between Noetherian schemes.
For any $\mathcal{F}\in D_{\rm ctf}(X/S,\Lambda)$, the object $b_X^\ast\mathcal F$ on $X'$ is also universally locally acyclic over $S'$. 
Thus we have a  well-defined functor
\begin{align}\label{eq:pullbackWelldefined}
b_X^*\colon &D_{\rm ctf}(X/S,\Lambda)\to D_{\rm ctf}(X'/S',\Lambda),
\end{align}
and a well-defined map
\begin{align}
\label{eq:pullbackWelldefined2} b_X^*\colon &K_0(X/S,\Lambda)\to K_0(X'/S',\Lambda).
\end{align}
Note that the following base change morphism
\begin{equation}\label{eq:3pullbackK}
	b_X^\ast \mathcal K_{X/S}=b_X^\ast f^!\Lambda\xrightarrow{\rm b.c}f^{\prime !} b^\ast\Lambda=\mathcal K_{X'/S'}
\end{equation}
induces a pull-back morphism on cohomology groups
\begin{equation}\label{eq:pullbackCG}
b_X^*\colon H^0(X,\mathcal K_{X/S}) \to H^0(X', \mathcal K_{X'/S'}).
\end{equation}
\begin{proposition}\label{cor:pullbackccc}
Given the cartesian diagram \eqref{eq:3pullbackFixDia},
the following diagram commutes
\begin{align}
\begin{gathered}
\xymatrix@C=3.5pc{
K_0(X/S,\Lambda)\ar[d]_{b_X^\ast}^{\eqref{eq:pullbackWelldefined2}}\ar[r]^-{~C_{X/S}~}&H^0(X, \mathcal K_{X/S})\ar[d]_{b_X^*}^-{\eqref{eq:pullbackCG}}\\
K_0(X'/S',\Lambda)\ar[r]^-{~C_{X'/S'}~}&H^0(X', \mathcal K_{X'/S'}).
}
\end{gathered}
\end{align}
\end{proposition}
In a previous version of this paper,  Proposition \ref{cor:pullbackccc} is proved by verifying the commutativity of certain diagrams. Weizhe Zheng remarked that the categorical method in \cite{LZ22} could also be used here in an email to the authors and Tomoyuki Abe. We refer to \cite[1.3]{Abe22} for that proof.
\begin{proof}
Here we give a sketched proof.
Let $\mathcal F\in D_{\rm ctf}(X/S,\Lambda)$ and  $\mathcal F'=b_X^\ast\mathcal F$.
We put 
\begin{align*}
	\mathcal H_S&=R\mathcal Hom_{X\times_SX}({\rm pr}_2^\ast \mathcal F,{\rm pr}_1^!\mathcal F),\;\;\qquad\mathcal T_S=\mathcal F\boxtimes^L_S D_{X/S}(\mathcal F),\\
	\mathcal H_{S'}&=R\mathcal Hom_{X'\times_{S'}X'}({\rm pr}_2^\ast \mathcal F',{\rm pr}_1^!\mathcal F'),\quad \mathcal T_{S'}=\mathcal F'\boxtimes^L_{S'} D_{X'/S'}(\mathcal F').
\end{align*}
Consider the following cartesian diagram
\begin{align}\label{eq:pullCproofhelp1}
\begin{gathered}
\xymatrix@R=0.6cm@C=4pc{
X'\ar[d]_-{\delta'}\ar[r]^-{b_X}\ar@{}|\Box[rd]&X\ar[d]^-{\delta}\\
X'\times_{S'}X'\ar[r]^-{b_X\times b_X}&X\times_SX,
}
\end{gathered}
\end{align}
where $\delta$ and $\delta'$ are the diagonal morphisms.
The characteristic class $C_{X/S}(\mathcal F)$ is given by
the composition $\Lambda\xrightarrow{}\delta^! \mathcal H_S\xleftarrow[\simeq]{\eqref{eq:Kunethiso}}\delta^!\mathcal T_S
\xrightarrow{}\delta^\ast \mathcal T_S \xrightarrow{\rm ev} \mathcal K_{X/S}$ (cf. \eqref{eq:evaMap}-\eqref{eq:RCCCCDEF}).
Similar for $C_{X'/S'}(\mathcal F')$.
There is a commutative diagram 
\begin{align}\label{eq:pullCproof}
\begin{gathered}
\xymatrix@R=0.6cm{
\Lambda\ar[r]&\delta^{\prime!}\mathcal H_{S'}&\delta^{\prime!}\mathcal T_{S'}\ar[l]_-\simeq\ar[r]&\delta^{\prime\ast}\mathcal T_{S'}\ar[r]^-{\rm ev}&\mathcal K_{X'/S'}\\
b_X^\ast\Lambda\ar[r]\ar@{=}[u]&b_X^\ast\delta^{!}\mathcal H_S\ar[u]_-{(a)}&b_X^\ast \delta^!\mathcal T_S\ar[r]\ar[u]_-{(b)}\ar[l]_-\simeq &b_X^\ast\delta^\ast \mathcal T_S\ar[r]^-{\rm ev}\ar[u]_-{(c)}&b_X^\ast\mathcal K_{X/S}\ar[u]_-{\eqref{eq:3pullbackK}},
}
\end{gathered}
\end{align}
where $(a)$ is given by
	$b_X^\ast\delta^!\mathcal H_S\xrightarrow[\eqref{eq:pullbackforanyGdefTrans}]{\rm b.c}\delta^{\prime!}(b_X\times b_X)^\ast\mathcal H_S\xrightarrow{\rm b.c}\delta^{\prime!}\mathcal H_{S'}$,
the morphism $(b)$ is given by
	$b_X^\ast \delta^!\mathcal T_S \xrightarrow[\eqref{eq:pullbackforanyGdefTrans}]{\rm b.c} \delta^{\prime!}(b_X\times b_X)^\ast \mathcal T_S \xrightarrow{\rm b.c}\delta^{\prime!}\mathcal T_{S'}$,
and $(c)$ is given by $b_X^\ast\delta^{\ast}\mathcal T_S\xrightarrow{\simeq}\delta^{\prime\ast}(b_X\times b_X)^\ast\mathcal T_S\xrightarrow{\rm b.c} \delta^{\prime\ast}\mathcal T_{S'}$.
By the above commutative diagram \eqref{eq:pullCproof}, we get $b_X^\ast C_{X/S}(\mathcal F)=C_{X'/S'}(b_X^\ast\mathcal F)$.
\end{proof}
\subsection{}Now we recall the proper push-forward property of the relative characteristic classes.
Let $s: X\to X'$ be a proper morphism  in ${\rm Sch}_S$. 
For any $\mathcal F\in D_{\rm ctf}(X/S,\Lambda)$, then 
$s_\ast \mathcal F$ is universally locally acyclic over $S$ by \cite[Theorem 2.16 and Proposition 2.23]{LZ22}.
Thus we have a well-defined functor
\begin{align}
s_\ast\colon D_{\rm ctf}(X/S,\Lambda)\to D_{\rm ctf}(X'/S,\Lambda),
\end{align}
and a well-defined map
\begin{align}
s_\ast\colon K_0(X/S,\Lambda)\to K_0(X'/S,\Lambda).
\end{align}

\begin{proposition}\label{prop:pushccc}
If $s\colon X\to X'$ is a proper morphism in ${\rm Sch}_S$, then
 the following diagram commutes
\begin{align}\label{eq:cor:pushccc2}
\begin{gathered}
\xymatrix@C=3pc{
K_0(X/S,\Lambda)\ar[d]_{s_\ast}\ar[r]^-{~C_{X/S}~}&H^0(X, \mathcal K_{X/S})\ar[d]^{s_*}\\
K_0(X'/S,\Lambda)\ar[r]^-{~C_{X'/S}~}&H^0(X', \mathcal K_{X'/S}),
}
\end{gathered}
\end{align}
where $s_\ast: H^0(X,\mathcal K_{X/S})\to H^0(X',\mathcal K_{X'/S})$ is induced by $s_\ast\mathcal K_{X/S}=s_! s^!\mathcal K_{X'/S}\xrightarrow{\rm adj} \mathcal K_{X'/S}$.
\end{proposition}
\begin{proof}
For a proof of this proposition, we refer to  \cite[Corollary 2.22]{LZ22} (see also \cite[Corollary 3.3.4]{YZ18}).
Here we give a sketched proof. Let $\mathcal F\in D_{\rm ctf}(X/S,\Lambda)$ and  $\mathcal F'=s_\ast\mathcal F$.
Consider the following commutative diagram
\begin{align}\label{eq:pushCproofhelp1}
\begin{gathered}
\xymatrix@C=1.5cm{
X\ar@/^1pc/[rr]^-s\ar[r]_-(0.7){\delta_0}\ar[rd]_-\delta&X\times_{X'}X\ar[d]^-\beta\ar[r]_-\alpha\ar@{}[rd]|{\Box}&X'\ar[d]^-{\delta'}\\
&X\times_SX\ar[r]^-{s\times s}&X'\times_SX',
}
\end{gathered}
\end{align}
where $\delta$, $\delta_0$ and $\delta'$ are the diagonal morphisms.
There is a commutative diagram 
\begin{align}\label{eq:pushCproof}
\begin{gathered}
\xymatrix@R=0.6cm{
\Lambda\ar[r]\ar[d]&\delta^{\prime\ast}(\mathcal F'\boxtimes^L_S D_{X'/S}(\mathcal F'))\ar[r]^-\simeq&\mathcal F'\otimes^L D_{X'/S}(\mathcal F')\ar[r]^-{\rm ev}&\mathcal K_{X'/S}\\
s_\ast\Lambda\ar[r]&s_\ast\delta^{\ast}(\mathcal F\boxtimes^L_S D_{X/S}(\mathcal F))\ar[r]^-\simeq\ar[u]_-{(a)}&s_\ast(\mathcal F\otimes^L D_{X/S}(\mathcal F))\ar[r]^-{\rm ev}&s_\ast\mathcal K_{X/S},\ar[u].
}
\end{gathered}
\end{align}
where $(a)$ is given by
\begin{align}\label{eq:pushCproofhelp2}
\begin{split}
s_\ast\delta^{\ast}(\mathcal F\boxtimes^L_S D_{X/S}(\mathcal F))
&\simeq \alpha_\ast\delta_{0\ast}\delta_0^\ast \beta^{\ast}(\mathcal F\boxtimes^L_S D_{X/S}(\mathcal F))\xrightarrow{\rm adj}\alpha_\ast \beta^{\ast}(\mathcal F\boxtimes^L_S D_{X/S}(\mathcal F))\\
&\xrightarrow[\simeq]{\rm proper~ b.c}\delta^{\prime\ast}(s\times s)_\ast(\mathcal F\boxtimes^L_S D_{X/S}(\mathcal F))\to \delta^{\prime\ast}(\mathcal F'\boxtimes^L_S D_{X'/S}(\mathcal F')).
\end{split}
\end{align}
By the commutativity of \eqref{eq:pushCproof} and an equivalent description \eqref{eq:traceEquiforcharclass} of the characteristic classes, we get $s_\ast C_{X/S}(\mathcal F)=C_{X'/S}(s_\ast\mathcal F)$.
\end{proof}

\subsection{}\label{sub:notaionformainBC}
We are now ready to prove Theorem \ref{introconj1} under the condition $Z=\emptyset$. 
Consider a commutative diagram in ${\rm Sch}_S$:
\begin{align}\label{eq:xys-thm}
	\begin{gathered}
		\xymatrix{
			X\ar[rr]^-f\ar[rd]_-h&&Y.\ar[ld]^-g\\
			&S&
		}
	\end{gathered}
\end{align}
Assume that $\mathcal K_{Y/S}=g^!\Lambda$ is locally constant and that $\mathcal K_{Y/S}\otimes^L \delta^!\Lambda\simeq \Lambda$ (cf. \eqref{eq:KYSLambda}). 
Consider the commutative diagram \eqref{xyd:prop:Gysim:diagonaldeltai}.
By Lemma \ref{lem:TransBforC}.(2), we get a commutative diagram
\begin{align}\label{eq:defDistrik1-section3}
	\begin{gathered}
		\xymatrix@C=2.5cm{
			\mathcal K_{X/Y}&\mathcal K_{X/{S}}\otimes^L f^\ast\delta^!\Lambda\ar[r]^-{c_{\delta,(f\times f)\delta_0,\mathcal K_{X/S}}}\ar[l]^-\simeq_-{\rm Prop.\ref{prop:Gysim}.(2)}& \mathcal K_{X/S}\\
			&\delta_1^\ast (i^\ast\delta_{0\ast}\mathcal K_{X/{S}}\otimes p^\ast \delta^!\Lambda)\ar[r]^-{\delta_1^\ast(c_{\delta,f\times f, \delta_{0\ast}\mathcal K_{X/S}})}\ar[u]_-{\simeq}&
			\delta_1^\ast i^!\delta_{0\ast}\mathcal K_{X/{S}}\ar[u]_-{\simeq},
		}
	\end{gathered}
\end{align}
which induces a homomorphism
\begin{align}\label{eq:cherncup1}
	\delta^!:H^{0}(X,\mathcal K_{X/Y})\xrightarrow{}  H^{0}(X,\mathcal K_{X/S}).
\end{align}
When $g: Y \to S$ is a smooth morphism, for any $\mathcal F\in D_{\rm ctf}(X/Y,\Lambda)$,
we have  $\mathcal F\in D_{\rm ctf}(X/S,\Lambda)$
by \cite[Appendice to Th.finitude, Corollaire 2.7]{SGA4h}. In this case, there is a canonical homomorphism
\begin{align}\label{eq:K2S}
	K_0(X/Y,\Lambda)\to K_0(X/S,\Lambda);\quad \mathcal F\mapsto \mathcal F.
\end{align}
\begin{theorem}[Fibration formula: transversal case] \label{thm:changebaseHomo}
	Consider the notation and assumptions in \ref{sub:notaionformainBC}. 
For any $\mathcal F\in D_{\rm ctf}(X,\Lambda)$ such that $f: X\to Y$ and $h: X\to S$ are universally locally acyclic relatively to  $\mathcal F$,  we have a commutative diagram
	\begin{equation}\label{eq:propchangebasehomo:4}
		\begin{gathered}
			\xymatrix{
				{\rm Hom}(\mathcal F, \mathcal F) \ar[r]^-{\rm Tr}\ar@{=}[d] & H^0(X,\mathcal K_{X/Y}) \ar[d]_-{\delta^!}^-{{\eqref{eq:cherncup1}}}\\
				{\rm Hom}(\mathcal F, \mathcal F) \ar[r]^-{\rm Tr}& H^0(X, \mathcal K_{X/S}).
			}
		\end{gathered}
	\end{equation}
If $g: Y\to S$ is smooth of relative dimension $r$,
then we have 
\begin{align}\label{delta!Cherneq}
	\delta^!=c_{r}(f^*\Omega_{Y/S}^{1,\vee}): H^0(X,\mathcal K_{X/Y})\to H^0(X, \mathcal K_{X/S})
\end{align} 
and the following commutative diagram  
	\begin{equation}\label{eq:propchangebasehomo:4-1}
		\begin{gathered}
			\xymatrix{
				K_0(X/Y, \Lambda) \ar[r]^-{\quad C_{X/Y}\quad}\ar[d]_{\eqref{eq:K2S}} & H^0(X,\mathcal K_{X/Y}) \ar[d]^-{c_{r}(f^*\Omega_{Y/S}^{1,\vee})}\\
				K_0(X/S, \Lambda) \ar[r]^-{\quad C_{X/S}\quad}& H^0(X, \mathcal K_{X/S}).
			}
		\end{gathered}
	\end{equation}
	
\end{theorem}
If  $X=Y$ is  smooth over $S$ and connected, then any $\mathcal F\in D_{\rm ctf}(X/X,\Lambda)$ has locally constant cohomology sheaves and $C_{X/S}(\mathcal F)={\rm rank}\mathcal F\cdot c_r(\Omega_{X/S}^{1,\vee})$. 

\begin{proof}
	First note that for any closed immersion $d$, there is a canonical morphism  $d^!\to d^\ast$ defined by
	$d^!\simeq d^\ast d_! d^!\xrightarrow{\rm adj} d^\ast$.
	Consider the commutative diagram \eqref{xyd:prop:Gysim:diagonaldeltai}.
	In the following, we put 
\begin{align}
	\mathcal H_S=R\mathcal Hom_{X\times_SX}({\rm pr}_2^\ast \mathcal F,{\rm pr}_1^!\mathcal F),&\quad \mathcal T_S=\mathcal F\boxtimes^L_S D_{X/S}(\mathcal F),\\
	\mathcal H_Y=R\mathcal Hom_{X\times_YX}({\rm pr}_2^\ast \mathcal F,{\rm pr}_1^!\mathcal F),&\quad \mathcal T_Y=\mathcal F\boxtimes^L_Y D_{X/Y}(\mathcal F).
\end{align}
	For any $\alpha:\mathcal F\to\mathcal F$ which is given by a morphism
	$\Lambda\to R\mathcal Hom(\mathcal F,\mathcal F)=\delta_0^! \mathcal H_S$, the trace
	$C_{X/S}(\alpha):={\rm Tr}(\alpha)\in H^0(X,\mathcal K_{X/S})$ is given by
	the composition (cf. \eqref{eq:evaMap}-\eqref{eq:RCCCCDEF})
	\begin{align}\label{eq:propchangebasehomo:5}
			\Lambda\xrightarrow{\alpha}\delta_0^! \mathcal H_S\xleftarrow[\simeq]{\eqref{eq:Kunethiso}}\delta_0^!\mathcal T_S
			\xrightarrow{}\delta_0^\ast \mathcal T_S \xrightarrow{\rm ev} \mathcal K_{X/S},
	\end{align}
	and the trace $C_{X/Y}(\alpha):={\rm Tr}(\alpha)\in H^0(X,\mathcal K_{X/Y})$ is given by
	the composition
	\begin{align}\label{eq:propchangebasehomo:6}
			\Lambda\xrightarrow{\alpha}\delta_1^! \mathcal H_Y\xleftarrow[\simeq]{\eqref{eq:Kunethiso}}\delta_1^!\mathcal T_Y
			\xrightarrow{}\delta_1^\ast \mathcal T_Y\xrightarrow{\rm ev} \mathcal K_{X/Y}.
	\end{align}
The morphism $\mathcal K_{X/Y}\xrightarrow{\eqref{eq:defDistrik1-section3}} \mathcal K_{X/S}$  induces a map $\mathcal T_Y\to i^\ast\mathcal T_S$.
There is a commutative diagram 
\begin{align}\small\label{eq:propchangebasehom:TS1}
	\begin{gathered}
		\xymatrix@R=0.3cm@C=1.8cm{
			\mathcal H_Y\ar[dd]_-{\eqref{eq:keyisoCshrinkRHOM}}^-\simeq&\mathcal T_Y\ar[dd]_-{\eqref{lem:eq:Gysim2}}^-\simeq\ar[l]_-{\eqref{eq:Kunethiso}}^-{\simeq}\ar@/^1.5pc/[rrd]^-{\eqref{eq:defDistrik1-section3}}\\
			&&i^\ast\mathcal T_S\otimes^L p^\ast\delta^!\Lambda\ar[r]^(0.5){\eqref{eq:pullbackforanyGdefTrans}}_(0.5){\rm b.c}\ar[lu]_{\eqref{eq:KXYSfdelta}}^{\simeq}\ar[ld]^{\eqref{eq:carGtrans1}}_{\simeq}&i^\ast\mathcal T_S,\\
			i^!\mathcal H_S&i^!\mathcal T_S\ar[l]_-{\eqref{eq:Kunethiso}}^-{\simeq}\ar@/_1.5pc/[rru]^-{i^!\to i^\ast }
		}
	\end{gathered}
\end{align}
where the left pentagon is commutative by Lemma \ref{lem:iShrinkRhom} below, the upper right triangle is commutative by the definition of $\mathcal T_Y\to i^\ast\mathcal T_S$ induced by \eqref{eq:defDistrik1-section3}, and the lower right triangle is commutative by \eqref{eq:Dchf-two}.
Applying $\delta_1^!$ to \eqref{eq:propchangebasehom:TS1} and noting that $\delta_0^!=\delta_1^!i^!$, we get a commutative diagram
	\begin{align}\label{eq:propchangebasehom:TS2}
	\begin{gathered}
		\xymatrix{
			\Lambda\ar[r]\ar@{=}[d]&\delta_1^!\mathcal H_Y\ar[d]_-{\eqref{eq:keyisoCshrinkRHOM}}^-\simeq &\delta_1^!\mathcal T_Y \ar[l]_-{\eqref{eq:Kunethiso}}^-{\simeq} \ar[r]^-{\delta_1^!\to\delta_1^\ast }\ar[d]_-{\eqref{lem:eq:Gysim2}}^-\simeq& \delta_1^\ast \mathcal T_Y\ar[d]_-{\eqref{lem:eq:Gysim2}}^-\simeq\ar[rd]^-{\eqref{eq:defDistrik1-section3}}\ar[rr]^-{\rm ev}&&\mathcal K_{X/Y}\ar[d]_-{\eqref{eq:defDistrik1-section3}}\\
			\Lambda\ar[r]&\delta_0^!\mathcal H_S&\delta_0^!\mathcal T_S=\delta_1^! i^!\mathcal T_S\ar[l]_-{\eqref{eq:Kunethiso}}^-{\simeq} \ar[r]_-{\delta_1^!\to\delta_1^\ast }&\delta_1^\ast i^!\mathcal T_S\ar[r]_-{i^!\to i^\ast }&\delta_1^\ast i^\ast\mathcal T_S=\delta_0^\ast\mathcal T_S\ar[r]^-{\rm ev}&\mathcal K_{X/S}. 
		}
	\end{gathered}
\end{align}
From this we get $C_{X/S}(\alpha)=\delta^!(C_{X/Y}(\alpha))$, which shows the commutativity of \eqref{eq:propchangebasehomo:4}.

If $g:Y\to S$ is smooth, the equality \eqref{delta!Cherneq} is obtained by applying Lemma \ref{lem:sameChernMapS} to the cartesian diagram \eqref{xyd:prop:Gysim:diagonaldeltai}.
\end{proof}

\begin{lemma}\label{lem:iShrinkRhom}
	Consider the following commutative diagrams in ${\rm Sch}_S$	
	\begin{align}\small
		\begin{gathered}
			\xymatrix{
				&X' \ar[d]_{f'}\ar[dr]^{h'}&          &X'\ar[rr]^(0.4){f'}\ar@{}|\Box[rrd]&&Y\\
				X\ar[r]^f \ar@/_1pc/[rr]_-{h}&Y\ar[r]^g &S &X\times_Y X'\ar@{}|\Box[rd]\ar@/^1pc/[rr]^(0.4){q_1}\ar[u]^-{q_2}\ar[r]_-i \ar[d]_-p&X\times_SX'\ar[r]_-{p_1}\ar[d]_-{f\times f'}\ar[rd]_-{p_2}\ar@{}|\Box[rrd]&X\ar[u]_-f\ar[rd]^-h\\
				&&&Y\ar[r]^-\delta&Y\times_SY&X'\ar[r]^-{h'}&S,
			}
		\end{gathered}
	\end{align}
where $\delta$ is the diagonal morphism, $i$ is the base change of $\delta$,
$p_i={\rm pr}_i$ and $q_j={\rm pr}_j$ are the projections, and $p=f'q_2$.
Let $\mathcal{F}\in D_{\rm{ctf}}(X,\Lambda)$ and $ \mathcal{G}\in D_{\rm{ctf}}(X',\Lambda)$.
Assume that $\mathcal K_{Y/S}=g^!\Lambda$ is locally constant and that $\mathcal K_{Y/S}\otimes^L \delta^!\Lambda\simeq \Lambda$ $($cf. \eqref{eq:KYSLambda}$)$. 
Then we have a commutative diagram
\begin{align}\label{eq:lem:iShrinkRhom-2}
	\begin{gathered}\small
		\xymatrix{
			i^\ast(\mathcal{F}\boxtimes^L_{S}D_{X'/S}(\mathcal{G}))\otimes p^\ast \delta^!\Lambda\ar[r]_-{\simeq}^-{\eqref{eq:KXYSfdelta}}\ar[rd]_{c_{\delta,f\times f',\mathcal F\boxtimes_S^L D_{X'/S}(\mathcal G)}}&\mathcal{F}\boxtimes^L_YD_{X'/Y}(\mathcal{G}) \ar[r]^-{\eqref{eq:Kunethiso}}\ar[d]& R\mathcal Hom_{X\times_{Y}X'}({q}_2^*\mathcal{G},{q}^!_1\mathcal{F}) \ar[d]_{}^-{\eqref{eq:keyisoCshrinkRHOM}}\\
			&i^!(\mathcal{F}\boxtimes^L_{S}D_{X'/S}(\mathcal{G})) \ar[r]^-{i^!\eqref{eq:Kunethiso}}&i^!R\mathcal Hom_{X\times_{S}X'}({p}_2^*\mathcal{G},{p}^!_1\mathcal{F}).
		}
	\end{gathered}
\end{align}
\end{lemma}

\begin{proof}In the following, we identify $\mathcal K_{X'/Y}$ with  $\mathcal K_{X'/S}\otimes^L f^{\prime\ast}\delta^!\Lambda$ by using the isomorphism \eqref{eq:KXYSfdelta}.
The composition 
$\mathcal{F}\boxtimes^L_YD_{X'/Y}(\mathcal{G}) \xrightarrow{\eqref{eq:Kunethiso}}R\mathcal Hom({q}_2^*\mathcal{G},{q}^!_1\mathcal{F})\xrightarrow{}i^!R\mathcal Hom({p}_2^*\mathcal{G},{p}^!_1\mathcal{F})$ is adjoint 
to  the following composition (cf. \eqref{Def-eq:Kunethiso} and \eqref{Def-eq:keyisoCshrinkRHOM})
\begin{align}
	i_!(\mathcal{F}\boxtimes^L_YD_{X'/Y}(\mathcal{G})\otimes^L {q}_2^\ast\mathcal G)\xrightarrow{\rm ev}i_!(\mathcal F\boxtimes^L_Y\mathcal K_{X'/Y})\xrightarrow{i_!c_{f',f,\mathcal F}}i_!{q}_1^!\mathcal F\simeq i_!i^!{p}_1^!\mathcal F\xrightarrow{\rm adj}{p}_1^!\mathcal F,
\end{align}
and 
$\mathcal{F}\boxtimes^L_YD_{X'/Y}(\mathcal{G}) \xrightarrow{c_{\delta,f\times f'}}i^!(\mathcal{F}\boxtimes^L_{S}D_{X'/S}(\mathcal{G}))\xrightarrow{}i^!R\mathcal Hom({p}_2^*\mathcal{G},{p}^!_1\mathcal{F})$ is adjoint to
	\begin{align}
		\begin{split}
		i_!(\mathcal{F}\boxtimes^L_YD_{X'/Y}(\mathcal{G})\otimes {q}_2^\ast\mathcal G)\xrightarrow{c_{\delta,f\times f'}}\mathcal F\boxtimes^L_S D_{X'/S}(\mathcal G)\otimes {q}_2^\ast\mathcal G\xrightarrow{\rm ev}\mathcal F\boxtimes^L_S\mathcal K_{X'/S}\xrightarrow{c_{h',h,\mathcal F}} {p}_1^!\mathcal F.
		\end{split}
	\end{align}
Now the result follows from the following commutative diagram
\begin{align}\label{eq:finalCheck}
	\begin{gathered}
		\xymatrix@C=1.5cm{
			\mathcal{F}\boxtimes^L_YD_{X'/Y}(\mathcal{G})\otimes^L{q}_2^\ast\mathcal G \ar[r]^-{\rm ev}\ar[d]_{c_{\delta,f\times f',\mathcal F\boxtimes_S^L D_{X'/S}(\mathcal G)}}&\mathcal F\boxtimes_Y^L\mathcal K_{X'/Y}\ar[r]^-{c_{f',f,\mathcal F}}\ar[d]_-{c_{\delta,f\times f',\mathcal F\boxtimes^L\mathcal K_{X'/S}}}& {q}_1^!\mathcal F \ar[d]^-{\simeq}\\
			i^!(\mathcal{F}\boxtimes^L_{S}D_{X'/S}(\mathcal{G})\otimes^L{p}_2^\ast\mathcal G) \ar[r]^-{\rm ev}&i^!(\mathcal F\boxtimes_S^L\mathcal K_{X'/S})\ar[r]^-{i^!c_{h',h,\mathcal F}}&i^!{p}_1^!\mathcal F,
		}
	\end{gathered}
\end{align}
where the left diagram is commutative by applying $c_{\delta,f\times f',-}$ to the  map $\mathcal F\boxtimes_S^L D_{X'/S}(\mathcal G)\xrightarrow{\rm ev}\mathcal F\boxtimes^L_S\mathcal K_{X'/Y}$, the right diagram is commutative since the two compositions are both adjoint to
\begin{align}
\begin{split}
	q_{1!}(q_1^\ast\mathcal F\otimes^L q_2^\ast\mathcal K_{X'/Y})&\xrightarrow{\rm proj.formula}\mathcal F\otimes^L q_{1!}q_2^\ast \mathcal K_{X'/Y}=\mathcal F\otimes^L q_{1!}q_2^\ast f^{\prime !}\Lambda\\
	&\xrightarrow{\eqref{eq:pullbackforanyGdefTrans}}\mathcal F\otimes^L q_{1!}q_1^!f^\ast \Lambda \xrightarrow{\rm adj}\mathcal F\otimes^L f^\ast\Lambda=\mathcal F.
\end{split}
\end{align}
\end{proof}

For later convenience, we note the following blow-up formula for characteristic classes.
\begin{lemma}\label{lem:Bup}
	Let $X$ and $Y$ be smooth connected schemes over a perfect field $k$ and let $i\colon Y\hookrightarrow X$ be a closed immersion of codimension $r\geq 1$.
	Let $\pi\colon \widetilde{X}\to X$ be the blow-up of $X$ along $Y$ and $\mathcal F\in D_{\rm ctf}(X,\Lambda)$.
	Then we have:
	\begin{align}
		\label{eq:blowupC}\pi_\ast(C_{\widetilde{X}/k}(\pi^\ast\mathcal F))&=C_{X/k}(\mathcal F)+ (r-1)\cdot i_\ast(C_{Y/k}(i^\ast\mathcal F))\quad{\rm in}\quad H^0(X,\mathcal K_{X/k}),\\
		\label{eq:blowupcc}
		\pi_\ast(cc_{\widetilde{X}}(\pi^\ast\mathcal F))&=cc_X(\mathcal F)+ (r-1)\cdot i_\ast(cc_Y(i^\ast\mathcal F))\qquad{\rm in}\quad CH_0(X),
	\end{align}
where $cc_{X}:=cc_{X,0}: K_0(X,\Lambda)\to {\rm CH}_0(X)$ is the morphism defined in \cite[Definition 6.7.2]{Sai17a} via the characteristic cycles of constructible \'etale sheaves.
\end{lemma}
In this paper, we only need the case where $r=2$.
\begin{proof}
	The equality \eqref{eq:blowupcc} follows from \cite[Lemma 3.3.2 and Remark 3.3.3]{UYZ}. We prove \eqref{eq:blowupC}.
	By \cite[Expos\'e XVI, Proposition 2.2.2.1]{ILO14}, we have an isomorphism on $X$:
	\begin{equation}\label{eq:blowupsheaf:1}
		R\pi_\ast \Lambda\simeq \Lambda\oplus \bigoplus_{t=1}^{r-1}(i_\ast \Lambda)(-t)[-2t].
	\end{equation}
	By the projection formula, we have 
	\begin{equation}\label{eq:blowupsheaf:3}
		R\pi_\ast \Lambda\otimes^L\mathcal F\simeq R\pi_\ast(\Lambda\otimes^L\pi^\ast\mathcal F)=R\pi_\ast\pi^\ast\mathcal F \qquad {\rm and}\qquad i_\ast\Lambda\otimes^L\mathcal F\simeq i_\ast i^\ast\mathcal F.
	\end{equation}
	By \eqref{eq:blowupsheaf:1} and \eqref{eq:blowupsheaf:3}, we get
	\begin{equation}\label{eq:blowupsheaf:4}
		R\pi_\ast \pi^\ast\mathcal F\simeq \mathcal F\oplus \bigoplus_{t=1}^{r-1}(i_\ast i^\ast\mathcal F)(-t)[-2t].
	\end{equation}
	Now the required formula \eqref{eq:blowupC} follows from Lemma \ref{lem:twsit}, Proposition \ref{prop:pushccc} and \eqref{eq:blowupsheaf:4}. 
\end{proof}

\subsection{}\label{NCdefsec}
Now we recall that the cohomological characteristic class is compatible with the specialization map.
Before that, let us recall the  nearby cycles functor over a general base. 
Let $h: X\to S$ be a morphism of finite type between Noetherian schemes. 
Let $s\leftarrow t$ be a specialization map of geometric points of $S$, i.e., 
a geometric point $t\to S_{(s)}$ of the strict henselization of $S$ at $s$. 
Consider the cartesian
diagram
\begin{align}\label{eq:specializationMap}
	\begin{gathered}
	\xymatrix@R=0.3cm{
	X_s\ar[d]\ar@{}|\Box[rd]\ar[r]^-i&X\times_SS_{(s)}\ar@{}|\Box[rd]\ar[d]&X_t\ar[l]_-{j}\ar[d]\\
	s\ar[r]&S_{(s)}&t.\ar[l]
}
	\end{gathered}
\end{align}
Then we define 
\begin{align}\label{eq:NCFunctordef}
	\Psi_{h,s\leftarrow t}:=i^\ast Rj_\ast: D^+(X_t,\Lambda)\to D^+(X_s,\Lambda).
\end{align}
In this paper, we denote $\Psi_{h,s\leftarrow t}$ by $\Psi_{h}$ for simplicity.
When $S_{(s)}$ is a trait, then $\Psi_{h,s\leftarrow t}$ is the classical nearby cycles functor (cf. \cite[Th.finitude, 3.1]{SGA4h}).

\subsection{}\label{subsec:spSec}
Assume that $S$ is a henselian trait with the generic point $\eta$. Let $s$ be a geometric point of $S$ above its closed point. 
Consider a commutative diagram in ${\rm Sch}_S$:
\begin{align}\label{spFGH}
	\begin{gathered}
		\xymatrix{
			X\ar[rr]^-f\ar[rd]_-h&&Y,\ar[ld]^-g\\
			&S&
		}
	\end{gathered}
\end{align}
with $g$ smooth. We denote again by $\Psi_h$ the composition $D^+(X_\eta,\Lambda)\to D^+(X_{\overline\eta},\Lambda)\xrightarrow{\Psi_{h,s\leftarrow \overline{\eta}}} D^+(X_s,\Lambda)$. 
Since $g$ is smooth, we have $\Psi_{g}(\Lambda_{Y_\eta})\simeq \Lambda_{Y_s}$.
There is a canonical morphism
\begin{align}\label{PsiKXYdef}
	\Psi_{h}(\mathcal K_{X_\eta/Y_\eta})\to \mathcal K_{X_s/Y_s}
\end{align}
given by the composition $\Psi_{h}(\mathcal K_{X_\eta/Y_\eta})=\Psi_{h}(f_\eta^!\Lambda)\xrightarrow{\rm b.c}f_s^!\Psi_{g}(\Lambda)\simeq f_s^!\Lambda=\mathcal K_{X_s/Y_s}$. The specialization map
\begin{align}\label{spDef}
	{\rm sp}={\rm sp}_s: H^0(X_\eta,\mathcal K_{X_\eta/Y_\eta})\to H^0(X_s, \mathcal{K}_{X_{s}/Y_{s}} )
\end{align}
is induced by the following composition
\begin{align}\label{spDef2}
	\begin{split}
		RHom(\Lambda,\mathcal K_{X_\eta/Y_\eta})
		&\xrightarrow{\Psi_{h}}RHom(\Psi_{h}(\Lambda),\Psi_{h}(\mathcal K_{X_\eta/Y_\eta}))\\
		&\xrightarrow{\Lambda\to \Psi_{h}(\Lambda)}RHom(\Lambda,\Psi_{h}(\mathcal K_{X_\eta/Y_\eta}))\xrightarrow{\eqref{PsiKXYdef}}RHom(\Lambda,\mathcal{K}_{X_{s}/Y_{s}}).
	\end{split}
\end{align}
When $g={\rm id}$, we have the following specialization formula for characteristic classes. 
\begin{proposition}[{cf. \cite[Corollary 3.10]{LZ22} and \cite[Proposition 1.3.5]{V07}}]\label{prop:sp}
	Let $S$ be a henselian trait with the generic point $\eta$. 
	Let $s$ be a geometric point of $S$ above its closed point. 
	Let $h:X\to S$ be a separated morphism of finite type and $\mathcal{F}\in D_{\rm ctf}(X_\eta,\Lambda)$. Then we have 
	\begin{align}\label{spFormula}
		{\rm sp}(C_{X_{\eta}/\eta}(\mathcal{F}))=C_{X_s/s}(\Psi_{h}(\mathcal{F})) \quad{\rm in}\quad H^0(X_s, \mathcal{K}_{X_{s}/{s}}),
	\end{align}
where ${\rm sp}: H^0(X_\eta,\mathcal K_{X_\eta/\eta})\to H^0(X_s, \mathcal{K}_{X_{s}/{s}} )$ is the specialization map.

In particular, if $\mathcal{F}\in D_{\rm ctf}(X,\Lambda) $ such that $h$ is universally locally acyclic relatively to $\mathcal{F}$, then $\Psi_{h}(\mathcal{F}|_{X_{\eta}})\simeq \mathcal{F}|_{X_s}$ and hence 
	\begin{align}\label{spFormulaCor}
		{\rm sp}(C_{X_{\eta}/{\eta}}(\mathcal{F}|_{X_{\eta}}))=C_{X_s/s}(\mathcal{F}|_{X_s})\quad{\rm in}\quad H^0(X_s, \mathcal{K}_{X_{s}/{s}}).
	\end{align}
\end{proposition}
\begin{remark}Consider  the diagram \eqref{spFGH} with $g$ smooth.
Let $\mathcal F\in D_{\rm ctf}(X_\eta,\Lambda)$ such that $f_\eta$ (resp. $f_s$) is universally locally acyclic relatively to $\mathcal F$ (resp. $\Psi_{h}(\mathcal F)$). Then one may prove the following formula
\begin{align}
	{\rm sp}(C_{X_{\eta}/Y_\eta}(\mathcal{F}))=C_{X_s/Y_s}(\Psi_{h}(\mathcal{F}))\quad{\rm in}\quad H^0(X_s, \mathcal{K}_{X_{s}/Y_{s}}).
\end{align}
\end{remark}

\section{Non-acyclicity class}\label{sec:CCC2}
\subsection{}\label{subsec:notationforXYS}
In this section, we construct the non-acyclicity classes supported on the non-acyclicity locus.
Consider a commutative diagram in ${\rm Sch}_S$:
\begin{align}\label{xyd:prop:Gysim:fgh2}
\begin{gathered}
\xymatrix{
Z\ar@{^(->}[r]^-\tau&X\ar[rr]^-f\ar[rd]_-h&&Y,\ar[ld]^-g\\
&&S&
}
\end{gathered}
\end{align}
where $\tau: Z\to X$ is a closed immersion. 
Let $\delta:Y\to Y\times_SY$ be the diagonal morphism.
Consider the following conditions:
\begin{itemize}
\item[(C1)] $\mathcal K_{Y/S}=g^!\Lambda$ is locally constant and that $\mathcal K_{Y/S}\otimes^L \delta^!\Lambda\simeq \Lambda$ (cf. \eqref{eq:KYSLambda}). 
For example, this condition  holds if $g$ is a smooth morphism.
\item[(C2)] The closed subscheme $Z\subseteq X$ satisfies $H^0(Z,\mathcal K_{Z/{Y}})=0$ and $H^1(Z,\mathcal K_{Z/{Y}})=0$.
\item[(C3)]Let $\mathcal F\in D_{\rm ctf}(X,\Lambda)$ such that 
$X\setminus Z\to Y$ is universally locally acyclic  relatively to $\mathcal F|_{X\setminus Z}$ and that $h:X\to S$  is universally locally acyclic  relatively to $\mathcal F$.
\end{itemize}
Let us clarify our strategy to define the non-acyclicity classes in the following:
\begin{itemize}
\item Using (C1), we  have   a distinguished triangle
$\mathcal K_{X/{Y}}\to \mathcal K_{X/S}\to \mathcal K_{X/Y/S}\xrightarrow{+1}$ (cf. \eqref{eq:distriK}).
\item Under (C1) and (C3), we construct the non-acyclicity class  $\widetilde{C}_{X/Y/S}^{Z}(\mathcal F)\in H^0_Z(X,\mathcal K_{X/Y/{S}})$.

\item When (C2) holds, we  show that
$ H^0(Z,\mathcal K_{Z/{S}})
=H^0_Z(X,\mathcal K_{X/{S}})\to H^0_Z(X,\mathcal K_{X/Y/{S}})$ is an isomorphism. Thus under the three  conditions (C1)-(C3), the class $\widetilde{C}_{X/Y/S}^{Z}(\mathcal F)$ defines the cohomology class $C_{X/Y/S}^{Z}(\mathcal F)\in H^0_Z(X,\mathcal K_{X/S})$. 
\end{itemize}
In the rest of this section, we always assume that the conditions (C1) and (C3) are satisfied.
\subsection{}
Let $i: X\times_Y X\to X\times_S X$ be the base change of the diagonal morphism $\delta:Y\to Y\times_SY$:
\begin{align}\label{eq:propchangebasehomo:1}
	\begin{gathered}
		\xymatrix{
			X\ar@/_2pc/[dd]_-{f}\ar@{_(->}[d]_-{\delta_1}\ar@{}[rd]|\Box\ar@{=}[r]&X\ar@{_(->}[d]^-{\delta_0}\\
			X\times_Y X\ar[r]^-{i}\ar[d]_-p \ar[r] & X\times_S X  \ar[d]^-{f\times f} \\
			Y\ar[r]^-\delta&Y\times_SY,\ar@{}|\Box[ul]
		}
	\end{gathered}
\end{align}
where $ \delta_0$ and $\delta_1$ are the diagonal morphisms. 
Applying \eqref{eq:Dchf} to $\mathcal K_{X/S}$ and $\delta_{0\ast}\mathcal K_{X/S}$, we get an isomorphism between distinguished triangles
\begin{align}\label{eq:defDistrik1}
	\begin{gathered}
		\xymatrix{
			\mathcal K_{X/{S}}\otimes^L f^\ast\delta^!\Lambda\ar[r]& \mathcal K_{X/S}\ar[r]& \delta^\Delta\mathcal K_{X/S}\ar[r]^-{+1}&\\
			\delta_1^\ast (i^\ast\delta_{0\ast}\mathcal K_{X/{S}}\otimes p^\ast \delta^!\Lambda)\ar[r]\ar[u]_-{\simeq}&
			\delta_1^\ast i^!\delta_{0\ast}\mathcal K_{X/{S}}\ar[r]\ar[u]_-{\simeq}&\delta_1^\ast \delta^\Delta\delta_{0\ast}\mathcal K_{X/{S}}\ar[r]^-{+1}\ar[u]_-{\simeq}&,	
		}
	\end{gathered}
\end{align}
where the commutativity follows from \eqref{eq:pullbackCF}.
We put
\begin{align}\label{eq:defineKXYS}
	\mathcal K_{X/Y/S}:=\delta^\Delta\mathcal K_{X/S}\simeq \delta_1^\ast \delta^\Delta\delta_{0\ast}\mathcal K_{X/{S}}.
\end{align}
By Proposition \ref{prop:Gysim}.(2), we have an isomorphism
\begin{align}\label{eq:defDistrik2}
	\begin{split}
		\mathcal K_{X/Y}&\xleftarrow{\simeq}\mathcal K_{X/S}\otimes^L f^\ast \delta^!\Lambda.
	\end{split}
\end{align}
Then we  rewrite the first row of \eqref{eq:defDistrik1} as the following distinguished triangle
\begin{align}\label{eq:distriK}
	\mathcal K_{X/{Y}}\to \mathcal K_{X/S}\to \mathcal K_{X/Y/S}\xrightarrow{+1}.
\end{align}
It induces an exact sequence
\begin{align}
	H_Z^0(X,\mathcal K_{X/{Y}})\to H_Z^0(X,\mathcal K_{X/{S}})
	\to H^0_Z(X,\mathcal K_{X/Y/S})\to H^1_Z(X,\mathcal K_{X/{Y}}).
\end{align}
If  the condition (C2) in \ref{subsec:notationforXYS} holds,  then the  map
\begin{align}\label{eq:isomOnHZassC2}
	H^0(Z,\mathcal K_{Z/{S}})=H_Z^0(X,\mathcal K_{X/{S}})
	\xrightarrow{\simeq} H^0_Z(X,\mathcal K_{X/Y/S})
\end{align}
is an isomorphism.
\begin{remark}\label{remark:expFordeltaDelta}
	By \eqref{eq:DefineDeltaDelta}, $\mathcal K_{X/Y/S}\simeq \mathcal K_{X/S}\otimes^L \delta_0^\ast (f\times f)^\ast j_{Y\ast}\Lambda$,
	where $j_Y: Y\times_SY\setminus \delta(Y)\to Y\times_SY$ is the open immersion.
\end{remark}

\subsection{}\label{subsec:discussDeltaSupport}
Since $i^\ast(\mathcal F\boxtimes_{S}^LD_{X/S}(\mathcal F))\otimes p^\ast \delta^!\Lambda\simeq (\mathcal F\boxtimes_{Y}^L D_{X/S}(\mathcal F))\otimes p^\ast \delta^!\Lambda\overset{\eqref{lem:eq:Gysim2}}{\simeq} \mathcal F\boxtimes^L_Y D_{X/Y}(\mathcal F)$, 
we have the following distinguished triangles by applying \eqref{eq:Dchf}
to $\mathcal F\boxtimes^L_S D_{X/S}(\mathcal F)$:
\begin{align}\label{eq:disOnBox}
\mathcal F\boxtimes^L_Y D_{X/Y}(\mathcal F) &\to i^!(\mathcal F\boxtimes_{S}^LD_{X/S}(\mathcal F))\to\delta^\Delta(\mathcal F\boxtimes_{S}^LD_{X/S}(\mathcal F))\xrightarrow{+1},\\
\label{eq:disOnBox2}\mathcal F\otimes^L D_{X/Y}(\mathcal F) &\to\delta_1^\ast i^!(\mathcal F\boxtimes_{S}^LD_{X/S}(\mathcal F))\to\delta_1^\ast \delta^\Delta(\mathcal F\boxtimes_{S}^LD_{X/S}(\mathcal F))\xrightarrow{+1}.
\end{align}
Since $X\setminus Z\to Y$ is universally locally acyclic  relatively to $\mathcal F$, the morphism $\mathcal F\otimes^L D_{X/Y}(\mathcal F) \to\delta_1^\ast i^!(\mathcal F\boxtimes_{S}^LD_{X/S}(\mathcal F))$ is an isomorphism on $X\setminus Z$ by Proposition \ref{prop:Gysim}.(3).
Thus $\delta_1^\ast \delta^\Delta(\mathcal F\boxtimes_{S}^LD_{X/{S}}(\mathcal F))$
is supported on $Z$.

\subsection{}\label{subsec:constructNAclass}
Recall that the relative characteristic class $C_{X/{S}}(\mathcal F)$
is defined by the composition (cf. \eqref{eq:traceEquiforcharclass})
\begin{align}\label{eq:defloc0}
\delta_{0\ast}\Lambda=\delta_{0!}\Lambda\to \mathcal F\boxtimes_{S}^LD_{X/S}(\mathcal F)\to \delta_{0\ast}\mathcal K_{X/{S}}.
\end{align}
It induces the following morphisms
\begin{align}\label{eq:defloc1}
\Lambda=\delta_1^\ast\delta_{1\ast}\Lambda\overset{\rm b.c}{\simeq}\delta_1^\ast i^!\delta_{0\ast}\Lambda\xrightarrow{\eqref{eq:Dchf}}\delta_1^\ast \delta^\Delta \delta_{0\ast}\Lambda\to \delta_1^\ast \delta^\Delta(\mathcal F\boxtimes_{S}^LD_{X/S}(\mathcal F))\to \delta_1^\ast \delta^\Delta\delta_{0\ast}\mathcal K_{X/{S}}.
\end{align}

Note that $\tau: Z\hookrightarrow X$ denotes the closed immersion.
Since the complex $\delta_1^\ast \delta^\Delta(\mathcal F\boxtimes_{S}^LD_{X/S}(\mathcal F))\simeq \tau_\ast\tau^\ast \delta_1^\ast \delta^\Delta(\mathcal F\boxtimes_{S}^LD_{X/S}(\mathcal F))$ is supported on $Z$, 
hence \eqref{eq:defloc1} defines a cohomology class 
\begin{align}
	\tau^\ast\Lambda\to \tau^\ast \delta_1^\ast \delta^\Delta(\mathcal F\boxtimes_{S}^LD_{X/S}(\mathcal F))\to \tau^!(\delta_1^\ast \delta^\Delta\delta_{0\ast}\mathcal K_{X/{S}}),
\end{align}
which we denote by
\begin{align}\label{eq:defEqOfCZXYS}
	\widetilde{C}_{X/Y/S}^{Z}(\mathcal F)\in H_Z^0(X,\delta_1^\ast \delta^\Delta\delta_{0\ast}\mathcal K_{X/{S}})=H_Z^0(X,\mathcal K_{X/Y/S}).
\end{align}
If  the condition (C2) in \ref{subsec:notationforXYS} holds,  then 
$ H^0_Z(X,\mathcal K_{X/{S}})
\xrightarrow{\eqref{eq:isomOnHZassC2}} H^0_Z(X,\mathcal K_{X/Y/{S}})$ is an isomorphism. In this case, the class 
$\widetilde{C}_{X/Y/S}^{Z}(\mathcal F)\in H_Z^0(X,\mathcal K_{X/Y/{S}})$
defines an element of $H^0_Z(X,\mathcal K_{X/{S}})$, which is  denoted by
$C_{X/Y/S}^{Z}(\mathcal F)$.
\begin{definition}\label{def:milclass}
Under the notation and conditions (C1) and (C3) in \ref{subsec:notationforXYS}.
We define  the 
{\it non-acyclicity class} of $\mathcal F$ to be the class
$\widetilde{C}_{X/Y/S}^{Z}(\mathcal F)\in H_Z^0(X,\mathcal K_{X/Y/{S}})$.
If moreover the condition (C2) holds, we also call
$C_{X/Y/S}^{Z}(\mathcal F)=\widetilde{C}_{X/Y/S}^{Z}(\mathcal F)\in H^0_Z(X,\mathcal K_{X/S})$ the {\it non-acyclicity class} of $\mathcal F$.
\end{definition}

\begin{remark}
	Under the conditions (C1) and (C3) in \ref{subsec:notationforXYS}.
	If $Z'\subseteq X$ is another closed subscheme such that $Z\subseteq Z'$, then $\widetilde{C}^{Z'}_{X/Y/S}(\mathcal F)$ equals the image of
	$\widetilde{C}^Z_{X/Y/S}(\mathcal F)$ under the canonical map $H_Z^0(X,\mathcal K_{X/Y/{S}})\to H_{Z'}^0(X,\mathcal K_{X/Y/{S}})$.
\end{remark}

\subsection{}\label{subsec:PBPX}
Now we turn to study functorial properties of the non-acyclicity classes. We first consider the pull-back property.
Consider the notation in \ref{subsec:notationforXYS} and assume the conditions (C1) and (C3) hold.
Let $b: S'\to S$ be a morphism of Noetherian schemes. Let $U=X\setminus Z$. We denote by $X':=X\times_SS'$ the base change of $X$. Similar to define $\mathcal F'$, $Y'$, $Z'$, $U'$, $f'$, $g'$, $h'$ and $\tau':Z'\to X'$. Let $b_X$ and $b_Z$ be the base change of $b$ by $X\to S$ and $Z\to S$ respectively.
We form the following commutative diagrams
\begin{align}\label{bcdiag12}\small
	\begin{gathered}
		\xymatrix@C=1.7em@R=1em{
			&&&&X'\ar[rr]|(0.3){b_X}\ar@{_(->}[dd]|(0.3){\delta'_1}\ar@{=}[dr]&&X\ar@{_(-->}[dd]|(0.3){\delta_1}\ar@{=}[dr]\\
			Z'\ar[rr]|{b_Z}\ar@{_(->}[d]_-{\tau'}&&Z\ar@{_(->}[d]_-{\tau}&&&X'\ar[rr]|(0.3){b_X}\ar[dd]|(0.3){\delta'_0}&&X\ar[dd]|(0.3){\delta_0}\\
			X'\ar[rr]|{b_X}\ar[dd]_(0.3){h'}\ar[dr]|{f'}&&X\ar[dd]_(.3)h\ar[dr]|f&&X'\times_{Y'}X'\ar@{-->}[rr]\ar[dr]|{i'}\ar[dd]|(0.3){p'}&&X\times_YX\ar@{-->}[dr]|{i}\ar@{-->}[dd]|(0.3){p}&\\
			&Y'\ar[rr]|(0.3){b_Y}\ar[ld]|{g'}&&Y\ar[ld]|{g}&&X'\times_{S'}X'\ar[rr]\ar[dd]|(0.3){f'\times f'}&&X\times_SX\ar[dd]|(0.3){f\times f}\\
			S'\ar[rr]|b&&S&&Y'\ar@{-->}[rr]|(0.3){b_Y}\ar[rd]|(0.3){\delta'}&&Y\ar@{-->}[rd]|(0.3){\delta}&\\
			&&&&&Y'\times_{S'}Y'\ar[rr]&&Y\times_SY,
		}
	\end{gathered}
\end{align}
where squares and parallelograms are cartesian diagrams (cf. \eqref{eq:propchangebasehomo:1}).
By Remark \ref{remark:expFordeltaDelta}, the base change morphism \eqref{eq:3pullbackK} induces a canonical morphism
\begin{align}\label{eq:BCforKXYS}
	b_{X}^\ast\mathcal K_{X/Y/S}\to \mathcal K_{X'/Y'/S'}.
\end{align}
Therefore we have a commutative diagram between distinguished triangles (cf. \eqref{eq:distriK}): 
\begin{align}\label{eq:pf:pullbackKDeltaForNTclass5}
	\begin{gathered}
		\xymatrix@R=0.6cm{
			b_{X}^\ast\mathcal K_{X/{Y}}\ar[r]\ar[d]_-{{\eqref{eq:3pullbackK}}}& b_{X}^\ast\mathcal K_{X/S}\ar[r]\ar[d]_-{{\eqref{eq:3pullbackK}}}& b_{X}^\ast\mathcal K_{X/Y/S}\ar[r]^-{+1}\ar[d]^-{\eqref{eq:BCforKXYS}}&\\
			\mathcal K_{X'/{Y'}}\ar[r]& \mathcal K_{X'/S'}\ar[r]& \mathcal K_{X'/Y'/S'}\ar[r]^-{+1}&.
		}
	\end{gathered}
\end{align}
\subsection{}\label{sub:assOnBC}
In the following, we always assume that the following condition $(C'1)$ holds:
\begin{itemize}
	\item[($C'1$)] $\mathcal K_{Y'/S'}=g'^!\Lambda$ is locally constant and that $\mathcal K_{Y'/S'}\otimes^L \delta'^!\Lambda\simeq \Lambda$.
\end{itemize}
Then the class $\widetilde{C}_{X'/Y'/S'}^{Z'}(\mathcal F')$ is well-defined. 
If moreover the following condition $(C'2)$  holds, then the class $C_{X'/Y'/S'}^{Z'}(\mathcal F')$ is well-defined.
\begin{itemize}
	\item[($C'2$)]  The closed subscheme $Z'\subseteq X'$ satisfies $H^0(Z',\mathcal K_{Z'/{Y'}})=0$ and $H^1(Z',\mathcal K_{Z'/{Y'}})=0$.
\end{itemize}

\subsection{} 
Applying $\delta_1^\ast i^\ast (-)\otimes^L f^\ast \delta^!\Lambda\to \delta_1^\ast i^!(-)\to \delta_1^\ast \delta^\Delta(-)\xrightarrow{+1}$ to \eqref{eq:defloc0}, we get a commutative diagram 
\begin{align}\label{bcdiag3}
	\begin{gathered}
		\xymatrix@R=0.6cm{
			f^\ast\delta^!\Lambda \ar[r]\ar[d]&\delta_1^\ast(\mathcal F\boxtimes^L_YD_{X/Y}(\mathcal F))\ar[d]\ar[r]&\mathcal K_{X/Y}\ar[d]\\
			\Lambda\ar[r]\ar[d]&\delta_1^\ast i^!(\mathcal F\boxtimes^L_SD_{X/S}(\mathcal F))\ar[d]\ar[r]\ar@{}[rd]|{(a)}&\mathcal K_{X/S}\ar[d]\\
			\delta_1^\ast \delta^\Delta \delta_{0\ast}\Lambda\ar[r]\ar[d]^-{+1}& \delta_1^\ast \delta^\Delta(\mathcal F\boxtimes_{S}^LD_{X/S}(\mathcal F))\ar[r]\ar[d]^-{+1}& \mathcal K_{X/Y/{S}},\ar[d]^-{+1}\\
			&&&
		}
	\end{gathered}
\end{align}
where columns are distinguished triangles. 
Note that the diagram $(a)$ in \eqref{bcdiag3} admits a factorization
\begin{align}\label{bcdiag3-factorization}
	\begin{gathered}
		\xymatrix@R=0.6cm{
			\delta_1^\ast i^!(\mathcal F\boxtimes^L_SD_{X/S}(\mathcal F))\ar[r]^-{\rm b.c}\ar[d]&\delta_0^\ast (\mathcal F\boxtimes^L_SD_{X/S}(\mathcal F))\ar[r]^-{v_{\mathcal F}}\ar[d]&\mathcal K_{X/S}\ar[d]\\
			\delta_1^\ast \delta^\Delta(\mathcal F\boxtimes_{S}^LD_{X/S}(\mathcal F))\ar[r]^-{\rm b.c}&\delta^\Delta \delta_0^\ast (\mathcal F\boxtimes^L_SD_{X/S}(\mathcal F))\ar[r]^-{v_{\mathcal F}}& \delta^\Delta\mathcal K_{X/S}\ar@{=}[r]&\mathcal K_{X/Y/{S}}.\\
		}
	\end{gathered}
\end{align}
We have a similar diagram after taking base change. 
By Proposition \ref{cor:pullbackccc}, \eqref{bcdiag3-factorization} and \eqref{eq:pf:pullbackKDeltaForNTclass5}, there is a commutative diagram 
\begin{align}\label{bcdiag4}\small
	\begin{gathered}
		\xymatrix@R=0.6cm@C=0.2em{
			&\Lambda\ar[rr]&&\delta_1^{\prime\ast} i^{\prime!}(\mathcal F'\boxtimes^L_{S'}D_{X'/S'}(\mathcal F'))\ar[dd]\ar[rr]&&\mathcal K_{X'/S'}\ar[dd]\\
			\qquad b_{X}^\ast\Lambda=\Lambda \ar[rr]\ar[ur]&&b_{X}^\ast\delta_1^\ast i^!(\mathcal F\boxtimes^L_SD_{X/S}(\mathcal F))\ar[dd]\ar[rr]&&b_{X}^\ast\mathcal K_{X/S}\ar[dd]\ar[ur]&\\
			&&& \delta_1^{\prime\ast} \delta^{\prime\Delta}(\mathcal F'\boxtimes^L_{S'}D_{X'/S'}(\mathcal F'))\ar[rr]&& \mathcal K_{X'/Y'/{S'}}\\
			&& b_{X}^\ast\delta_1^\ast \delta^\Delta(\mathcal F\boxtimes^L_SD_{X/S}(\mathcal F))\ar[rr]&& b_{X}^\ast\mathcal K_{X/Y/{S}}.\ar[ur]&\\
		}
	\end{gathered}
\end{align}
Consider the pull-back morphisms 
\[b_{X}^\ast: H^{0}_{Z}(X,\mathcal K_{X/Y/S})\to H^{0}_{Z'}(X',\mathcal K_{X'/Y'/S'})\quad{\rm and}\quad b_{X}^\ast: H^{0}_{Z}(X,\mathcal K_{X/S})\to H^{0}_{Z'}(X',\mathcal K_{X'/S'}). 
\]
The above commutative diagram implies  the following result:
\begin{proposition}\label{prop:LocBC}
	Under the notation in \ref{subsec:PBPX}, and the condition $(C'1)$ in \ref{sub:assOnBC}, we have
	\begin{align}\label{eq:prop:LocBC1-1}
		b_{X}^\ast \widetilde{C}_{X/Y/S}^{Z}(\mathcal F)=\widetilde{C}_{X'/Y'/S'}^{Z'}(\mathcal F')\qquad {\rm in}\qquad H^{0}_{Z'}(X',\mathcal K_{X'/Y'/S'}).
	\end{align}
	If moreover the conditions (C2) and $(C'2)$ hold, then we have
	\begin{align}\label{eq:prop:LocBC1}
		b_{X}^\ast C_{X/Y/S}^{Z}(\mathcal F)=C_{X'/Y'/S'}^{Z'}(\mathcal F')\qquad {\rm in}\qquad H^{0}_{Z'}(X',\mathcal K_{X'/S'}).
	\end{align}
\end{proposition}

\subsection{}\label{sub:notationForpushforwadNT}
Now we study the proper push-forward property of  non-acyclicity classes.
We consider a morphism between two such diagrams \eqref{xyd:prop:Gysim:fgh2}, which is depicted 
as the following left commutative diagram in ${\rm Sch}_S$:
\begin{align}\label{eq:ppForNTclass1}\small
	\begin{gathered}
		\xymatrix@C=2em{
			&&&&&X'\ar@{=}[rr]\ar@{_(->}[dd]|(0.3){\delta_1^{\prime}}&&X'\ar@{_(->}[dd]|(0.3){\delta_0^{\prime}}\\
			Z\ar@/^2pc/[rrr]^-{s_Z}\ar@{^(->}[r]^-{\tau_0}\ar@{_(->}[rd]_-\tau&s^{-1}(Z')\ar@{_(->}[d]_-{\tau_1}\ar[rr]^-{s}\ar@{}[rrd]|\Box&&Z'\ar@{_(->}[d]_-{\tau'}&X\ar@/_2pc/[ddd]_-f\ar@{=}[rr]\ar[ur]^-s\ar[dd]|(0.3){\delta_1}&&X\ar[ur]_-s\ar[dd]|(0.3){\delta_0}\\
			&X\ar[rr]^-s\ar[rd]^f\ar[rdd]_-h&&X'\ar[ld]_{f'}\ar[ldd]^-{h'}&&X'\times_YX'\ar[ddl]^(0.2){p'}\ar[rr]_(0.3){i'}&&X'\times_SX'\ar[ddl]^(0.3){p}\\
			&&Y\ar[d]^-(0.2)g&&X\times_YX\ar[ur]|{s\times s}\ar[d]|p\ar[rr]_-i&&X\times_SX\ar[d]|{f\times f}\ar[ur]|{s\times s}\\
			&&S&&Y\ar[rr]^-\delta&&Y\times_SY,
		}
	\end{gathered}
\end{align}
where $\tau$, $\tau'$, $\tau_0$ and $\tau_1$ are closed immersions.
Consider the notation in the diagram \eqref{eq:propchangebasehomo:1}. We put the prime symbol $'$ on the corresponding diagram for $Z'/X'/Y/S$. Now
the left diagram \eqref{eq:ppForNTclass1} induces the right one.
Assume that the conditions (C1) and (C3) in \ref{subsec:notationforXYS} hold and that $s:X\to X'$ is proper.
Then the morphism $X'\setminus Z'\to Y$ is also universally locally acyclic relatively to $s_\ast\mathcal F|_{X'\setminus Z'}$ by \cite[Theorem 2.16 and Proposition 2.23]{LZ22}. Now the non-acyclicity classes 
\begin{align}
	\widetilde{C}^{Z}_{X/Y/S}(\mathcal F)\in H_Z^0(X,\mathcal K_{X/Y/S})\quad{\rm and}\quad \widetilde{C}^{Z'}_{X'/Y/S}(s_\ast\mathcal F)\in H_{Z'}^0(X',\mathcal K_{X'/Y/S})
\end{align} 
are well-defined.

\subsection{} We define proper push-forward maps
\begin{align}\label{eq:ppForNTclass2}
	&s_\ast: H^0_Z(X,\mathcal K_{X/S})\to H^0_{Z'}(X',\mathcal K_{X'/S}),\\
	\label{eq:ppForNTclass2-2}
	&s_\ast:  H^0_Z(X,\mathcal K_{X/Y/S})\to H^0_{Z'}(X',\mathcal K_{X'/Y/S}). 
\end{align}
First, we have a morphism
\begin{align}\label{eq:ppForNTclass3}
	s_\ast\mathcal K_{X/Y}=s_!\mathcal K_{X/Y}\to \mathcal K_{X'/Y}=f^{\prime!}\Lambda,
\end{align}
which is defined to be the adjunction of the  composition
$f'_!s_!f^!\Lambda\simeq f_!f^!\Lambda\xrightarrow{{\rm adj}} \Lambda$.
Similarly, there is a canonical morphism 
\begin{align}\label{eq:ppForNTclass4}
	s_\ast\mathcal K_{X/S}\to \mathcal K_{X'/S}.
\end{align}
Therefore the composition
\begin{align}\label{eq:ppForNTclass5}
	\tau^\prime_!s_{Z\ast}\tau^!\mathcal K_{X/S}\simeq s_\ast\tau_!\tau^!\mathcal K_{X/S}\xrightarrow{\rm adj}s_\ast\mathcal K_{X/S}\xrightarrow{\eqref{eq:ppForNTclass4}} \mathcal K_{X'/S}
\end{align}
induces a morphism 
$s_{Z\ast}\tau^!\mathcal K_{X/S}\to \tau^{\prime!}\mathcal K_{X'/S}$,
which gives the required morphism \eqref{eq:ppForNTclass2}. By Remark \ref{remark:expFordeltaDelta}, the morphism \eqref{eq:ppForNTclass3} induces a canonical morphism
\begin{align}\label{eq:ppForNTclass6-7}
	s_{\ast}\mathcal K_{X/Y/S}\to \mathcal K_{X'/Y/S}.
\end{align}
It further induces the required morphism \eqref{eq:ppForNTclass2-2}.
Consider the distinguished triangle \eqref{eq:distriK}. We obtain a commutative diagram
\begin{align}\label{eq:pf:ppForNTclass5}
	\begin{gathered}
		\xymatrix@R=0.6cm{
			s_\ast\mathcal K_{X/{Y}}\ar[r]\ar[d]_-{{\eqref{eq:ppForNTclass3}}}& s_\ast\mathcal K_{X/S}\ar[r]\ar[d]_-{{\eqref{eq:ppForNTclass4}}}& s_\ast \mathcal K_{X/Y/S}\ar[r]^-{+1}\ar[d]^-{\eqref{eq:ppForNTclass6-7}}&\\
			\mathcal K_{X'/{Y}}\ar[r]& \mathcal K_{X'/S}\ar[r]& \mathcal K_{X'/Y/S}\ar[r]^-{+1}&.
		}
	\end{gathered}
\end{align}
\begin{proposition}\label{thm:ppForNTclass}
	Under the notation and assumptions in \ref{sub:notationForpushforwadNT},
	we have
	\begin{align}\label{eq:thm:ppForNTclass3}
		s_\ast (\widetilde{C}^{Z}_{X/Y/S}(\mathcal F))=\widetilde{C}^{Z'}_{X'/Y/S}(s_\ast\mathcal F)\quad{\rm in}\quad H^0_{Z'}(X',\mathcal K_{X'/Y/S}).
	\end{align}
	If moreover $H_{Z'}^0(X',\mathcal K_{X'/{Y'}})=H_{Z'}^1(X',\mathcal K_{X'/{Y'}})=0$, then   
	\begin{align}\label{eq:thm:ppForNTclass3-3}
		s_\ast (\widetilde{C}^{Z}_{X/Y/S}(\mathcal F))=C^{Z'}_{X'/Y/S}(s_\ast\mathcal F)\quad{\rm in}\quad H^0_{Z'}(X',\mathcal K_{X'/S}).
	\end{align}
\end{proposition}
\begin{proof}
	In the following, we put $\mathcal F'=s_\ast\mathcal F$.
	Consider the diagrams \eqref{bcdiag3}, \eqref{bcdiag3-factorization} and similar diagrams 
	for $X'/Y/S$.
	By Proposition \ref{prop:pushccc}, \eqref{bcdiag3-factorization} and \eqref{eq:pf:ppForNTclass5}, there is a commutative diagram 
	\begin{align}\label{eq:pf:ppForNTclass6}\small
		\begin{gathered}
			\xymatrix@R=0.4cm@C=1em{
				&\Lambda\ar[rr]\ar[dl]&&\delta_1^{\prime\ast} i^{\prime!}(\mathcal F'\boxtimes^L_{S}D_{X'/S}(\mathcal F'))\ar[dd]\ar[rr]&&\mathcal K_{X'/S}\ar[dd]\\
				 s_\ast\Lambda \ar[rr]&&s_\ast\delta_1^\ast i^!(\mathcal F\boxtimes^L_SD_{X/S}(\mathcal F))\ar[dd]\ar[rr]&&s_\ast\mathcal K_{X/S}\ar[dd]\ar[ur]&\\
				&&& \delta_1^{\prime\ast} \delta^{\Delta}(\mathcal F'\boxtimes^L_{S}D_{X'/S}(\mathcal F'))\ar[rr]&& \mathcal K_{X'/Y/{S}}\\
				&& s_\ast\delta_1^\ast \delta^\Delta(\mathcal F\boxtimes^L_SD_{X/S}(\mathcal F))\ar[rr]&& s_\ast \mathcal K_{X/Y/{S}}.\ar[ur]&\\
			}
		\end{gathered}
	\end{align}
	This proves the equality \eqref{eq:thm:ppForNTclass3}.
\end{proof}

The construction of non-acyclicity classes is \'etale local in the following sense.
\begin{proposition}\label{prop:invariantUnderEtale}
Consider a commutative  diagram in ${\rm Sch}_S$
	\begin{align}\label{eq:prop:invariantUnderEtale:1}
		\begin{gathered}
			\xymatrix{
				Z'\ar@{}[rd]|{\Box}\ar@{^(->}[r]^-\tau\ar[d]&X'\ar[rdd]^(0.3){h'}\ar[rr]^-{f'}\ar[d]_-{e_1}&&Y'\ar[d]^-{e_0}\ar[ldd]_-(0.3){g'}\\
				Z\ar@{^(->}[r]^-\tau&X\ar[rr]^-f\ar[rd]_-h&&Y,\ar[ld]^-{g}\\
				&&S&
			}
		\end{gathered}
	\end{align}
	where $e_1$ and $e_0$ are \'etale morphisms, $Z'=e_1^{-1}(Z)$, $h'=he_1$ and $g'=ge_0$.
	Let $\mathcal F\in D_{\rm ctf}(X,\Lambda)$ and $\mathcal F'=e_1^\ast\mathcal F$.
	Assume that the conditions (C1) and (C3)
	in \ref{subsec:notationforXYS} hold for  $g$ and $\mathcal F$ respectively.
	Then we have 
	\begin{align}\label{eq:prop:invariantUnderEtale:0-0}
		  e_1^\ast\widetilde{C}^{Z}_{X/Y/S}(\mathcal F)=\widetilde{C}^{Z'}_{X'/Y'/S}(\mathcal F')=\widetilde{C}^{Z'}_{X'/Y/S}(\mathcal F')\;{\rm in}\; H^0_{Z'}(X',\mathcal K_{X'/Y'/S})=H^0_{Z'}(X',\mathcal K_{X'/Y/S}),
	\end{align}
where $e_1^\ast:H^0_Z(X,\mathcal K_{X/Y/S})\to H^0_{Z'}(X',\mathcal K_{X'/Y'/S})=H^0_{Z'}(X',\mathcal K_{X'/Y/S})$ is the pull-back morphism.
	If moreover the condition (C2)  in \ref{subsec:notationforXYS} holds for $Z/X/Y/S$ and $Z'/X'/Y/S$, then we have
	\begin{align}\label{eq:prop:invariantUnderEtale:0}
		e_1^\ast C^{Z}_{X/Y/S}(\mathcal F)=C^{Z'}_{X'/Y'/S}(\mathcal F')=C^{Z'}_{X'/Y/S}(\mathcal F')\quad{\rm in}\quad H^0_{Z'}(X',\mathcal K_{X'/S}),
	\end{align}
where $e_1^\ast:H^0_Z(X,\mathcal K_{X/S})\to H^0_{Z'}(X',\mathcal K_{X'/S})$ is the pull-back morphism.
\end{proposition}
\begin{proof}
It is clear that the condition (C1) in \ref{subsec:notationforXYS} holds for $g'$.
Since $e_1$ is \'etale, the morphism $h'$ (resp. $X'\setminus Z'\to Y$) is universally locally acyclic relatively to $\mathcal F'$ (resp. $\mathcal F'|_{X'\setminus Z'}$).
Since $e_0$ is \'etale, the morphism $X'\setminus Z'\to Y'$ is also universally locally acyclic relatively to
$\mathcal F'|_{X'\setminus Z'}$.
So the condition (C3)
in \ref{subsec:notationforXYS} holds for $X'/Y'/S$ and $X'/Y/S$. Hence $\widetilde{C}^{Z'}_{X'/Y'/S}(\mathcal F')$ and $\widetilde{C}^{Z'}_{X'/Y/S}(\mathcal F')$ are well-defined.
Now we show $\widetilde{C}^{Z'}_{X'/Y'/S}(\mathcal F')=\widetilde{C}^{Z'}_{X'/Y/S}(\mathcal F')$. Consider the following commutative diagram

\begin{equation}
\begin{tikzcd}[row sep=2.5em, column sep=small]
X' \arrow[rr,"\delta_0'"] \arrow[rd,equal] \arrow[dd,equal] &&X'\times_{S}X'\arrow[rr, "f'\times f'"] && Y'\times_{S} Y' \arrow[rr, "e_0\times e_0"] &&Y\times_{S} Y\\
&X' \arrow[rr,"\delta_1''" near start]\arrow[dl,equal]&&X'\times_Y X' \arrow[rr,"p''" near start]\arrow[ul,"i''"] && Y'\times_YY' \arrow[ul,"\gamma_1"]\ar[rr] &&Y\arrow[ul,"\delta"]\\
X' \arrow[rr,"\delta_1'"] &&X'\times_{Y'}X' \arrow[rr,"p'" ] \arrow[ur,hookrightarrow,"\eta'"]\arrow[uu,crossing over,"i'" near start]&&Y', \arrow[rrru,"e_0"'] \arrow[uu, crossing over, "\delta'" near start]\arrow[ur,hookrightarrow, "\gamma_0"]
\end{tikzcd}
\end{equation}
where squares and parallelograms are cartesian. Since $e_0:Y'\to Y$ is \'etale, the diagonal morphism $\gamma_0: Y'\to Y'\times_{Y}Y'$ is an open immersion. So is its base change $\eta': X'\times_{Y'} X'\to X'\times_{Y}X'$.
For $X/Y/S$, we have the diagram \eqref{bcdiag3}. Similar for $X'/Y'/S$ and $X'/Y/S$.
Since $\gamma_0$ is an open immersion and that $e_0\times e_0$ is \'etale, 
we have isomorphisms ${(f'e_0)}^\ast {\delta}^{!}\Lambda\simeq f^{\prime\ast} e_0^\ast {\delta}^{!}\Lambda\simeq f^{\prime\ast} \gamma_0^\ast {\gamma}_1^{!}\Lambda\simeq f^{\prime\ast}\delta^{\prime !}\Lambda$ and 
$\delta_1^{\prime\ast}(\mathcal F'\boxtimes_{Y'}^L D_{X'/Y'}(\mathcal F'))\simeq \delta_1^{\prime\prime\ast}(\mathcal F'\boxtimes_{Y}^L D_{X'/Y}(\mathcal F'))$. 
Similarly, one can check that all objects in the diagram \eqref{bcdiag3} for 
$X'/Y'/S$ are isomorphic to those for $X'/Y/S$. This proves $\widetilde{C}^{Z'}_{X'/Y'/S}(\mathcal F')=\widetilde{C}^{Z'}_{X'/Y/S}(\mathcal F')$.

Consider the diagram \eqref{bcdiag3}. Since its pull-back by $e_1^\ast$ is isomorphic to the diagram for $X'/Y/S$, therefore $e_1^\ast\widetilde{C}^{Z}_{X/Y/S}(\mathcal F)=\widetilde{C}^{Z'}_{X'/Y/S}(\mathcal F')$.

Now assume the condition (C2)  in \ref{subsec:notationforXYS} holds for $Z'/X'/Y/S$.
Since $\mathcal K_{Z'/Y'}=\mathcal K_{Z'/Y}$,  the condition (C2)
also holds for $Z'/X'/Y'/S$ and then $C^{Z'}_{X'/Y'/S}(\mathcal F')$ is well-defined. 
Now \eqref{eq:prop:invariantUnderEtale:0} follows from 
\eqref{eq:prop:invariantUnderEtale:0-0}.
\end{proof}

\subsection{}\label{subsec:spOfNAclass}
At last, we show that the non-acyclicity class is compatible with the specialization map.
Consider the commutative diagram \eqref{xyd:prop:Gysim:fgh2}.
Assume that $g$ is a smooth morphism and that $S$ is a henselian trait. Let $\eta$ be the  generic point of $S$ and $s$ a geometric point of $S$ above its closed point. 
We first define a canonical morphism
\begin{align}\label{eq:PsiKXYS}
	\Psi_{h}(\mathcal{K}_{X_\eta/Y_\eta/\eta})\to \mathcal{K}_{X_s/Y_s/s}.
\end{align}
Consider the commutative diagram \eqref{eq:propchangebasehomo:1} and similar diagrams over $\eta$ and $s$:
\begin{align}\label{eq:PsiKXYS-0}
	\begin{gathered}
		\xymatrix{
			X_\eta\ar@{_(->}[d]_-{\delta_{\eta1}}\ar@{}[rd]|\Box\ar@{=}[r]\ar@/_2.5pc/[dd]_(0.3){f_\eta}&X_\eta\ar@{_(->}[d]^-{\delta_{\eta0}}&X_s\ar@/_2.5pc/[dd]_(0.3){f_s}\ar@{_(->}[d]_-{\delta_{s1}}\ar@{}[rd]|\Box\ar@{=}[r]&X_s\ar@{_(->}[d]^-{\delta_{s0}}\\
			X_\eta\times_{Y_\eta} X_{\eta}\ar[r]^-{i_\eta}\ar[d] \ar[r] & X_\eta\times_\eta X_\eta  \ar[d]^-{f_\eta\times f_\eta}& X_s\times_{Y_s} X_{s}\ar[r]^-{i_s}\ar[d] \ar[r] & X_s\times_s X_s  \ar[d]^-{f_s\times f_s} \\
			Y_\eta\ar[r]^-{\delta_\eta}&Y_\eta\times_\eta Y_\eta,\ar@{}|\Box[ul]&Y_s\ar[r]^-{\delta_s}&Y_s\times_s Y_s.\ar@{}|\Box[ul]
		}
	\end{gathered}
\end{align}
We define $j_Y: Y\times_SY\setminus \delta(Y)\to Y\times_SY$  to be the open immersion and similar for $j_{Y_\eta}$ and $j_{Y_s}$.
Note that $\delta^\ast\delta_\ast\delta^!\Lambda=\delta^!\Lambda$ and $\mathcal K_{Y/S}$ are locally constant on $Y$ since $g$ is smooth. 
Applying $c_{f,{\rm id},-,\mathcal K_{Y/S}}$ to the distinguished triangle $\delta^\ast\delta_\ast\delta^!\Lambda\to \delta^\ast\Lambda\to \delta^\ast j_{Y\ast}\Lambda\xrightarrow{+1}$, we get
a commutative diagram between distinguished triangles:
\begin{align}\label{eq:PsiKXYS-1}
\begin{gathered}
\xymatrix{
f^\ast\delta^\ast\delta_\ast\delta^!\Lambda\otimes^L f^!\mathcal K_{Y/S}\ar[d]_-\simeq^-{c_{f,{\rm id},\delta^\ast\delta_\ast\delta^!\Lambda,\mathcal K_{Y/S}}}\ar[r]&f^\ast\delta^\ast\Lambda\otimes^L f^!\mathcal K_{Y/S}\ar[r]\ar[d]_-\simeq^-{c_{f,{\rm id},\delta^\ast\Lambda,\mathcal K_{Y/S}}}&f^\ast \delta^\ast j_{Y\ast}\Lambda\otimes^L f^!\mathcal K_{Y/S}\ar[d]^-{c_{f,{\rm id},\delta^\ast j_{Y\ast}\Lambda,\mathcal K_{Y/S}}}\ar[r]^-{+1}&\\
f^!(\delta^\ast\delta_\ast\delta^!\Lambda \otimes^L \mathcal K_{Y/S})\ar[r]&f^!(\delta^\ast\Lambda \otimes^L \mathcal K_{Y/S})\ar[r]&f^!(\delta^\ast j_{Y\ast}\Lambda \otimes^L \mathcal K_{Y/S})\ar[r]^-{+1}&,
}
\end{gathered}
\end{align}
which induces an isomorphism
\begin{align}
\label{eq:PsiKXYS-2}
f^\ast\delta^\ast j_{Y\ast}\Lambda\otimes^Lf^!\mathcal K_{Y/S} \xrightarrow[\simeq]{c_{f,{\rm id},\delta^\ast j_{Y\ast}\Lambda,\mathcal K_{Y/S}}}f^!(\delta^\ast j_{Y\ast}\Lambda\otimes^L \mathcal K_{Y/S}).
\end{align}
We have similar diagrams for $\Psi_{g}$ and $c_{f_t,{\rm id},-,\mathcal K_{Y_t/t}}$ ($t\in\{\eta,s\}$), which induce the following isomorphisms
\begin{align}
	\label{eq:PsiKXYS-3}
	&\Psi_{g}(\mathcal K_{Y_\eta/\eta}\otimes^L\delta_\eta^\ast j_{Y_\eta\ast}\Lambda)\simeq \mathcal K_{Y_s/s}\otimes^L\delta_s^\ast j_{Y_s\ast}\Lambda,\\
	\label{eq:PsiKXYS-2-2}
	&f_t^\ast\delta_t^\ast j_{Y_t\ast}\Lambda\otimes^Lf_t^!\mathcal K_{Y_t/t} \xrightarrow[\simeq]{c_{f_t,{\rm id},\delta_t^\ast j_{Y_t\ast}\Lambda,\mathcal K_{Y_t/t}}}f_t^!(\delta_t^\ast j_{Y_t\ast}\Lambda\otimes^L \mathcal K_{Y_t/t}).
\end{align}
By \ref{remark:expFordeltaDelta} and \eqref{eq:PsiKXYS-2-2}, we get isomorphisms
\begin{align}\label{eq:PsiKXYS-5}
\begin{split}
	\mathcal K_{X_t/Y_t/t}\overset{\ref{remark:expFordeltaDelta}}{\simeq}
	\mathcal K_{X_t/t}\otimes^L \delta_{t0}^\ast (f_t\times f_t)^\ast j_{Y_t\ast}\Lambda&\simeq 
	f_t^!\mathcal K_{Y_t/t}\otimes^L f_t^\ast\delta_t^\ast j_{Y_t\ast}\Lambda\\
	&\xrightarrow[\simeq]{\eqref{eq:PsiKXYS-2-2}}f_t^!(\mathcal K_{Y_t/t}\otimes^L\delta_t^\ast j_{Y_t\ast}\Lambda).
\end{split}
\end{align}
Similar to the definition of \eqref{PsiKXYdef}, we define \eqref{eq:PsiKXYS} to be the following composition
\begin{align}\label{eq:PsiKXYS-4}
\begin{split}
\Psi_{h}(\mathcal K_{X_\eta/Y_\eta/\eta})&\overset{\eqref{eq:PsiKXYS-5}}{\simeq}\Psi_{h}(f_\eta^!(\mathcal K_{Y_\eta/\eta}\otimes^L\delta_\eta^\ast j_{Y_\eta\ast}\Lambda))\\
&\xrightarrow{\rm b.c}f_s^!\Psi_{g}(\mathcal K_{Y_\eta/\eta}\otimes^L\delta_\eta^\ast j_{Y_\eta\ast}\Lambda)
\overset{\eqref{eq:PsiKXYS-3}}{\simeq} f_s^!(\mathcal K_{Y_s/s}\otimes^L\delta_s^\ast j_{Y_s\ast}\Lambda) \overset{\eqref{eq:PsiKXYS-5}}{\simeq} \mathcal K_{X_s/Y_s/s}.
\end{split}
\end{align}
Applying $\Psi_{h}: D^+(X_\eta,\Lambda)\to D^+(X_s,\Lambda)$ to the following distinguished triangle (cf. \eqref{eq:distriK}) 
\[\mathcal{K}_{X_\eta/Y_\eta} \to\mathcal{K}_{X_\eta/\eta}\to \mathcal{K}_{X_\eta/Y_\eta/\eta}\xrightarrow{+1},\] 
we obtain a commutative diagram by the constructions
\begin{align}\label{eq:PsiandKXYS}
	\begin{gathered}
		\xymatrix{
			\Psi_{h}(\mathcal{K}_{X_\eta/Y_\eta}) \ar[r]\ar[d]^-{\eqref{PsiKXYdef}}&\Psi_{h}(\mathcal{K}_{X_\eta/\eta})\ar[d]^-{\eqref{PsiKXYdef}} \ar[r]& \Psi_{h}(\mathcal{K}_{X_\eta/Y_\eta/\eta}) \ar[r]^-{+1} \ar[d]^-{\eqref{eq:PsiKXYS}}&\\
			\mathcal{K}_{X_s/Y_s} \ar[r]&\mathcal{K}_{X_s/s}\ar[r]& \mathcal{K}_{X_s/Y_s/s}\ar[r]^-{+1}& .
		}
	\end{gathered}
\end{align}
Similar to \eqref{spDef2}, the morphism $\Psi_{h}(\mathcal{K}_{X_\eta/Y_\eta/\eta})\to \mathcal{K}_{X_s/Y_s/s}$ induces the following specialization maps
\begin{align}
	\begin{gathered}
		\xymatrix{ 
			H^0_{Z_\eta}(X_\eta,\mathcal{K}_{X_\eta/Y_\eta/\eta}) \ar[d]\ar[r]^-{\rm sp} &H^0_{Z_s}(X_s,\mathcal{K}_{X_s/Y_s/s})\ar[d]\\
			H^0(X_\eta,\mathcal{K}_{X_\eta/Y_\eta/\eta}) \ar[r]^-{\rm sp} &H^0(X_s,\mathcal{K}_{X_s/Y_s/s}).
		}
	\end{gathered}
\end{align}
\begin{proposition}\label{prop:spOfNA}
Under the notation and assumptions in \ref{subsec:spOfNAclass}.
Let $\mathcal{F}\in D_{\rm ctf}(X_\eta,\Lambda)$ such that $f_\eta$ $($resp. $f_s)$ is universally locally acyclic relatively to $\mathcal{F}$ $($resp. $\Psi_{h}(\mathcal{F}))$ outside $Z_\eta$ $($resp. $Z_s)$. Then we have 
	\begin{align}\label{spOfNAinprop}
		{\rm sp} (\widetilde{C}_{X_\eta/Y_\eta/\eta}^{Z_\eta}(\mathcal F))= \widetilde{C}_{X_s/Y_s/s}^{Z_s}(\Psi_{h}(\mathcal F))\quad{\rm in}\quad H^0_{Z_s}(X_s,\mathcal{K}_{X_s/Y_s/s}).
	\end{align}
If moreover $H^0(Z_t,\mathcal K_{Z_t/{Y_t}})=0~{\rm and}~H^1(Z_t,\mathcal K_{Z_t/{Y_t}})=0$ for $t\in\{s,\eta\}$, then we have 
\begin{align}
	{\rm sp}(C^{Z_\eta}_{X_\eta/Y_\eta/\eta}(\mathcal F))={C}_{X_s/Y_s/s}^{Z_s}(\Psi_{h}(\mathcal F))\quad{\rm in}\quad H^0_{Z_s}(X_s,\mathcal{K}_{X_s/s}).
\end{align}
In particular, if $\mathcal{F}\in D_{\rm ctf}(X,\Lambda) $ such that $h$ is universally locally acyclic relatively to $\mathcal{F}$  and that $f_\eta$ $($resp. $f_s)$ is universally locally acyclic relatively to $\mathcal{F}|_{X_\eta}$ $($resp. $\mathcal{F}|_{X_s})$  outside $Z_\eta$ $($resp. $Z_s)$, then $\Psi_{h}(\mathcal{F}|_{X_{\eta}})\simeq \mathcal{F}|_{X_s}$ and hence 
	\begin{align}\label{spOfNAinpropCor}
		{\rm sp}(\widetilde{C}_{X_\eta/Y_\eta/\eta}^{Z_\eta}(\mathcal{F}|_{X_{\eta}}))=\widetilde{C}_{X_s/Y_s/s}^{Z_s}(\mathcal{F}|_{X_s})\quad{\rm in}\quad H^0_{Z_s}(X_s,\mathcal{K}_{X_s/Y_s/s}).
	\end{align}
\end{proposition}
\begin{proof}
The equality \eqref{spOfNAinprop}
results from the following commutative diagram
		\begin{align*}\small
		\begin{gathered}
			\xymatrix@R=0.6cm@C=0.2em{
				&\Psi_{h}(\Lambda_{X_{\eta}})\ar[rr]&&\Psi_{h}(\delta_{\eta1}^*i_\eta^!(\mathcal{F}\boxtimes_\eta D_{X_\eta/\eta}(\mathcal{F}))) \ar[dd]\ar[rr]&&\Psi_{h}(\mathcal{K}_{X_\eta/\eta})\ar[dd]\ar[dl]\\
				\Lambda \ar[rr]\ar[ur]&&{\delta}_{s1}^{\ast}{i}_s^{!}(\Psi_{h}(\mathcal{F})\boxtimes_s D_{X_s/s}(\Psi_{h}(\mathcal{F}))) \ar[dd]\ar[rr]&&\mathcal{K}_{X_s/s}\ar[dd]&\\
				&&& \Psi_{h}(\delta_{\eta1}^*\delta_\eta^{\triangle}(\mathcal{F}\boxtimes_\eta D_{X_\eta/\eta}(\mathcal{F}))) \ar[rr]&& \Psi_{h}(\mathcal{K}_{X_\eta/Y_\eta/\eta}) \ar[dl]\\
				&& {\delta}_{s1}^{\ast}{\delta}_s^{\triangle}(\Psi_{h}(\mathcal{F})\boxtimes_s D_{X_s/s}(\Psi_{h}(\mathcal{F}))) \ar[rr]&& \mathcal{K}_{X_s/Y_s/s},&\\
			}
		\end{gathered}
	\end{align*}	
where the commutativity follows from Proposition \ref{prop:sp}, \eqref{bcdiag3-factorization} and \eqref{eq:PsiandKXYS}.
\end{proof}

\section{Fibration formula}\label{sec:fib}

Inspired by the induction formula for the  characteristic classes \cite[Proposition 5.3.7]{UYZ}, we prove the following fibration formula.
\begin{theorem}[Fibration formula]\label{conj:milclass}
	Under the notation and conditions (C1)-(C3) in \ref{subsec:notationforXYS}, we have  
	\begin{align}
		\label{eq:milclass}&C_{X/S}(\mathcal F)=\delta^!C_{X/Y}(\mathcal F)+
		C_{X/Y/S}^{Z}(\mathcal F)\quad{\rm in}\quad H^0(X,\mathcal K_{X/S}),
	\end{align}
	where $\delta^!:H^{0}(X,\mathcal K_{X/Y})\xrightarrow{}  H^{0}(X,\mathcal K_{X/S})$ is defined in \eqref{eq:cherncup1}.
\end{theorem}

When $g: Y\to S$ is smooth of relative dimension $r$, we have
$\delta^!=c_{r}(f^*\Omega_{Y/S}^{1,\vee})$ by \eqref{delta!Cherneq}. Then \eqref{eq:milclass} can be rewritten as
\begin{align}\label{eq:milclass1-1}
	\begin{split}
		C_{X/S}(\mathcal F)&=c_r(f^\ast\Omega^{1,\vee}_{Y/S})\cap C_{X/Y}(\mathcal F)+
		C_{X/Y/S}^{Z}(\mathcal F)\quad{\rm in}\quad H^0(X,\mathcal K_{X/S}).\\
	\end{split}
\end{align}

In the case that $X=Y\to S={\rm Spec}k$ is smooth over a field $k$. Since ${\rm id}:X\setminus Z\to X\setminus Z$ is universally locally acyclic  relatively to $\mathcal F|_{X\setminus Z}$,  the cohomology sheaves of $\mathcal F|_{X\setminus Z}$ are locally constant on $X\setminus Z$. In this case, 
the class $C_{X/Y/S}^{Z}(\mathcal F)$ is Abbes-Saito's localized characteristic class \cite[Definition 5.2.1]{AS07} and
\eqref{eq:milclass} follows from \cite[Proposition 5.2.3]{AS07}.

\subsection{}
Theorem \ref{conj:milclass} will follow from Proposition \ref{lem:TisWellDefined} later.
Let us explain the main idea of the proof.
Let $U=X\setminus Z$. 
Theorem \ref{thm:changebaseHomo} implies that
the image of $C_{X/S}(\mathcal F)-\delta^!C_{X/Y}(\mathcal F)$ 
 in $H^0(U,\mathcal K_{U/S})$ is zero.
From the exact sequence
\begin{align}
	H^{-1}(U,\mathcal K_{U/S})\to H^0(Z,\mathcal K_{Z/S})\xrightarrow{\tau_\ast}H^0(X,\mathcal K_{X/S})\to H^0(U,\mathcal K_{U/S}),
\end{align}
we have
$C_{X/S}(\mathcal F)-\delta^!C_{X/Y}(\mathcal F)\in {\rm Im}(\tau_\ast)$.
The lifting of $C_{X/S}(\mathcal F)-\delta^!C_{X/Y}(\mathcal F)$ 
to $H^0(Z,\mathcal K_{Z/S})$ may not be unique. In order to overcome this problem, we  enhance the constructions of $C_{X/S}$ and $C^Z_{X/Y/S}$ to the $\infty$-categorical level and construct an intermediate map $L^Z_{X/Y/S}(\mathcal F)$ together with a coherent commutative diagram (cf. \eqref{eq:CTDiffExplain})
\begin{align}\label{eq:CTDiffExplain-0}
	\begin{gathered}
		\xymatrix@C=2cm{
			&\Lambda \ar[ld]_-{C^Z_{X/Y/S}(\mathcal{F})} \ar[d]^-{L^Z_{X/Y/S}(\mathcal{F})}\ar@{=}[r]&\Lambda\ar[d]^-{C_{X/S}(\mathcal F)-\delta^!C_{X/Y}(\mathcal F)}\\
			\tau_*\tau^!\mathcal{K}_{X/Y/S}& \tau_*\mathcal{K}_{Z/S} \ar[l]\ar[r]&\mathcal K_{X/S}.
		}
	\end{gathered}
\end{align}
Since $H^0_Z(X,\mathcal K_{X/Y/S})\simeq H^0(Z,\mathcal K_{Z/S})$, the diagram \eqref{eq:CTDiffExplain-0}
implies the fibration formula \eqref{eq:milclass}.

\subsection{}\label{subsec:LiftCofiber}
Let us recall a lifting result in $\infty$-category, which will be used in constructing the map  $L^Z_{X/Y/S}(\mathcal F)$ and the diagram \eqref{eq:CTDiffExplain-0}.
Let $\mathcal{C}$ be a stable $\infty$-category. 
Let $\mathcal{E}\subseteq {\rm Fun}(\Delta^1\times\Delta^1,\mathcal{C})$ be the full sub-$\infty$-category spanned by cofiber sequences in $\mathcal C$. Let $\theta: \mathcal{E}\to {\rm Fun}(\Delta^1,\mathcal C)$ be the functor sending a cofiber sequence $P''\to P\to P'$ to $P\to P'$, which is  a trivial Kan fibration by \cite[Remark 1.1.1.7]{HA}.

Consider a  commutative diagram in $\mathcal C$ between cofiber sequences
\begin{align}
	\begin{gathered}
		\xymatrix@R=0.4cm{
			P''\ar[r] &P \ar[r]\ar[d] &P' \ar[d]\\
			Q''\ar[r] &Q\ar[r]&Q'.
		}
	\end{gathered}
\end{align}
We view the right square as a $1$-simplex $\Delta^1\to {\rm Fun}(\Delta^1,\mathcal{C})$.
Since $\theta$ is a trivial Kan fibration, there is a lifting $\Delta^1\to \mathcal{E}$, which is unique up to a contractible space, making the following diagram
\begin{equation}
	\begin{gathered}
		\xymatrix@R=0.4cm{
			\partial \Delta^1 \ar[r] \ar[d] & \mathcal{E}\ar[d]^{\theta}\\
			\Delta^1 \ar[r] \ar@{-->}[ur]& {\rm Fun}(\Delta^1,\mathcal C)}
	\end{gathered}
\end{equation}
commutes. Thus there is a morphism $P''\to Q''$ such that the following diagram commutes:
\begin{equation}\label{eq:constructP2Q}
	\begin{gathered}
		\xymatrix@R=0.6cm{
			P''\ar[r]\ar@{-->}[d] &P \ar[r]\ar[d] &P' \ar[d]\\
				Q''\ar[r] &Q\ar[r]&Q'.
		}
	\end{gathered}
\end{equation}
Now for a  commutative diagram (with solid arrows) between three cofiber sequences
\begin{equation}
	\begin{gathered}\label{eq:triprism}
		\xymatrix@R=0.4cm{
			&P''\ar[rr] \ar@{-->}[d]\ar@{-->}[ddl]&&P \ar[rr]\ar[d] \ar[ddl]&&P' \ar[d]\ar[ddl]\\
			&Q''\ar[rr] \ar@{-->}[dl]&&Q\ar[rr]\ar[dl]&&Q'\ar[dl]\\
			R'' \ar[rr] &&R\ar[rr] &&R',
		}
	\end{gathered}
\end{equation}
by \eqref{eq:constructP2Q}, we may find dashed arrows $P''\dashrightarrow Q'', Q''\dashrightarrow  R''$ and $P''\dashrightarrow  R''$ such that any of the three lateral faces 
of the triangular prism \eqref{eq:triprism} are commutative.
We claim that the triangle
\begin{align}\label{eq:leftfaceInTriprism}
	\begin{gathered}
		\xymatrix@R=0.6cm{
			&P''\ar@{-->}[ld]\ar@{-->}[d]\\
			R''&Q''\ar@{-->}[l]
		}
	\end{gathered}
\end{align}
formed by dashed arrows in \eqref{eq:triprism} is also commutative. Indeed the right-smaller triangular prism in \eqref{eq:triprism} defines a $2$-simplex $\Delta^2\to {\rm Fun}(\Delta^1,\mathcal C)$, again by the triviality of $\theta$, there is a lifting $\Delta^2\to \mathcal{E}$ making 
\begin{equation}
	\begin{gathered}
		\xymatrix@R=0.4cm{
			\partial \Delta^2 \ar[r] \ar[d] & \mathcal{E}\ar[d]^-{\theta}\\
			\Delta^2 \ar[r] \ar@{-->}[ur]& {\rm Fun}(\Delta^1,\mathcal C)}
	\end{gathered}
\end{equation}
into a commutative diagram. This proves the claim. 
Note that if $\mathcal C$ is only a triangulated category, even though the three dashed arrows may still exist such that any of the three lateral faces of the triangular prism \eqref{eq:triprism} are commutative, but the diagram \eqref{eq:leftfaceInTriprism} may not commute.

\subsection{}\label{subsec:diffHomotopy}
To handle the difference
$C_{X/S}-\delta^!C_{X/Y}$ and construct the triangular prisms \eqref{eq:triprism}, we formulate the following fact.
Let $\mathcal{C}$ be a stable $\infty$-category.
Consider the following two commutative diagrams in $\mathcal{C}$ (viewed as maps $\Delta^1\times \Lambda_1^2\to \mathcal{C}$), where the different edges are $f$ and $h$:
\begin{align}
	\begin{gathered}\label{eq:originalPrism0}
		\xymatrix@R=0.6cm{
			P\ar[r]^-{g'}\ar[d]_-{f}&P'\ar[d]^-{f'}&P\ar[r]^-{g'}\ar[d]_-{h}&P'\ar[d]^-{f'}\\
			Q\ar[r]^-g\ar[d]_-{r}&Q'\ar[d]^-{r'}&	Q\ar[r]^-g\ar[d]_-{r}&Q'\ar[d]^-{r'}\\
			R\ar[r]^-s&R',&	R\ar[r]^-s&R'.
		}
	\end{gathered}
\end{align}
We extend them to triangular prisms (viewed as maps $\Delta^1\times \Delta^2\to \mathcal{C}$)
\begin{align}\label{eq:originalPrism1}
	\begin{gathered}
		\xymatrix@C=2cm@R=0.6cm{
			&P\ar[r]^-{g'}\ar[d]^-{f}\ar[ddl]_(0.3){rf}&P'\ar[d]^-{f'}\ar[ddl]_(0.3){r'f'}&P\ar[r]^-{g'}\ar[d]^-{h}\ar[ddl]_(0.3){rh}&P'\ar[d]^-{f'}\ar[ddl]_(0.3){r'f'}\\
			&Q \ar[r]_(0.2)g \ar[ld]^-{r}&Q' \ar[ld]^-{r'}&  Q \ar[r]_(0.2)g \ar[ld]^-{r}&Q' \ar[ld]^-{r'}\\
			R  \ar[r]_-s&R',&  R  \ar[r]_-s&R', \\
		}
	\end{gathered}
\end{align}
which define the compositions $rf,rh$ and $r'f'$.
Then we have commutative diagrams in $\mathcal{C}$
\begin{align}\label{eq:HomotopyZero}
	\begin{gathered}
		\xymatrix@C=2cm@R=0.6cm{
			P\ar[r]\ar[d]_-{f-h}&0\ar[d]	&&P\ar[r]\ar[d]^-{f-h}\ar[ddl]_-{rf-rh}&0\ar[d]\ar[ddl]\\
			Q\ar[r]_-g&Q' ,              &&Q \ar[r]_(0.2)g \ar[ld]^-{r}&Q' \ar[ld]^-{r'} \\
			&&R  \ar[r]_-s&R' ,\\
		}
	\end{gathered}
\end{align}
which depend on \eqref{eq:originalPrism0} and \eqref{eq:originalPrism1} respectively.
Indeed, the diagram \eqref{eq:originalPrism0} gives a homotopy $gf\simeq gh$, thus
$P\xrightarrow{f-h}Q\xrightarrow{g} Q'$ is homotopic to zero. 
This determines the left  diagram in \eqref{eq:HomotopyZero}. 
More precisely, the two diagrams can be constructed as follows.
We denote the composition $Q\oplus Q\simeq Q\amalg  Q\xrightarrow{\rm can}Q$ again by ${\rm can}$.
The cofiber sequence $P\xrightarrow{({\rm id,-id})}P\oplus P\xrightarrow{\rm can}P$ is defined by a pull-back square
\begin{align}
	\begin{gathered}\label{eq:sumFunctor}
		\xymatrix@R=0.4cm{
			P\ar[d]_-{({\rm id,-id})}\ar[r]&0\ar[d]\\
			P\oplus P\ar[r]^-{\rm can}&P,
		}
	\end{gathered}
\end{align}
which is also a push-out square. 
The  morphism $g':P\to P'$ induces a commutative diagram
\begin{align}\label{eq:sumSquare}
	\begin{gathered}
		\xymatrix@R=0.4cm{
			P\ar[dd]_-{({\rm id,-id})}\ar[rd]^-{g'}\ar[rr]&&0\ar[rd]\ar[dd]\\
			&P'\ar[rr]\ar[dd]_(0.3){({\rm id,-id})}&&0\ar[dd]\\
			P\oplus P\ar[rr]^(0.7){\rm can}\ar[rd]_-{(g',g')}&&P\ar[rd]^-{g'}\\
			&P'\oplus P'\ar[rr]^-{\rm can}&&P'.
		}
	\end{gathered}
\end{align}
The morphism $P\xrightarrow{f-h}Q$ is the composition $P\xrightarrow{({\rm id},-{\rm id})}P\oplus P\xrightarrow{(f,h)}Q\oplus Q\xrightarrow{\rm can}Q$.
Then the  diagrams in \eqref{eq:HomotopyZero} are defined by the following commutative diagram
\begin{equation}{\small
	\begin{tikzcd}[row sep=small]
	&P \ar[rr]\ar[dd,"{({\rm id,-id})}"] &&P'\ar[rr]\ar[dd,"{({\rm id,-id})}"]&&0\ar[dd]\\
	&&&&&\\
	&P\oplus P\ar[rr,"{(g',g')}"] \ar[dd,"{(f,h)}",dashed]\ar[dddl,"{(rf,rh)}"']&&P'\oplus P' \ar[rr,"{\rm can}" near start] \ar[dd,"{(f',f')}",dashed]\ar[dddl]&&P'\ar[dd]\ar[dddl]\\
	&&&&&\\
	&Q\oplus Q \ar[dd,dashed] \ar[dl,"{(r,r)}",dashed]\ar[rr,"{(g,g)}" near start,dashed]&&Q'\oplus Q' \ar[dd,dashed] \ar[dl,dashed]\ar[rr,"{\rm can}" near start,dashed]&&Q'\ar[dd,equal] \ar[dl]\\
	R\oplus R \ar[dd,"{\rm can}"']\ar[rr]&&R'\oplus R' \ar[dd,"{\rm can}" near start]\ar[rr,"{\rm can}" near end]&&R' \ar[dd,equal]&\\
	&Q \ar[rr,dashed]\ar[dl,dashed]&&Q'\ar[rr,dashed,"=" near start]\ar[dl,dashed]&&Q'\ar[dl]\\
	R\ar[rr]&&R'\ar[rr,equal]&&R'.&\\
\end{tikzcd}
}
\end{equation}

\subsection{}
To construct prisms similar to \eqref{eq:originalPrism1} in our setting (cf. \ref{lem:TisWellDefined}),
we first give a coherent commutative cube in Lemma \ref{lem:cubicdiagram} below.
Let ${\rm Corr}_S^\otimes={\rm Corr}^{\rm proper}_{\rm sep;all}(\rm Sch_S)$ be the symmetric monoidal $(\infty,2)$-category of correspondences (cf. \cite[Part III]{GR17}).
Let $\mathcal D: {\rm Corr}^\otimes_S\to {\rm Cat}_\infty$ be a 6-functor formalism of \'etale sheaves of $\Lambda$-modules on schemes over $S$ such that $D(-,\Lambda)=h\mathcal D(-)$ is the homotopy category of $\mathcal{D}(-)$.
Let $\mathcal D_{\rm cons}(-)\subseteq \mathcal D(-)$ be the full sub-$\infty$-category spanned by perfect-constructible complexes.
For any Noetherian scheme $X$ over $S$, we have $D_{\rm ctf}(X,\Lambda)=h\mathcal D_{\rm cons}(X)$ (cf. \cite[7.2]{HRS} and \cite[\S 2]{HS}).
Consider a commutative diagram in ${\rm Sch}_S$
\begin{align}\small
\begin{gathered}\label{eq:cubicSch}
\xymatrix@R=0.4cm{
&U'\ar@{-->}[dd]^(0.3){\delta_1^\prime}\ar[rr]^-{r'}\ar[ld]_-{j_1^\prime}&&V'\ar[dd]^-{\delta_0^\prime}\ar[ld]^-{j_0^\prime}\\
U\ar[rr]_(0.7)r\ar[dd]_-{\delta_1}&&V\ar[dd]^(0.3){\delta_0}\\
&W'\ar@{-->}[rr]_(0.3){i'}\ar@{-->}[ld]_(0.5){j_1}\ar@{-->}[dd]^(0.2){p_1^\prime}&&T'\ar[ld]^-{j_0}\ar[dd]^-{p_0^\prime}\\
W\ar[rr]_(0.7){i}\ar[dd]_-{p_1}&&T\ar[dd]^(0.3){p_0}\\
&W_0^\prime\ar@{-->}[ld]_-{s_1}\ar@{-->}[rr]_(0.3){\delta'}&&T_0^\prime\ar[ld]^-{s_0}\\
W_0\ar[rr]_-\delta&&T_0,
}
\end{gathered}
\end{align}
where, except for the squares on the left and right sides, all squares are cartesian.
Let $\mathcal F\in\mathcal D_{\rm cons}(T)$ and $\mathcal G\in\mathcal D_{\rm cons}(T_0)$. Recall that we have a canonical morphism (cf. \eqref{eq:carGtrans1})
\begin{align}\label{eq:i}
c_{\delta,p_0,\mathcal F,\mathcal G}: i^\ast\mathcal F\otimes^L p_1^\ast\delta^!\mathcal G\xrightarrow{} i^!(\mathcal F\otimes^Lp_0^\ast\mathcal G).
\end{align}
By Lemma \ref{lem:TransBforC}.(2),
there is a commutative diagram for the face $UW_0VT_0$ in \eqref{eq:cubicSch}
\begin{align}\label{eq:UVWT}
\begin{gathered}
\xymatrix@C=2cm{
\delta_1^\ast(i^\ast\mathcal F\otimes^L p_1^\ast\delta^!\mathcal G)\ar[r]^-{\delta_1^\ast c_{\delta,p_0,\mathcal F,\mathcal G}}\ar[d]^-\simeq&	\delta_1^\ast i^!(\mathcal F\otimes^Lp_0^\ast\mathcal G)\ar[d]^-{\delta_1^\ast  i^!\to r^!\delta_0^\ast}_-{\rm b.c}\\
r^\ast \delta_0^\ast\mathcal F\otimes^L \delta_1^\ast p_1^\ast\delta^!\mathcal G\ar[r]^-{c_{\delta,p_0\delta_0,\delta_0^\ast\mathcal F,\mathcal G}}&r^!(\delta_0^\ast\mathcal F\otimes^L \delta_0^\ast p_0^\ast\mathcal G).
}
\end{gathered}
\end{align}
Let $\mathcal F'=j_0^\ast\mathcal F$ and $\mathcal G'=s_0^\ast\mathcal G$.
For the face $U'W'_0V'T'_0$  in \eqref{eq:cubicSch}, we have a commutative diagram
\begin{align}\label{eq:UVWTprime}
\begin{gathered}
\xymatrix@C=2cm{
\delta_1^{\prime\ast}(i^{\prime\ast}\mathcal F'\otimes^L p_1^{\prime\ast}\delta^{\prime!} \mathcal G')\ar[r]^-{\delta_1^{\prime\ast} c_{\delta',p_0^\prime,\mathcal F',\mathcal G'}}\ar[d]^-{\simeq}&	\delta_1^{\prime\ast} i^{\prime!}(\mathcal F'\otimes^L p_0^{\prime\ast}\mathcal G')\ar[d]^-{\delta_1^{\prime\ast}  i^{\prime!}\to r^{\prime!}\delta_0^{\prime\ast}}_-{\rm b.c}\\
r^{\prime\ast} \delta_0^{\prime\ast} \mathcal F'\otimes^L \delta_1^{\prime\ast} p_1^{\prime\ast}\delta^{\prime!}\mathcal G'\ar[r]^-{c_{\delta',p_0^\prime\delta_0^\prime,\delta_0^{\prime\ast}\mathcal F',\mathcal G'}}&r^{\prime!}(\delta_0^{\prime\ast} \mathcal F'\otimes^L \delta_0^{\prime\ast} p_0^{\prime\ast}\mathcal G').
}
\end{gathered}
\end{align}
\begin{lemma}\label{lem:cubicdiagram}
Consider the commutative diagram \eqref{eq:cubicSch}.
Let $\mathcal F\in\mathcal D_{\rm cons}(T)$ and $\mathcal G\in\mathcal D_{\rm cons}(T_0)$.
Put $\mathcal F'=j_0^\ast\mathcal F$ and $\mathcal G'=s_0^\ast\mathcal G$.
There is a coherent commutative diagram $\Delta^1\times\Delta^1\times\Delta^1\to \mathcal D(U')$:
\begin{align}\label{eq:Cubic}
\begin{gathered}\small
\xymatrix@C=0.1cm@R=0.2cm{
&j_1^{\prime\ast} \delta_1^\ast(i^\ast\mathcal F\otimes^L p_1^\ast\delta^!\mathcal G)\ar[rr]\ar@{-->}[dd]\ar[ld]^-{}_{\rm b.c}&&	j_1^{\prime\ast} \delta_1^\ast i^!(\mathcal F\otimes^Lp_0^\ast\mathcal G)\ar[dd]\ar[ld]^-{}_{\rm b.c}\\
\delta_1^{\prime\ast}(i^{\prime\ast}\mathcal F'\otimes^L p_1^{\prime\ast}\delta^{\prime!} \mathcal G')\ar[rr]\ar[dd]&&	\delta_1^{\prime\ast} i^{\prime!}(\mathcal F'\otimes^L p_0^{\prime\ast}\mathcal G')\ar[dd]\\
&j_1^{\prime\ast} r^\ast \delta_0^\ast\mathcal F\otimes^L j_1^{\prime\ast}\delta_1^\ast p_1^\ast\delta^!\mathcal G\ar@{-->}[rr]\ar@{-->}[ld]^-{}_{\rm b.c}&&j_1^{\prime\ast} r^!(\delta_0^\ast\mathcal F\otimes^L \delta_0^\ast p_0^\ast\mathcal G)\ar[ld]^-{}_{\rm b.c}\\
r^{\prime\ast} \delta_0^{\prime\ast} \mathcal F'\otimes^L \delta_1^{\prime\ast}p_1^{\prime\ast} \delta^{\prime!}\mathcal G'\ar[rr]&&r^{\prime!}(\delta_0^{\prime\ast} \mathcal F'\otimes^L \delta_0^{\prime\ast} p_0^{\prime\ast}\mathcal G').
}
\end{gathered}
\end{align}
\end{lemma}

In the proof, we will use a result of the exponentiation of cocartesian fibrations \cite[Proposition 3.1.2.1]{HTT}.
Let $\mathcal E\to \mathcal C$ be a cocartesian fibration of $\infty$-categories and $K$ a simplicial set.
Then ${\rm Fun}(K,\mathcal E)\to {\rm Fun}(K,\mathcal C)$ is a cocartesian fibration. An edge
$f\xrightarrow{e}f'$ in ${\rm Fun}(K,\mathcal E)$ is a cocartesian edge (for the fibration ${\rm Fun}(K,\mathcal E)\to {\rm Fun}(K,\mathcal C)$) if and only if for any vertex $v\in K$, $f(v)\xrightarrow{e(v)}f'(v)$ is a cocartesian edge in $\mathcal E$ (for the fibration $\mathcal E\to \mathcal C$).

The following proof is due to Jiangnan Xiong. The main idea is to give a universal characterization of $c_{\delta,p_0,\mathcal F,\mathcal G}$ by using cocartesian/cartesian edges.
\begin{proof}
We recall some standard constructions. 
By unstraightening $\mathcal D$, we have a cocartesian fibration $\mathcal C^\otimes_{S,\Lambda}\to {\rm Corr}_S^\otimes$. 	In the following, we will denote a correspondence
\begin{align}
	\begin{gathered}
		\xymatrix@R=0.4cm@C=0.4cm{
			C\ar[r]\ar[d]&X\\
			Y
		}
	\end{gathered}
\end{align}
simply by $Y\leftarrow C\to X$.
Objects in ${\rm Corr}_S^\otimes$ are given by $(X_1,\cdots, X_n)$ 
with $n\geq 0$ and $X_i\in {\rm Sch}_S$. 
Objects in $\mathcal C^\otimes_{S,\Lambda}$ are given by $(X_1,\cdots, X_n;\mathcal F_1,\cdots,\mathcal F_n)$
with $\mathcal F_i\in\mathcal D(X_i)$.
We have an equivalence
\begin{align}
	\mathcal D(X_1,\cdots, X_n)\xrightarrow{\simeq}\prod\limits_{i=1}^n \mathcal D(X_i).
\end{align}
Let $\langle n\rangle=\{\ast,1,\cdots,n\}$.
A morphism $(X_1,\cdots, X_n)\to (Y_1,\cdots, Y_m)$ in ${\rm Corr}_S^\otimes$ is given by a map $\alpha:\langle n\rangle\to\langle m\rangle$ with $\alpha(\ast)=\ast$ and  correspondences
\begin{align}
	\begin{gathered}
		\xymatrix@R=0.4cm@C=2cm{
			Y_j&C_j\ar[r]^-{f_j=(f_{ji})}\ar[l]_-{g_j}&\prod\limits_{i\in \alpha^{-1}(j)}X_i
		}
	\end{gathered}
\end{align}
for any $1\leq j\leq m$.
For simplicity, we denote such a morphism by $(\alpha,g_!f^*)=(\alpha,(g_{j!}(f_{ji})^*))$ or by $g_!f^*=(g_{j!}(f_{ji})^*)$.
The functor $\mathcal D(g_!f^*):\prod\limits_{i=1}^n \mathcal D(X_i)\to \prod\limits_{j=1}^m \mathcal D(Y_j)$
sends $(\mathcal F_1,\cdots,\mathcal F_n)$ to $(\mathcal G_1,\cdots,\mathcal G_m)$ with
\begin{align}\label{eq:detGj}
	\mathcal G_j=g_{j!}f_j^\ast(\boxtimes^L_{i\in\alpha^{-1}(j)}\mathcal F_i)=g_{j!}(\otimes^L_{i\in\alpha^{-1}(j)}f_{ji}^\ast\mathcal F_i).
\end{align}
We have a  cocartesian edge (for the  fibration $\mathcal C^\otimes_{S,\Lambda}\to {\rm Corr}_S^\otimes$) above $g_!f^*$
\[(X_1,\cdots, X_n;\mathcal F_1,\cdots,\mathcal F_n)\to (Y_1,\cdots, Y_m;\mathcal G_1,\cdots,\mathcal G_m),\] 
where $\mathcal G_i$ is defined by \eqref{eq:detGj}.
If $m=n$ and $\alpha:\langle n\rangle\to\langle m\rangle$ is the identity, then
the functor
\begin{align}
	\begin{split}
		\mathcal D(g_!f^*):\prod\limits_{i=1}^n \mathcal D(X_i)&\to \prod\limits_{i=1}^n \mathcal D(Y_i);\quad (\mathcal F_i)\mapsto (g_{i!}f_i^\ast\mathcal F_i)
	\end{split}
\end{align}
admits a right adjoint
\begin{align}
	\begin{split}
		\prod\limits_{i=1}^n \mathcal D(Y_i)&\to \prod\limits_{i=1}^n \mathcal D(X_i);\quad
		(\mathcal G_i)\mapsto (f_{i\ast}g_i^!\mathcal G_i).
	\end{split}
\end{align}
In this case, there is a locally cartesian edge
\[
(X_1,\cdots,X_n;f_{1\ast}g_1^!\mathcal G_1,\cdots, f_{n\ast}g_n^!\mathcal G_n)\to (Y_1,\cdots, Y_n;\mathcal G_1,\cdots,\mathcal G_n)
\]
above $g_!f^*:(X_1,\cdots, X_n)\to (Y_1,\cdots, Y_n)$ (cf. the dual version of \cite[Corollary 5.2.2.5]{HTT}).

Now we go back to the proof of Lemma \ref{lem:cubicdiagram}.
For convenience, we will use the arrow $\hookrightarrow$ (resp. $\twoheadrightarrow$) to indicate a locally cartesian edge (resp. cocartesian edge).
We first give a universal characterization of $c_{\delta,p_0,\mathcal F,\mathcal G}$.
Indeed, we have a commutative diagram in ${\rm Corr}_S^\otimes$ (cf. the square $WW_0TT_0$ in \eqref{eq:cubicSch}):
\begin{align}
	\begin{gathered}\label{eq:redefineCbelow}
		\xymatrix@C=1.5cm@R=0.4cm{
			(W)\ar@{=}[rr]&&(W)\ar[d]^-{i_!}\\
			(T,W_0)\ar[r]^-{({\rm id},\delta_!)}\ar[u]^-{(i,p_1)^*}&(T,T_0)\ar[r]^-{({\rm id},p_0)^*}&(T),
		}
	\end{gathered}
\end{align}
where $(T,W_0)\xrightarrow{(i,p_1)^\ast} (W)$ is given by the correspondence
$W=W\xrightarrow{(i,p_1)}T\times_S W_0$, $(W)\xrightarrow{i_!} (T)$ is given by the correspondence $T\xleftarrow{i}W=W$, $	(T,W_0)\xrightarrow{({\rm id}, \delta_!)}(T,T_0)$ is given by $(T=T=T, T_0\xleftarrow{\delta}W_0=W_0)$ and $(T,T_0)\xrightarrow{({\rm id},p_0)^\ast}(T)$
is given by $T=T\xrightarrow{({\rm id},p_0)}T\times_S T_0$. The two compositions 
\[
(T,W_0)\to (W)=(W)\to (T) \quad{\rm and}\quad (T,W_0)\to (T,T_0)\to (T)
\]
are both equal to $i_!(i,p_1)^*$, i.e., the correspondence $T\xleftarrow{i}W\xrightarrow{(i,p_1)}T\times_S W_0$.
Then we can characterize $c_{\delta,p_0,\mathcal F,\mathcal G}$ by using the following lifting problem over \eqref{eq:redefineCbelow}:
\begin{align}
	\begin{gathered}\label{eq:redefineCabove}
		\xymatrix{
			(W;i^\ast\mathcal F\otimes^L p_1^*\delta^!\mathcal G)\ar@{-->}[rrd]^-{c'}\ar@{-->}[rr]^-{c_{\delta,p_0,\mathcal F,\mathcal G}}&&(W; i^!(\mathcal F\otimes^Lp_0^*\mathcal G))\ar@{^(->}[d]^-{\rm cartesian}\\
			(T,W_0;\mathcal F, \delta^!\mathcal G)\ar@{^(->}[r]\ar@{->>}[u]^-{\rm cocartesian}&(T,T_0;\mathcal F,\mathcal G)\ar@{->>}[r]&(T,\mathcal F\otimes^Lp_0^*\mathcal G).
		}
	\end{gathered}
\end{align}
Indeed, the lifting $c'$ uniquely exists by the property of cocartesian edges (cf. \cite[2.4.1.1]{HTT}) and $c_{\delta,p_0,\mathcal F,\mathcal G}$ is the unique morphism (up to a contractible space) 
making \eqref{eq:redefineCabove} commutes by the property of cartesian edges.

We denote the  diagram \eqref{eq:redefineCabove} by $\Gamma_{WW_0TT_0}$.
We also have three similar diagrams $\Gamma_{UW_0VT_0}$, $\Gamma_{W'W_0^\prime T'T_0^\prime}$
and $\Gamma_{U'W_0^\prime V'T_0^\prime}$.
Now we prove Lemma \ref{lem:cubicdiagram} by constructing these four diagrams simultaneously. 
Here is a brief summary of the notation:
\begin{enumerate}
	\item The diagram $M:\Delta^2\to {\rm Fun}(\Delta^1\times \Delta^1,\mathcal C^\otimes_{S,\Lambda})$ (cf. \eqref{eq:realizeBottomLine}) realizes 
	 the bottom lines of these four diagrams  at the same time.  
	\item The left (resp. right) vertical arrows of these four diagrams are realized by the diagram $L:\Delta^1\to{\rm Fun}(\Delta^1\times\Delta^1, \mathcal C^\otimes_{S,\Lambda})$ (resp. $R:\Delta^1\to{\rm Fun}(\Delta^1\times\Delta^1, \mathcal C^\otimes_{S,\Lambda})$).
	\item  The diagram $K:\Delta^1\to{\rm Fun}(\Delta^1\times\Delta^1, \mathcal C^\otimes_{S,\Lambda})$ realizes the diagonal dashed arrows (cf. $c'$ in \eqref{eq:redefineCabove}) of these four diagrams.
	\item  The top arrows (cf. $c_{\delta,p_0,\mathcal F,\mathcal G}$ in \eqref{eq:redefineCabove}) of these four diagrams are realized by the diagram $F:\Delta^1\to{\rm Fun}(\Delta^1\times\Delta^1, \mathcal C^\otimes_{S,\Lambda})$ (cf. \eqref{eq:cforFourSquare}). The diagram $F$ determines the required commutative diagram \eqref{eq:Cubic} by pulling back to $U'$.
\end{enumerate}
More precisely, let $M_1:\Delta^0\to {\rm Fun}(\Delta^1\times\Delta^1, \mathcal C^\otimes_{S,\Lambda})$ be the  commutative diagram
\begin{align}\label{eq:Middle1}
	\begin{gathered}
		\xymatrix{
			(T,T_0;\mathcal F,\mathcal G)\ar@{->>}[d]\ar@{->>}[r]&(T',T'_0;\mathcal F',\mathcal G')\ar@{->>}[d]\\
			(V,T_0;\delta_0^\ast \mathcal F,\mathcal G)\ar@{->>}[r]&(V',T_0';\delta_0^{\prime\ast}\mathcal F',\mathcal G'),	
		}
	\end{gathered}
\end{align}
which is above $N_1:\Delta^0\cong\Delta^{\{1\}}\subset\Delta^2\overset{N}{\to}{\rm Fun}(\Delta^1\times\Delta^1,{\rm Corr}_S^\otimes)$, where $N$ is the following diagram:
\begin{align}\small
	\begin{gathered}\label{eq:cubic1}
		\xymatrix@R=0.6cm@C=1.5cm{
			(T,W_0)\ar[rr]^{({\rm id},\delta_!)}\ar[dd]^{(\delta_0^*,{\rm id})}\ar[rd]^{(j_0^*,s_1^*)}&&(T,T_0)\ar@{-->}[dd]_(0.25){(\delta_0^*,\rm id)}\ar[rr]^{({\rm id},p_0)^*}\ar[rd]^{(j_0^*,s_0^*)}&&(T)\ar@{-->}[dd]_(0.25){\delta_0^*}\ar[rd]^{j_0^*}\\
			&(T',W'_0)\ar[dd]_(0.25){(\delta_0^{\prime\ast},{\rm id})}\ar[rr]_(0.35){({\rm id},\delta'_!)}&&(T',T'_0)\ar[dd]_(0.25){(\delta_0^{\prime\ast},{\rm id})}\ar[rr]_(0.35){({\rm id},p_0')^*}&&(T')\ar[dd]^{\delta_0^{\prime\ast}}\\
			(V,W_0)\ar@{-->}[rr]_(0.7){({\rm id},\delta_!)}\ar[rd]_{(j_0^{\prime\ast},s_1^*)}&&(V,T_0)\ar@{-->}[rd]_(0.4){(j_0^{\prime\ast},s_0^*)}\ar@{-->}[rr]_(0.7){({\rm id},p_0\delta_0)^*}&&(V)\ar[rd]_(0.4){j_0^{\prime\ast}}\\
			&(V',W'_0)\ar[rr]_{({\rm id},\delta'_!)}&&(V',T'_0)\ar[rr]_{({\rm id},p_0'\delta'_0)^*}&&(V').
		}
	\end{gathered}
\end{align}
Let $N_{01}=N|_{\Delta^{\{0,1\}}}$ be the left cubic diagram in \eqref{eq:cubic1} and $N_{12}=N|_{\Delta^{\{1,2\}}}$ the right cubic diagram. 
We can extend $M_1$ (cf. \eqref{eq:Middle1}) to a locally cartesian edge 
$M_{01}:\Delta^{\{0,1\}}\to {\rm Fun}(\Delta^1\times\Delta^1, \mathcal C^\otimes_{S,\Lambda})$ above $N_{01}$, and
extend $M_1$ to a cocartesian edge $M_{12}:\Delta^{\{1,2\}}\to {\rm Fun}(\Delta^1\times\Delta^1, \mathcal C^\otimes_{S,\Lambda})$ above $N_{12}$.
Now $M_{01}$ and $M_{12}$ give a diagram $\Lambda_1^2\to{\rm Fun}(\Delta^1\times\Delta^1, \mathcal C^\otimes_{S,\Lambda})$, which can be extended to a diagram $M:\Delta^2\to{\rm Fun}(\Delta^1\times\Delta^1, \mathcal C^\otimes_{S,\Lambda})$ as follows
\begin{align}\tiny\label{eq:realizeBottomLine}
	\begin{gathered}
		\xymatrix@R=0.2cm@C=0.1em{
			(T,W_0;\mathcal F,\delta^!\mathcal G)\ar@{^(->}[rr]\ar[dd]\ar[rd]&&(T,T_0;\mathcal F,\mathcal G)\ar@{-->>}[dd]\ar@{->>}[rr]\ar@{->>}[rd]&&(T;\mathcal F\otimes^Lp_0^*\mathcal G)\ar@{-->>}[dd]\ar@{->>}[rd]\\
			&(T',W'_0;\mathcal{F}',\delta^{\prime!}\mathcal{G}')\ar[dd]\ar@{^(->}[rr]&&(T',T'_0;\mathcal{F}',\mathcal{G}')\ar@{->>}[dd]\ar@{->>}[rr]&&(T';\mathcal{F}'\otimes^Lp_0^{\prime*}\mathcal{G}')\ar@{->>}[dd]\\
			(V,W_0;\delta_0^\ast\mathcal{F},\delta^!\mathcal{G})\ar@{^(-->}[rr]\ar[rd]&&(V,T_0;\delta_0^\ast\mathcal F,\mathcal G)\ar@{-->>}[rd]\ar@{-->>}[rr]&&(V;\delta_0^\ast\mathcal F\otimes^L \delta_0^\ast p_0^*\mathcal G)\ar@{->>}[rd]\\
			&(V',W'_0;\delta_0^{\prime*}\mathcal{F}',\delta^{\prime!}\mathcal{G}')\ar@{^(->}[rr]&&(V',T'_0;\delta_0^{\prime*}\mathcal{F}',\mathcal{G}')\ar@{->>}[rr]&&(V';\delta_0^{\prime\ast}\mathcal{F}'\otimes^L\delta_0^{\prime\ast}p_0^{\prime\ast}\mathcal{G}').
		}
	\end{gathered}
\end{align}
Now let $H:\Delta^2\to{\rm Fun}(\Delta^1\times\Delta^1,{\rm Corr}_S^\otimes)$ be the following diagram
\begin{align}\small
	\begin{gathered}\label{eq:cubic2}
		\xymatrix@R=0.6cm@C=1.5cm{
			(T,W_0)\ar[rr]^{(i,p_1)^*}\ar[dd]_{(\delta_0^*,{\rm id})}\ar[rd]^{(j_0^*,s_1^*)}&&(W)\ar@{-->}[dd]_(0.25){\delta_1^*}\ar[rr]^{i_!}\ar[rd]^{j_1^*}&&(T)\ar@{-->}[dd]_(0.25){\delta_0^*}\ar[rd]^{j_0^*}\\
			&(T',W'_0)\ar[dd]_(0.25){(\delta_0^{\prime\ast},{\rm id})}\ar[rr]_(0.35){(i',p_1')^*}&&(W')\ar[dd]_(0.25){\delta_1^{\prime\ast}}\ar[rr]_(0.35){i'_!}&&(T')\ar[dd]^{\delta_0^{\prime\ast}}\\
			(V,W_0)\ar@{-->}[rr]_(0.7){(r,p_1\delta_1)^*}\ar[rd]_{(j_0^{\prime\ast},s_1^*)}&&(U)\ar@{-->}[rd]_(0.4){j_1^{\prime\ast}}\ar@{-->}[rr]_(0.7){r_!}&&(V)\ar[rd]^{j_0^{\prime\ast}}\\
			&(V',W'_0)\ar[rr]_{(r',p_1'\delta_1')^*}&&(U')\ar[rr]_{r'_!}&&(V').
		}
	\end{gathered}
\end{align}
Note that $N_{02}=H_{02}$ (cf. \eqref{eq:redefineCbelow}).
In the following, we use the standard notation that $\Delta^2\xrightarrow{(\alpha,\beta,\gamma)}\mathcal{C}$ means a functor $f:\Delta^2\to\mathcal{C}$ such that $f|_{\Delta^{\{1,2\}}}=\alpha,f|_{\Delta^{\{0,2\}}}=\beta$ and $f|_{\Delta^{\{0,1\}}}=\gamma$.
Let $L:\Delta^1\to{\rm Fun}(\Delta^1\times\Delta^1, \mathcal C^\otimes_{S,\Lambda})$ be a cocartesian lifting over $H_{01}$ with $L_0=M_0$. 
By the property of cocartesian edge, $\Lambda_0^2\xrightarrow{(\bullet,M_{02},L)}{\rm Fun}(\Delta^1\times\Delta^1, \mathcal C^\otimes_{S,\Lambda})$ can be extended to $\Delta^2\xrightarrow{(K,M_{02},L)}{\rm Fun}(\Delta^1\times\Delta^1, \mathcal C^\otimes_{S,\Lambda})$ over $H$. 
Let $R:\Delta^1\to{\rm Fun}(\Delta^1\times\Delta^1, \mathcal C^\otimes_{S,\Lambda})$ be a locally cartesian lifting over $H_{12}$ with $R_1=M_2$. 
Then $\Lambda_2^2\xrightarrow{(R,K,\bullet)}{\rm Fun}(\Delta^1\times\Delta^1, \mathcal C^\otimes_{S,\Lambda})$ can be extended to $\Delta^2\xrightarrow{(R,K,F)}{\rm Fun}(\Delta^1\times\Delta^1, \mathcal C^\otimes_{S,\Lambda})$. 
By construction, the diagram $F$ is of the following form
\begin{align}\tiny
\begin{gathered}\label{eq:cforFourSquare}
\xymatrix@C=0.2cm@R=0.2cm{
(W;i^\ast\mathcal F\otimes^L p_1^\ast\delta^!\mathcal G)\ar[dd]\ar[rr]\ar[rd]&&(W;i^!(\mathcal F\otimes^L p_0^\ast\mathcal G))\ar@{-->}[dd]\ar[rd]\\
&(W';i^{\prime\ast}\mathcal F'\otimes^L p_1^{\prime\ast}\delta^{\prime!}\mathcal G')\ar[rr]\ar[dd]&&(W';i^{\prime!}(\mathcal F'\otimes^L p_0^{\prime\ast}\mathcal G'))\ar[dd]\\
(U;r^\ast\delta_0^\ast\mathcal F\otimes^L \delta_1^\ast p_1^\ast \delta^!\mathcal G)\ar@{-->}[rr]\ar[rd]&&(U;r^!(\delta_0^\ast\mathcal F\otimes^L \delta_0^\ast p_0^\ast\mathcal G))\ar@{-->}[rd]\\
&(U';r^{\prime\ast}\delta_0^{\prime\ast}\mathcal F'\otimes^L \delta_1^{\prime\ast} p_1^{\prime\ast} \delta^{\prime!}\mathcal G')\ar[rr]&&(U';r^{\prime!}(\delta_0^{\prime\ast}\mathcal F'\otimes^L \delta_0^{\prime\ast} p_0^{\prime\ast}\mathcal G')),
}
\end{gathered}
\end{align}
which determines the required commutative diagram \eqref{eq:Cubic} by pulling back to $U'$. Indeed,
let $A:\Delta^1\to{\rm Fun}(\Delta^1\times\Delta^1,{\rm Corr}_S^\otimes)$ be the following diagram
\begin{align}\small
	\begin{gathered}\label{eq:cubic5}
		\xymatrix{
			(W)\ar[dd]_{\delta_1^*}\ar[rr]^{j_1^{\prime\ast}\delta_1^*}\ar[rd]_-{j_1^*}&&(U')\ar@{-->}[dd]_(0.25){\rm id}\ar[rd]^{\rm id}\\
			&(W')\ar[dd]_(0.25){\delta_1^{\prime\ast}}\ar[rr]_(0.35){\delta_1^{\prime\ast}}&&(U')\ar[dd]^{\rm id}\\
			(U)\ar[rd]_(0.4){j_1^{\prime\ast}}\ar@{-->}[rr]_(0.7){j_1^{\prime\ast}}&&(U')\ar[rd]^{\rm id}\\
			&(U')\ar[rr]^-{\rm id}&&(U'),
		}
	\end{gathered}
\end{align}
where $A_0=H_1$ and $A_1$ is the constant functor taking value $(U')$.
Put ${\rm Fun}_{\mathcal{C}}={\rm Fun}(\Delta^1\times\Delta^1, \mathcal C^\otimes_{S,\Lambda})$ and ${\rm Fun}_S={\rm Fun}(\Delta^1\times\Delta^1,{\rm Corr}_S^\otimes)$. 
Consider the covariant transport functor along $A$
\begin{align}
	A_!:{\rm Fun}_\mathcal{C}\times_{{\rm Fun}_S}\{A_0\}\to{\rm Fun}_\mathcal{C}\times_{{\rm Fun}_S}\{A_1\}={\rm Fun}(\Delta^1\times\Delta^1,\mathcal{D}(U')).
\end{align}
The diagram $F$ defines an edge $\Delta^1\to {\rm Fun}_\mathcal{C}\times_{{\rm Fun}_S}\{A_0\}$.
Now let $G=A_!\circ F:\Delta^1\to{\rm Fun}(\Delta^1\times\Delta^1,\mathcal{D}(U'))$, which defines a diagram $\Delta^1\times\Delta^1\times\Delta^1\to\mathcal{D}(U')$. Unwinding the definitions, $G$ is the required commutative diagram \eqref{eq:Cubic}.
\end{proof}
\subsection{}\label{sub:ConstructCube}
We apply Lemma \ref{lem:cubicdiagram} to construct the diagram \eqref{eq:adjToDeltaKcubic5} below, which is crucial for proving the fibration formula \eqref{eq:milclass}.
Consider the diagrams \eqref{xyd:prop:Gysim:fgh2} and \eqref{eq:propchangebasehomo:1}.
Let $j: U=X\setminus Z\to X$ be the open immersion and $\mathcal F\in D_{\rm cons}(X)$. We put
\begin{align}
	\mathcal H_S&=R\mathcal Hom_{X\times_SX}({\rm pr}_2^\ast \mathcal F,{\rm pr}_1^!\mathcal F),\qquad\qquad \mathcal T_S=\mathcal F\boxtimes^L_S D_{X/S}(\mathcal F),\\
	\mathcal H'_S&=R\mathcal Hom_{U\times_SU}({\rm pr}_2^\ast j^\ast \mathcal F,{\rm pr}_1^!j^\ast \mathcal F),\qquad\;\mathcal T'_S=j^\ast \mathcal F\boxtimes^L_{S}D_{U/S}(j^\ast \mathcal F).
\end{align}
For later convenience,
we form the following commutative diagram
\begin{align}\small
	\begin{gathered}\label{eq:cubicSchidelta}
		\xymatrix@R=0.4cm{
			&U\ar@{-->}[dd]^(0.3){\delta_1^\prime}\ar@{=}[rr]\ar@{_(->}[ld]_-{j}&&U\ar[dd]^-{\delta_0^\prime}\ar@{_(->}[ld]_-{j}\\
			X\ar@/_2.5pc/[dddd]_-f\ar@{=}[rr]\ar[dd]_-{\delta_1}&&X\ar[dd]^(0.3){\delta_0}\\
			&U\times_YU\ar@{-->}[rr]_(0.3){i'}\ar@{-->}[ld]_-{j\times j}\ar@{-->}[dd]^(0.3){p'}&&U\times_SU\ar[ld]^-{j\times j}\ar[dd]^-{f'\times f'}\\
			X\times_YX\ar[rr]_(0.3){i}\ar[dd]^-p&&X\times_SX\ar[dd]^(0.3){f\times f}\\
			&Y\ar@{-->}[rr]_(0.3){\delta}\ar@{==}[ld]&&Y\times_SY\ar@{=}[ld]\\
			Y\ar[rr]_-\delta&&Y\times_SY,
		}
	\end{gathered}
\end{align}
where $f'=f|_{U}$, $i'$ is the base change of $\delta$, $\delta'_0$ and $\delta'_1$ are the diagonal morphisms.
Then we have a commutative diagram
\begin{align}\small
	\begin{gathered}\label{eq:adjToDeltaKcubic2}
		\xymatrix@R=0.25cm@C=0.2cm{
			&j^\ast \delta_1^\ast i^\ast \mathcal T_S\otimes^L j^\ast f^\ast\delta^!\Lambda \ar[rr]\ar@{-->}[dd]\ar[ld]_-\simeq&&j^\ast \delta_1^\ast i^!\mathcal T_S\ar@{-->}[dd]\ar[ld]^-\simeq\ar[rr]&&j^\ast\delta_1^\ast\delta^\Delta\mathcal T_S\\
			\delta_1^{\prime\ast} i^{\prime\ast} \mathcal T'_S\otimes^L f^{\prime\ast}\delta^!\Lambda\ar[rr]\ar[dd]&& \delta_1^{\prime\ast}i^{\prime!}\mathcal T'_S\ar[dd]\ar[rr]&&\delta_1^{\prime\ast}\delta^\Delta\mathcal T_S^\prime\\
			&j^\ast \delta_0^\ast\mathcal T_S\otimes^L j^\ast f^\ast\delta^!\Lambda\ar@{-->}[ld]_-\simeq\ar@{-->}[rr]\ar@{-->}[dd]^(0.3){\rm ev}&& j^\ast\delta_0^\ast\mathcal T_S\ar[ld]^-\simeq\ar[dd]^(0.3){\rm ev}\\
			\delta_0^{\prime\ast}  \mathcal T'_S\otimes^L f^{\prime\ast}\delta^!\Lambda\ar[rr]\ar[dd]^-{\rm ev}&& \delta_0^{\prime\ast}\mathcal T'_S\ar[dd]^(0.3){\rm ev}\\
			&j^\ast\mathcal K_{X/S}\otimes^Lj^\ast f^\ast\delta^!\Lambda\ar@{-->}[ld]_-\simeq\ar@{-->}[rr]&&j^\ast\mathcal K_{X/S}\ar[ld]^-\simeq\ar[rr]&&j^\ast\mathcal K_{X/Y/S}\\
			\mathcal K_{U/S}\otimes^L f^{\prime\ast}\delta^!\Lambda\ar[rr]&& \mathcal K_{U/S}\ar[rr]&&\mathcal K_{U/Y/S},
		}
	\end{gathered}
\end{align}
where the upper cube is obtained by applying Lemma \ref{lem:cubicdiagram} to the diagram \eqref{eq:cubicSchidelta}, the lower cube is obtained by
applying the following diagram to $j^\ast f^\ast\delta^!\Lambda\simeq  f^{\prime\ast}\delta^!\Lambda\xrightarrow{\rm b.c}\Lambda$:
\begin{align}
	\begin{gathered}\label{eq:adjToDeltaKcubic3-0}
		\xymatrix@R=0.6cm{
			 \delta_0^{\prime\ast}  \mathcal T'_S\otimes^L -\ar[d]^-{\rm ev}&j^\ast\delta_0^\ast\mathcal T_S\otimes^L -\ar[l]\ar[d]^-{\rm ev}\\
			\mathcal K_{U/S}\otimes^L -&j^\ast\mathcal K_{X/S}\otimes^L -.\ar[l]
		}
	\end{gathered}
\end{align}
Note that the horizontal rows in the top and bottom faces of \eqref{eq:adjToDeltaKcubic2} are cofiber sequences (cf. \eqref{eq:Dchf}).
The diagram \eqref{eq:adjToDeltaKcubic2} further induces the following commutative diagram  (by the same reason as in \ref{subsec:LiftCofiber} or applying the dual version of \ref{subsec:LiftCofiber} twice):
\begin{align}\small
	\begin{gathered}\label{eq:adjToDeltaKcubic5}
		\xymatrix@R=0.2cm@C=0.2cm{
			&j^\ast\delta_1^\ast i^! \mathcal T_S \ar[rr]\ar@{-->}[dd]\ar[ld]_-\simeq&&j^\ast\delta_1^\ast \delta^\Delta\mathcal T_S\ar[dd]\ar[ld]^-\simeq\\
			\delta_1^{\prime\ast} i^{\prime!} \mathcal T'_S\ar[rr]\ar[dd]&& \delta_1^{\prime\ast}\delta^{\Delta}\mathcal T'_S\ar[dd]\\
			&j^\ast\mathcal K_{X/S}\ar@{-->}[ld]_-\simeq\ar@{-->}[rr]&&j^\ast\mathcal K_{X/Y/S}\ar[ld]^-\simeq\\
			\mathcal K_{U/S}\ar[rr]&&\mathcal K_{U/Y/S}.\\
		}
	\end{gathered}
\end{align}

\subsection{}\label{sec:ConstructT}
Using the 6-functor formalism $\mathcal D: {\rm Corr}^\otimes_S\to {\rm Cat}_\infty$, we  can extend the definitions 
of $C_{X/S}$ (cf. \eqref{eq:traceEquiforcharclass} and \eqref{eq:defCharClassInfinity}) and $C^Z_{X/Y/S}$ (cf. \ref{subsec:constructNAclass} and \eqref{eq:priciseDefCZXYS}) to objects in $\mathcal D_{\rm cons}(X)$ satisfying universal local acyclicity conditions\footnote{
	Let $\mathcal C_{S,\Lambda}^\otimes\to {\rm Corr}_S^\otimes$ be the 
	unstraightening of $\mathcal D: {\rm Corr}_S^\otimes\to {\rm Cat}_\infty$.
	A morphism $X\to S$ is universally locally acyclic relatively to 
	an object $\mathcal F\in\mathcal D_{\rm cons}(X)$ if and only if $(X,\mathcal F)$ is 
	dualisable in the symmetric 
	monoidal $\infty$-category $\mathcal C_{S,\Lambda}^\otimes$  (cf. \cite[Theorem 2.16]{LZ22} and \cite[Definition 3.2]{HS}).}.
We can also enhance the functor $\delta^\Delta$ (cf. \ref{subsec:defdeltaDelta}) and 
upgrade \eqref{eq:distriK} to be a cofiber sequence in $\mathcal D(X)$. 
Some diagrams in Section \ref{sec:TC}
and Section \ref{sec:CCC} can be upgraded to coherent commutative diagrams in stable $\infty$-categories, 
for example, \eqref{eq:propchangebasehom:TS2} and \eqref{bcdiag3}. 

In the following, we
assume the conditions (C1) and (C2) in \ref{subsec:notationforXYS} hold.
We  fix an object $\mathcal F\in \mathcal D_{\rm cons}(X)$ such that
$X\setminus Z\to Y$ is universally locally acyclic  relatively to $\mathcal F|_{X\setminus Z}$ and that $X\to S$  is universally locally acyclic  relatively to $\mathcal F$.
By the assumptions on $\mathcal F$, we have morphisms $
C_{X/S}(\mathcal F):\Lambda\to \mathcal K_{X/S}$ and $ C_{U/Y}(\mathcal F):\Lambda\to \mathcal K_{U/Y}$.
More precisely, the morphism $C_{X/S}(\mathcal F)$ is defined by
the following composition\footnote{In an $\infty$-category, the composition of two maps is well-defined up to a contractible space.}
\begin{align}\label{eq:defCharClassInfinity}
	\Lambda\xrightarrow{\rm id}R\mathcal Hom(\mathcal F,\mathcal F)\xrightarrow[\simeq]{\eqref{eq:keyisoCshrinkRHOM}}\delta_0^! \mathcal H_S\xleftarrow[\simeq]{\eqref{eq:Kunethiso}}\delta_0^!\mathcal T_S
	\xrightarrow{}\delta_0^\ast \mathcal T_S \xrightarrow{\rm ev} \mathcal K_{X/S}.
\end{align}
Let $C_{X/Y/S}(\mathcal F)$ be the following composition (cf. \eqref{bcdiag3})
\begin{align}\label{def:CXYNewdef}
	\Lambda\to \delta^!_0 \mathcal H_S\xleftarrow{\simeq}\delta_0^!\mathcal T_S
	\simeq \delta_1^!i^!\mathcal T_S\xrightarrow{}\delta_1^\ast i^!\mathcal T_S\to \delta_1^\ast \delta^\Delta \mathcal T_S \xrightarrow{\rm b.c}\delta^\Delta  \delta_0^\ast\mathcal T_S\xrightarrow{\rm ev}\delta^\Delta\mathcal K_{X/S}=\mathcal K_{X/Y/S},
\end{align}
which is  equivalent to the composition $\Lambda\xrightarrow{C_{X/S}(\mathcal F)}\mathcal K_{X/S}\xrightarrow{\eqref{eq:distriK}}\mathcal K_{X/Y/S}$.
Similar for  $C_{U/Y/S}(\mathcal F):=C_{U/Y/S}(\mathcal F|_U)$.  We have a commutative diagram
\begin{align}
	\begin{gathered}\label{eq:CXYSUhomotopyZero}
		\xymatrix@R=0.5cm@C=1.5cm{
			\Lambda\ar[r]\ar[d]_-{C_{X/Y/S}(\mathcal F)}&j_\ast\Lambda\ar[d]_-{ C_{U/Y/S}(\mathcal F|_U)}\ar[r]&j_\ast \delta_1^{\prime\ast} \delta^\Delta\mathcal T'_S\simeq 0\ar[ld]\\
			\mathcal K_{X/Y/S}\ar[r]&j_\ast\mathcal K_{U/Y/S},
		}
	\end{gathered}
\end{align}
where $\delta_1^{\prime\ast} \delta^\Delta\mathcal T'_S\simeq 0$ by \eqref{lem:eq:Gysim2}
since $U\to Y$ is universally locally acyclic relatively to $\mathcal F|_U$. 
Applying $\tau_\ast\tau^!\to {\rm id}\to j_\ast j^\ast$ to \eqref{def:CXYNewdef}, we get a commutative diagram
between cofiber sequences
\begin{align}
	\begin{gathered}\label{eq:TZXYcompaireNA}
		\xymatrix@R=0.5cm{
			\tau_\ast\tau^!\Lambda\ar[d]\ar[r]&\Lambda\ar[d]\ar[r]&j_\ast \Lambda\ar[d]\\
			\tau_\ast\tau^!\delta_1^\ast i^! \mathcal T_S \ar[r]\ar[d] &\delta_1^\ast i^! \mathcal T_S \ar[r]\ar[d]&j_\ast \delta_1^{\prime\ast} i^{\prime!} \mathcal T'_S\ar[d]\\
			\tau_\ast\tau^! \delta_1^\ast \delta^\Delta\mathcal T_S\ar[r]^-\simeq \ar[d]&\delta_1^\ast \delta^\Delta\mathcal T_S\ar[r]\ar[d]&j_\ast \delta_1^{\prime\ast}\delta^{\Delta}\mathcal T'_S\simeq 0\ar[d]\\
			\tau_\ast\tau^! \mathcal K_{X/Y/S}\ar[r] &\mathcal K_{X/Y/S}\ar[r]&j_\ast \mathcal K_{U/Y/S}.\\
		}
	\end{gathered}
\end{align}
Note that the non-acyclicity class $C^Z_{X/Y/S}(\mathcal F)$ is the composition (cf. \ref{subsec:constructNAclass})
\begin{align}\label{eq:priciseDefCZXYS}
	\Lambda\to  \delta^\ast_1 i^!\mathcal T_S\to \delta_1^\ast \delta^\Delta\mathcal T_S\xleftarrow{\simeq} \tau_\ast\tau^! \delta_1^\ast \delta^\Delta\mathcal T_S\to \tau_\ast\tau^! \mathcal K_{X/Y/S}
\end{align}
and \eqref{eq:TZXYcompaireNA} gives the following commutative diagram between cofiber sequences 
\begin{align}\label{eq:CZXYSlifting}
	\begin{gathered}
		\xymatrix@R=0.5cm{
			\Lambda\ar@{=}[r]\ar[d]_-{C^Z_{X/Y/S}(\mathcal F)}&\Lambda\ar[d]^-{C_{X/Y/S}(\mathcal F)}\ar[r]&0\ar[d]\\
			\tau_\ast \tau^!\mathcal K_{X/Y/S}\ar[r]&\mathcal K_{X/Y/S}\ar[r]&
			j_\ast\mathcal K_{U/Y/S}.
		}
	\end{gathered}
\end{align}
More precisely, it is given as follows
\begin{equation}\label{eq:CZXYSlifting--}
	\begin{gathered}
		\begin{tikzcd}[row sep=0.5cm]
			\Lambda\arrow[r,equal]\dar \arrow[dd, bend right=70, "C^Z_{X/Y/S}(\mathcal F)"'] &\Lambda\dar\rar&j_\ast\Lambda\arrow[rd]\arrow[rr]&&0\arrow[dd]\\
			\tau_\ast\tau^! \delta_1^\ast \delta^\Delta\mathcal T_S\arrow[r, "\simeq"] \dar&\delta_1^\ast \delta^\Delta\mathcal T_S\dar&&j_\ast \delta_1^{\prime\ast}\delta^{\Delta}\mathcal T'_S\arrow[dr]\arrow[ur, Rightarrow, shorten <=0.8em, shorten >=1em, "\gamma"']&\\
			\tau_\ast\tau^! \mathcal K_{X/Y/S}\rar&\mathcal K_{X/Y/S}\arrow[rrr]\arrow[urr, Rightarrow, shorten <=1em, shorten >=1em, "\beta"]&&&j_\ast \mathcal K_{U/Y/S},\\
		\end{tikzcd}
	\end{gathered}
\end{equation}
where $\beta$ is the homotopy defined by the right squares in \eqref{eq:TZXYcompaireNA}, 
$\gamma$ is the homotopy determined by the isomorphism $j_\ast \delta_1^{\prime\ast}\delta^{\Delta}\mathcal T'_S\simeq 0$.
\subsection{}In the following, let us construct the required map $L^Z_{X/Y/S}(\mathcal F)$.
We denote by ${\rm Map}(\Lambda,\mathcal K_{X/S})={\rm Map}_{\mathcal D(X)}(\Lambda,\mathcal K_{X/S})$ the mapping space whose $\pi_0$ is $H^0(X,\mathcal K_{X/S})$.
Since $\pi_0({\rm Map}(\Lambda,\mathcal K_{X/Y}))\simeq \pi_0({\rm Map}(\Lambda,\mathcal K_{U/Y}))$,
there is a morphism $C_{X/Y}(\mathcal F): \Lambda\to \mathcal K_{X/Y}$ with a commutative diagram
\begin{align}
	\begin{gathered}\label{eq:chooseCXY}
		\xymatrix@R=0.5cm{
			\Lambda\ar[d]_-{C_{X/Y}(\mathcal F)}\ar[r]&j_\ast\Lambda\ar[d]^-{C_{U/Y}(\mathcal F)}\\
			\mathcal K_{X/Y}\ar[r]& j_\ast\mathcal K_{U/Y}.
		}
	\end{gathered}
\end{align} 
Let $\delta^! C_{X/Y}(\mathcal F)$ be the composition
$\Lambda\xrightarrow{C_{X/Y}(\mathcal F)}\mathcal K_{X/Y}\xrightarrow{}\mathcal K_{X/S}$. 
We put $d_{X/Y/S}(\mathcal F)=C_{X/S}(\mathcal F)-\delta^! C_{X/Y}(\mathcal F)$. Consider the following two commutative diagrams in $\mathcal D(X)$:
\begin{align}\label{eq:backfacediff}
	\begin{gathered}
		\xymatrix@R=0.5cm{
			\Lambda\ar[r]\ar[d]_-{C_{X/S}(\mathcal F)}&j_\ast\Lambda\ar[d]^-{C_{U/S}(\mathcal F)}&&&\Lambda\ar[r]\ar[d]_-{\delta^!C_{X/Y}(\mathcal F)}&j_\ast\Lambda \ar[d]_-{\delta^!C_{U/Y}(\mathcal F)}\ar@{=}[r]&j_\ast\Lambda\ar[ld]^-{C_{U/S}(\mathcal F)}\\
			\mathcal K_{X/S}\ar[r]&j_\ast\mathcal K_{U/S},&&&	\mathcal K_{X/S}\ar[r]&j_\ast\mathcal K_{U/S},
		}
	\end{gathered}
\end{align}
where the rightmost triangle is given the commutative diagram \eqref{eq:propchangebasehom:TS2}.
By \eqref{eq:HomotopyZero}, we have a commutative diagram 
\begin{align}
	\begin{gathered}
		\xymatrix@R=0.5cm{
			\Lambda\ar[d]_-{d_{X/Y/S}(\mathcal F)}\ar[r]&0\ar[d]\\
			\mathcal K_{X/S}\ar[r]&
			j_\ast\mathcal K_{U/S}.\\
		}
	\end{gathered}
\end{align}
By \ref{subsec:LiftCofiber}, 
there exists a morphism $L^Z_{X/Y/S}(\mathcal F):\Lambda\to \tau_\ast\mathcal K_{Z/S}$  
making
\begin{align}\label{eq:TDXYS}
	\begin{gathered}
		\xymatrix@R=0.5cm@C=1.5cm{
			\Lambda\ar@{-->}[d]_-{L^Z_{X/Y/S}(\mathcal F)}\ar@{=}[r]&\Lambda\ar[d]^-{d_{X/Y/S}(\mathcal F)}\ar[r]&0\ar[d]\\
			\tau_\ast\mathcal K_{Z/S}\ar[r]&\mathcal K_{X/S}\ar[r]&
			j_\ast\mathcal K_{U/S}\\
		}
	\end{gathered}
\end{align}
into a commutative diagram in $\mathcal D(X)$. 
Even though the morphism  $L^Z_{X/Y/S}(\mathcal F):\Lambda\to \tau_\ast\mathcal K_{Z/S}$
may depends on the extension $C_{X/Y}(\mathcal F)$ of $C_{U/Y}(\mathcal F)$, we have the following result.
\begin{proposition}\label{lem:TisWellDefined}
Consider the notation in \ref{subsec:notationforXYS}. Let $\mathcal F\in \mathcal D_{\rm cons}(X)$ such that the conditions (C1)-(C3) in \ref{subsec:notationforXYS} hold.
Then we have a commutative diagram in $\mathcal D(X)$:
	\begin{align}\label{eq:CTDiffExplain}
		\begin{gathered}
			\xymatrix@R=0.5cm@C=2cm{
				&\Lambda \ar[ld]_-{C^Z_{X/Y/S}(\mathcal{F})} \ar[d]^-{L^Z_{X/Y/S}(\mathcal{F})}\ar@{=}[r]&\Lambda\ar[d]^-{C_{X/S}(\mathcal F)-\delta^!C_{X/Y}(\mathcal F)}\\
				\tau_*\tau^!\mathcal{K}_{X/Y/S}& \tau_*\mathcal{K}_{Z/S} \ar[l]\ar[r]&\mathcal K_{X/S}.
			}
		\end{gathered}
	\end{align}
	In particular, 
	the composition $\Lambda\xrightarrow{L^Z_{X/Y/S}(\mathcal F)} \tau_\ast\mathcal K_{Z/S}\to \tau_\ast\tau^!\mathcal K_{X/Y/S}$ is homotopic to $C^Z_{X/Y/S}(\mathcal F)$, which is independent of the choice of $C_{X/Y}(\mathcal F)$. 
	From \eqref{eq:CTDiffExplain}, we get the following equality:
	\begin{align}\label{eq:TequalDiff0}
		C_{X/S}(\mathcal F)=\delta^!C_{X/Y}(\mathcal F)+C^Z_{X/Y/S}(\mathcal F)\quad{\rm in}\quad H^0(X,\mathcal K_{X/S}).
	\end{align}
\end{proposition}
Theorem \ref{conj:milclass} follows directly from Proposition \ref{lem:TisWellDefined} by $D_{\rm ctf}(X,\Lambda)=h\mathcal D_{\rm cons}(X)$ and the fact that $h: X\to S$ is universally locally acyclic relatively to $\mathcal F\in \mathcal D_{\rm cons}(X)$ if and only if $h: X\to S$ is universally locally acyclic relatively to the image of $\mathcal F$ in $D_{\rm ctf}(X,\Lambda)$ (cf. \cite[Proposition 3.3]{HS} and \cite[1.2.4]{HTT}).
\begin{proof}
Consider the diagram \eqref{eq:cubicSchidelta}.
Applying  $\tau_\ast\tau^!\to{\rm id}\to j_\ast j^\ast$ to the following commutative diagram (cf. \eqref{bcdiag3})
\begin{align}\small\label{eq:initialKXYS}
\begin{gathered}
\xymatrix@R=0.4cm{
	&\Lambda\ar[d]\\
\delta_1^\ast\delta^\Delta\mathcal T_S\ar[d]&\delta_1^\ast i^!\mathcal T_S\ar[d]\ar[l]\\
\mathcal K_{X/Y/S}&\mathcal K_{X/S},\ar[l]
}
\end{gathered}
\end{align}
we get a  commutative diagram between cofiber sequences
\begin{align}\small
	\begin{gathered}\label{eq:adjToDeltaKcubic6}
		\xymatrix@R=0.4cm@C=0.4cm{
			&\tau_\ast\tau^!\Lambda\ar[d]\ar[rr]&&\Lambda\ar[d]\ar[rr]&&j_\ast \Lambda\ar[d]\\
			&\tau_\ast\tau^!\delta_1^\ast i^! \mathcal T_S \ar[rr]\ar@{-->}[dd]\ar[ld] &&\delta_1^\ast i^! \mathcal T_S \ar[rr]\ar@{-->}[dd]\ar[ld]&&j_\ast \delta_1^{\prime\ast} i^{\prime!} \mathcal T'_S\ar[dd]\ar[ld]\\
			\tau_\ast\tau^! \delta_1^\ast \delta^\Delta\mathcal T_S\ar[rr]^(0.7)\simeq \ar[dd]&&\delta_1^\ast \delta^\Delta\mathcal T_S\ar[rr]\ar[dd]&&j_\ast \delta_1^{\prime\ast}\delta^{\Delta}\mathcal T'_S\ar[dd]\\
			&\tau_\ast\mathcal K_{Z/S} \ar@{-->}[ld]\ar@{-->}[rr]&&\mathcal K_{X/S}\ar@{-->}[ld]\ar@{-->}[rr]&&j_\ast \mathcal K_{U/S}\ar[ld]\\
			\tau_\ast\tau^! \mathcal K_{X/Y/S}\ar[rr] &&\mathcal K_{X/Y/S}\ar[rr]&&j_\ast \mathcal K_{U/Y/S},\\
		}
	\end{gathered}
\end{align}
where the right diagram is obtained from the fact that
$j_\ast j^\ast\eqref{eq:initialKXYS}$ is isomorphic to the rightmost face by \eqref{eq:adjToDeltaKcubic5}. 
By \eqref{eq:adjToDeltaKcubic6}, we have a  commutative diagram
 \begin{align}\small\label{keydiagA}
 	\begin{gathered}
 		\xymatrix@R=0.7cm@C=3.5cm{
 			&\Lambda\ar@{-->}[d]^-{ C_{X/S}(\mathcal F)}\ar[r]\ar[ddl]_-{C_{X/Y/S}(\mathcal F)}&j_\ast\Lambda \ar[d]^-{C_{U/S}(\mathcal F)}\ar[ldd]_(0.8){C_{U/Y/S}(\mathcal F)}\\
 			&\mathcal K_{X/S}\ar@{-->}[r]\ar@{-->}[ld]&
 			j_\ast\mathcal K_{U/S}\ar[ld]\\
 			\mathcal K_{X/Y/S}\ar[r]&
 			j_\ast\mathcal K_{U/Y/S},
 		}
 	\end{gathered}
 \end{align}
where the homotopy in the top face is given by $\beta$ (cf. \eqref{eq:CZXYSlifting--}).

Now we construct a triangular prism for $\delta^!C_{X/Y}(\mathcal F)$.
As 0 is a final object, we can extend  \eqref{eq:chooseCXY} to a prism (unique up to a contractible space) 
\begin{align}\small\label{keydiagD}
	\begin{gathered}
		\xymatrix@R=0.7cm{
			&\Lambda\ar@{-->}[d]^-{C_{X/Y}(\mathcal F)}\ar[rr]\ar[ldd]&&j_\ast\Lambda\ar[d]^-{C_{U/Y}(\mathcal F )}\ar[ldd]\\
			&\mathcal K_{X/Y}\ar@{-->}[rr]\ar@{-->}[ld]&&j_\ast\mathcal K_{U/Y}\ar[ld]\\
			0\ar[rr]&&0.&
		}
	\end{gathered}
\end{align}

Now we have the following commutative diagram
\begin{align}\small\label{keydiagC2}
	\begin{gathered}
		\xymatrix@R=0.4cm{
			&\Lambda\ar@{-->}[d]^-{C_{X/Y}(\mathcal F)}\ar[rr]\ar[ldd]&&j_\ast\Lambda\ar[d]^-{C_{U/Y}(\mathcal F )}\ar[ldd]\\
			&\mathcal K_{X/Y}\ar@{-->}[rr]\ar@{-->}[ld]\ar@{-->}[dd]^(0.3){\delta^!}&&j_\ast\mathcal K_{U/Y}\ar[dd]^(0.5){\delta^!}\ar[ld]\\
			0\ar[rr]\ar[dd]&&0\ar[dd]&\\
			&\mathcal K_{X/S}\ar@{-->}[rr]\ar@{-->}[ld]&&j_\ast\mathcal K_{U/S}\ar[ld]\\
			\mathcal K_{X/Y/S}\ar[rr]&&j_\ast\mathcal K_{U/Y/S},
		}
	\end{gathered}
\end{align}
where the lower cube is obtained by applying ${\rm id}\to j_\ast j^\ast$ to its left face.
To compare the right face of \eqref{keydiagC2} with that of \eqref{keydiagA}, we construct the  commutative diagram \eqref{keydiagC5} below.
Note that the following two diagrams are both pull-back and push-out squares (cf. \eqref{eq:distriK} and \eqref{eq:Dchf}):
\begin{align}\small
\begin{gathered}
\xymatrix@R=0.4cm{
0\ar[d]&j_\ast\mathcal K_{U/Y}\ar[l]\ar[d]&0\ar[d]^-\simeq&j_\ast\delta^{\prime\ast}\mathcal T^\prime_Y\ar[d]\ar[l]\\
j_\ast\mathcal K_{U/Y/S}&j_\ast\mathcal K_{U/S},\ar[l]&j_\ast \delta_1^{\prime\ast}\delta^{\Delta}\mathcal T'_S&j_\ast \delta_1^{\prime\ast}i^{\prime!}\mathcal T'_S.\ar[l]
}
\end{gathered}
\end{align}
We obtain the following commutative cube by  taking the fiber of its front face  (cf. the rightmost square of \eqref{eq:adjToDeltaKcubic6})
\begin{align}\small\label{keydiagC5-2}
	\begin{gathered}
		\xymatrix@R=0.2cm@C=0.4cm{
			0\ar[rd]^(0.5)\simeq\ar[dd]&&j_\ast\delta_1^{\prime\ast}\mathcal T^{\prime}_Y\ar@{-->}[rd]\ar@{-->}[ll]\ar@{-->}[dd]\\
			&j_\ast \delta_1^{\prime\ast}\delta^{\Delta}\mathcal T'_S\ar[dd]&&j_\ast \delta_1^{\prime\ast}i^{\prime!}\mathcal T'_S\ar[dd]\ar[ll]\\
			0\ar[rd]&&j_\ast\mathcal K_{U/Y}\ar@{-->}[ll]\ar@{-->}[rd]^-{\delta^!}\\
			&j_\ast\mathcal K_{U/Y/S}&&j_\ast\mathcal K_{U/S},\ar[ll]
		}
	\end{gathered}
\end{align}
where the right face is given by the diagram \eqref{eq:propchangebasehom:TS2} for $U/Y/S$.
Since $0$ is a final object, there is a  commutative diagram (defined up to a contractible space)
\begin{align}\small\label{keydiagC5-1}
\begin{gathered}
\xymatrix@R=0.4cm@C=0.4cm{
&&&j_\ast\Lambda\ar[llld]\ar@{=}[rd]\ar@{-->}[d]\\
0\ar[rd]^-\simeq&&&j_\ast\delta_1^{\prime\ast}\mathcal T^{\prime}_Y\ar@{-->}[rd]\ar@{-->}[lll]&j_\ast\Lambda\ar[llld]\ar[d]\\
&j_\ast \delta_1^{\prime\ast}\delta^{\Delta}\mathcal T'_S&&&j_\ast \delta_1^{\prime\ast}i^{\prime!}\mathcal T'_S,\ar[lll]
}
\end{gathered}
\end{align} 
where the right face is given by \eqref{eq:propchangebasehom:TS2}.
Combining \eqref{keydiagC5-2} and \eqref{keydiagC5-1}, we get the following commutative diagram
\begin{align}\small\label{keydiagC5}
	\begin{gathered}
		\xymatrix@R=0.4cm@C=0.4cm{
&&j_\ast\Lambda\ar[lld]\ar@{=}[rd]\ar@{-->}[d]\\
0\ar[rd]^-\simeq\ar[dd]&&j_\ast\delta_1^{\prime\ast}\mathcal T^{\prime}_Y\ar@{-->}[rd]\ar@{-->}[ll]\ar@{-->}[dd]&j_\ast\Lambda\ar[lld]\ar[d]\\
			&j_\ast \delta_1^{\prime\ast}\delta^{\Delta}\mathcal T'_S\ar[dd]&&j_\ast \delta_1^{\prime\ast}i^{\prime!}\mathcal T'_S\ar[dd]\ar[ll]\\
			0\ar[rd]&&j_\ast\mathcal K_{U/Y}\ar@{-->}[ll]\ar@{-->}[rd]^-{\delta^!}\\
			&j_\ast\mathcal K_{U/Y/S}&&j_\ast\mathcal K_{U/S},\ar[ll]
		}
	\end{gathered}
\end{align}  
where the composition $j_\ast\Lambda\to j_\ast\delta_1^{\prime\ast}\mathcal T^{\prime}_Y\to j_\ast\mathcal K_{U/Y}$ is equal to $j_\ast\Lambda \xrightarrow{C_{U/Y}(\mathcal F)}j_\ast\mathcal K_{U/Y}$, the composition $j_\ast\Lambda\to  j_\ast \delta_1^{\prime\ast}i^{\prime!}\mathcal T'_S\to j_\ast\mathcal K_{U/S}$ is equal to $j_\ast \Lambda\xrightarrow{C_{U/S}(\mathcal F)}j_\ast\mathcal K_{U/S}$ and
 the right face of \eqref{keydiagC5} is given by  Theorem \ref{thm:changebaseHomo} (cf. \eqref{eq:propchangebasehom:TS2}).
Note  that the composition $j_\ast\Lambda\to  j_\ast \delta_1^{\prime\ast}\delta^{\Delta}\mathcal T'_S\to j_\ast\mathcal K_{U/Y/S}$ on the front face of \eqref{keydiagC5} is the morphism $j_\ast\Lambda\xrightarrow{C_{U/Y/S}(\mathcal F)}j_\ast\mathcal K_{U/Y/S}$ (cf. \eqref{def:CXYNewdef}). 
The diagrams \eqref{keydiagC2} and \eqref{keydiagC5} define the following commutative diagram
\begin{align}\small\label{keydiagC6}
	\begin{gathered}
		\xymatrix@R=0.4cm{
			&\Lambda\ar@{-->}[d]^-{C_{X/Y}(\mathcal F)}\ar[rr]\ar[ldd]&&j_\ast\Lambda\ar@{-->}[d]^-{C_{U/Y}(\mathcal F )}\ar[ldd]\ar@{=}[r]&j_\ast\Lambda \ar@/^3pc/[lddd]^(0.4){C_{U/S}(\mathcal F)}\ar@/^0.5pc/[lldddd]^(0.4){C_{U/Y/S}(\mathcal F)}\\
			&\mathcal K_{X/Y}\ar@{-->}[rr]\ar@{-->}[ld]\ar@{-->}[dd]^(0.3){\delta^!}&&j_\ast\mathcal K_{U/Y}\ar@{-->}[dd]_(0.3){\delta^!}\ar@{-->}[ld]\\
			0\ar[rr]\ar[dd]&&0\ar[dd]&\\
			&\mathcal K_{X/S}\ar@{-->}[rr]\ar@{-->}[ld]&&j_\ast\mathcal K_{U/S}\ar[ld]\\
			\mathcal K_{X/Y/S}\ar[rr]&&j_\ast\mathcal K_{U/Y/S}.
		}
	\end{gathered}
\end{align}
It gives the following commutative diagram 
 \begin{align}\small\label{keydiagB}
 	\begin{gathered}
 		\xymatrix@R=0.6cm@C=3.5cm{
 			&\Lambda\ar@{-->}[d]^-{ \delta^!C_{X/Y}(\mathcal F)}\ar[r]\ar[ddl]_-0&j_\ast\Lambda \ar[d]^-{C_{U/S}(\mathcal F) }\ar[ldd]_(0.8){C_{U/Y/S}(\mathcal F)}\\
 			&\mathcal K_{X/S}\ar@{-->}[r]\ar@{-->}[ld]&
 			j_\ast\mathcal K_{U/S}\ar[ld]\\
 			\mathcal K_{X/Y/S}\ar[r]&
 			j_\ast\mathcal K_{U/Y/S},
 		}
 	\end{gathered}
 \end{align}
where the right face is equal to that of \eqref{keydiagA}.
By construction, the homotopy on the top face of \eqref{keydiagB} is given as follows
\begin{equation}\small\label{keydiagB2}
	\begin{gathered}
		\begin{tikzcd}[row sep=0.2cm]
			\Lambda\rar \dar&j_\ast\Lambda\ar[rd]\arrow[rr]&&j_\ast \delta_1^{\prime\ast}\delta^{\Delta}\mathcal T'_S\arrow[dd]\\
			0\dar \arrow[rr, Rightarrow, shorten <= 2em, shorten >= 2em, "0"]&&0\arrow[dr]\arrow[ur, Rightarrow, shorten <= 0.5em, shorten >= 0.5em, near end, "\gamma^{-1}"']&\\
			\mathcal K_{X/Y/S}\arrow[rrr]&&&j_\ast \mathcal K_{U/Y/S},\\
		\end{tikzcd}
	\end{gathered}
\end{equation}
where $\gamma$ is defined in \eqref{eq:CZXYSlifting--}.
Applying \eqref{eq:HomotopyZero} to \eqref{keydiagA} and \eqref{keydiagB}
  first, and then using \ref{subsec:LiftCofiber} together with \eqref{eq:CZXYSlifting} and \eqref{eq:TDXYS}, we get a  commutative diagram  
 \begin{align}\small 
 	\begin{gathered}\label{finalSquare}
 		\xymatrix@R=0.6cm@C=2cm{
 			&\Lambda\ar@{=}[r]\ar@{-->}[d]^-{L^Z_{X/Y/S}(\mathcal F)}\ar[ddl]_-{C^Z_{X/Y/S}(\mathcal F)}&\Lambda\ar[r]\ar@{-->}[d]^-{d_{X/Y/S}(\mathcal F)}\ar[ddl]_(0.8){C_{X/Y/S}(\mathcal F)}&0\ar[d]\ar[ldd] \\
 			&\tau_*\mathcal{K}_{Z/S} \ar@{-->}[r] \ar@{-->}[ld]&\mathcal{K}_{X/S} \ar@{-->}[r]\ar@{-->}[ld] & j_*\mathcal{K}_{U/S}\ar[ld]\\
 			\tau_*\tau^!\mathcal{K}_{X/Y/S} \ar[r] &\mathcal{K}_{X/Y/S} \ar[r] & j_*\mathcal{K}_{U/Y/S}.\\
 		}
 	\end{gathered}
 \end{align}

 Now we show that the top face of \eqref{finalSquare} is equal to \eqref{eq:CZXYSlifting}. 
 It is enough to compare their right squares.
 Let $E$ be the right square on the top face of \eqref{finalSquare}. Consider the notation $\gamma$ and $\beta$ in \eqref{eq:CZXYSlifting--}. Since the homotopies on the top faces of \eqref{keydiagA} and \eqref{keydiagB}  are given by $\beta$ and $\gamma^{-1}$ respectively, 
 the homotopy in $E$  is equal to the composition of $\beta$ and $\gamma=(\gamma^{-1})^{-1}$. Thus $E$ is equal to the right square of \eqref{eq:CZXYSlifting} by \eqref{eq:CZXYSlifting--}. 
We remark that the back face of \eqref{finalSquare} is equal to \eqref{eq:TDXYS} since
 the back face of \eqref{keydiagA} (resp. \eqref{keydiagB}) is equal to the left (resp. right) diagram in \eqref{eq:backfacediff}.
Finally, the diagram  \eqref{finalSquare} gives the required diagram \eqref{eq:CTDiffExplain}. This finishes the proof.
\end{proof}

\section{Applications: Milnor formula, conductor formula and Saito's conjecture}\label{sec:applications}

\subsection{}\label{subsec:LocBC}
In this section, we assume $k$ is a perfect field of characteristic $p>0$.
Let $\Lambda$ be a finite local ring whose residue field is of characteristic $\ell\neq p$.
We first apply the pull-back property of the non-acyclicity classes to prove a cohomological Milnor formula.

\begin{theorem}\label{thm:MF}
Let $Y$ be a smooth curve over $k$.  
Let    $f: X\to Y$ be a separated morphism of finite type and $x\in|X|$ a closed point. Let $\mathcal F\in D_{\rm ctf}(X,\Lambda)$ such that $f|_{X\setminus\{x\}}$ is universally locally acyclic relatively to $\mathcal F|_{X\setminus\{x\}}$. 
Then we have 
\begin{align}\label{eq:thm:MF1}
C_{X/Y/k}^{\{x\}}(\mathcal F)=-{\rm dimtot} R\Phi_{\bar{x}}(\mathcal F,f)\quad {\rm in}\quad  \Lambda=H^0_{x}(X,\mathcal K_{X/k}),
\end{align}
where $R\Phi(\mathcal F,f)$ is the complex of vanishing cycles and ${\rm dimtot}={\rm dim}+{\rm Sw}$ is the total dimension.
\end{theorem}
The above theorem gives a cohomological analogue of the Milnor formula proved by Saito in \cite[Theorem 5.9]{Sai17a}. Thanks to the cohomological argument, we don't need the smoothness assumption on $X$.
Before the proof, let us recall Abe's reformulation of Laumon's results on local Fourier transform (cf. \cite[6.3]{Abe21}).
We choose a non-trivial additive character $\psi:\mathbb F_p\to \Lambda^\ast$, and 
let $\mathcal L$ be the Artin-Schreier sheaf on $\mathbb A^1_{\mathbb F_p}$ associated with $\psi$.
For any morphism $g: X\to \mathbb A^1_{\mathbb F_p}$, we denote by $\mathcal L(g)$
the pullback $g^\ast\mathcal L$.
\begin{theorem}[Laumon]\label{thm:Lau}
	Let $T=(\mathbb A_k^1)_{(0)}$ $($resp. $T'=(\mathbb A_k^1)_{(0')})$ be
	the henselization with coordinate $t$ $($resp. $t')$. Let $T\xleftarrow{{\rm pr}_1}T\times T'\xrightarrow{{\rm pr}_2} T'$ be the projections. 
	For any $\mathcal F\in D_{\rm ctf}(T,\Lambda)$,
	we have
	\begin{align}\label{eq:thm:Lau}
		-{\rm dimtot}(R\Phi_{\bar 0}(\mathcal F,id))={\rm dim} \Psi_{{\rm pr}_2}({\rm pr}_1^\ast\mathcal F\otimes^L \mathcal L_!(t/t'))_{(\bar 0,\bar 0')},
	\end{align}
	where $\mathcal L_!(t/t')$ is the zero extension to $T\times T'$ of $\mathcal L(t/t')$ on $T\times T'\setminus T\times\{0'\}$.
\end{theorem}
\begin{proof}
	For convenience, we record the proof given in \cite[6.3]{Abe21}. 
	By applying the reduction map, we may assume that $\Lambda$ is a field. Let $i:\{0\}\hookrightarrow T$ be the closed point and $j:\eta\hookrightarrow T$ the generic point. By \cite[5.2.6]{UYZ}, there is a finite extension $E$ of $\mathbb Q_\ell$ with residue field $\Lambda$ and a lift $\widetilde{\mathcal F}_\eta\in D_c^b(\eta, E)$ such that the reduction of 
	$\widetilde{\mathcal F}_\eta$ is equal to $\mathcal F_\eta:=j^\ast\mathcal F$ in the Grothendieck group $K_0(\eta,\Lambda)$.
	
	Let $\mathfrak F^{(0,\infty')}(-)$ be the functor of local Fourier transform defined in \cite[2.4.2.3]{Lau87}.
	By \cite[Proposition 2.4.2.2 and Th\'eor\`eme 2.4.3]{Lau87}, we have 
	\begin{align}
		\begin{split}
			{\rm dim} \Psi_{{\rm pr}_2}({\rm pr}_1^\ast\mathcal F\otimes^L \mathcal L_!(t/t'))_{(\bar 0,\bar 0')}&={\rm dim} \Psi_{{\rm pr}_2}({\rm pr}_1^\ast j_!j^\ast \mathcal F\otimes^L \mathcal L_!(t/t'))_{(\bar 0,\bar 0')}\\
			&\qquad\qquad \qquad+{\rm dim} \Psi_{{\rm pr}_2}({\rm pr}_1^\ast i_\ast i^\ast \mathcal F\otimes^L \mathcal L_!(t/t'))_{(\bar 0,\bar 0')}\\
			&=-{\rm dim}\mathfrak F^{(0,\infty')}(\widetilde{\mathcal F}_\eta)+{\rm dim}\mathcal F_{\bar 0}\\
			&=-{\rm dimtot}(\widetilde{\mathcal F}_\eta)+{\rm dim}\mathcal F_{\bar 0}\\
			&=-{\rm dimtot}(R\Phi_{\bar 0}(\mathcal F,id)).
		\end{split}
	\end{align}
	This finishes the proof.
\end{proof}
\subsection{Proof of Theorem \ref{thm:MF}}
Let $Z=\{x\}$. Since $C^Z_{X/Y/k}(\mathcal F)$ depends only on a neighborhood of $Z$ by Proposition \ref{prop:invariantUnderEtale}, we may shrink $X$ and $Y$ for calculating $C_{X/Y/k}^{Z}(\mathcal F)$. Thus we may assume there is an \'etale morphism $Y\to \mathbb A_k^1$. Since $C_{X/Y/k}^{Z}(\mathcal F)=C_{X/\mathbb A_k^1/k}^{Z}(\mathcal F)$ by Proposition \ref{prop:invariantUnderEtale}, we may reduce the proof to the case where $Y=\mathbb A_k^1$. 
	Since $Z=\{x\}$, in order to calculate $C_{X/Y/k}^{Z}(\mathcal F)$,
	we may further assume that $k$ is algebraically closed.

Now we use the strategy in \cite[Theorem 1.6]{Abe22}.  
Let $ft': X\times \mathbb A^1_{t'}\xrightarrow{f\times{\rm id}}  \mathbb A^1_{t}\times  \mathbb A^1_{t'}\xrightarrow{\mu} \mathbb A^1$, where $\mu$ is the multiplication morphism.
We choose a non-trivial additive character $\psi:\mathbb F_p\to \Lambda^\ast$, and 
let $\mathcal L$ be the Artin-Schreier sheaf on $\mathbb A^1$ associated with $\psi$.
Let $\mathcal L_!(\mu)$ be the zero extension to $\mathbb A^1_t\times \mathbb P^1_{t'}$ of  $\mu^\ast\mathcal L$.
Let $\mathcal L_!(ft')$ be the zero extension to $X\times \mathbb P^1_{t'}$ of  $(ft')^\ast\mathcal L$.
	
We first show that ${\rm pr}_2: (X\times \mathbb P^1_{t'})\setminus(Z\times\infty)\to \mathbb P^1_{t'}$ is universally locally acyclic relatively to ${\rm pr}_1^\ast\mathcal F\otimes^L \mathcal L_!(ft')$.
By \cite[Th\'eor\`eme 1.3.1.2]{Lau87}, ${\rm pr}_2: \mathbb A_t^1\times\mathbb P^1_{t'}\to \mathbb P^1_{t'}$ is universally locally acyclic relatively to $\mathcal L_!(\mu)$.
By \cite[Proposition 1.4]{Sai17b}, ${\rm pr}_2: (X\times \mathbb P^1_{t'})\setminus(Z\times\mathbb P^1_{t'})\to \mathbb P^1_{t'}$ is universally locally acyclic relatively to ${\rm pr}_1^\ast\mathcal F\otimes^L \mathcal L_!(ft')$. Since $\mathcal L_!(\mu)$ is locally constant on $\mathbb A^1_t\times \mathbb A^1_{t'}$, the morphism ${\rm pr}_2: (X\times \mathbb P^1_{t'})\setminus(X\times\infty)\to \mathbb P^1_{t'}$ is universally locally acyclic relatively to ${\rm pr}_1^\ast\mathcal F\otimes^L \mathcal L_!(ft')$. Thus ${\rm pr}_2: (X\times \mathbb P^1_{t'})\setminus(Z\times\infty)\to \mathbb P^1_{t'}$ is universally locally acyclic relatively to ${\rm pr}_1^\ast\mathcal F\otimes^L \mathcal L_!(ft')$.

We claim that there exists a pair ($\mathbb P'$, $\widetilde{\mathcal F}$) satisfying the following conditions:\footnote{When $\Lambda$ is a finite field and $\mathcal F$ is perverse, the claim also follows from \cite[Corollary 1.2.3]{Sai21}.}
\begin{itemize}
	\item $\mathbb P'\to \mathbb P^1_{t'}$ is a finite surjective morphism with  $\mathbb P'$  integral.
	\item $\widetilde{\mathcal F}\in D_{\rm ctf}(X\times\mathbb P',\Lambda)$ is an extension of $({{\rm pr}_1^{\prime\ast}\mathcal F\otimes^L \mathcal L_!(ft')})|_{X\times(\mathbb A^1_{t'}\times_{\mathbb P^1_{t'}}\mathbb P')}$ to $X\times\mathbb P'$.
	\item The projection $X\times \mathbb P'\to \mathbb P'$ is universally locally acyclic relatively  to $\widetilde{\mathcal F}$.
\end{itemize}
Indeed, this follows from \cite[Lemma 1.5]{Abe22}. Here we include the proof for convenience.
Let $\xi$ be an algebraic geometric generic point of $\mathbb P^1_{t'}$. Let $\overline{\mathbb P^1_{t'}}$ be the normalization of $\mathbb P^1_{t'}$ in $\xi$. 
Let $j: X\times \xi\to X\times \overline{\mathbb P^1_{t'}}$ be the canonical morphism. 
Since $\overline{\mathbb P^1_{t'}}$ is absolutely integrally closed, the morphism ${\rm pr}_2: X\times \overline{\mathbb P^1_{t'}}\to \overline{\mathbb P^1_{t'}}$ is universally locally acyclic relatively to $Rj_\ast(({\rm pr}_1^\ast\mathcal F\otimes^L \mathcal L_!(ft'))|_{X\times \xi})$ (cf. \cite[Corollary 3.10 and Theorem 4.1]{HS}). 
Since $Rj_\ast(({\rm pr}_1^\ast\mathcal F\otimes^L \mathcal L_!(ft'))|_{X\times \xi})\in D_{\rm ctf}(X\times\overline{\mathbb P^1_{t'}},\Lambda)$ (cf. \cite[Corollary 4.2]{HS}), 
there is a finite surjective morphism $T\to \mathbb P^1_{t'}$ and $\mathcal G\in D_{\rm c}(X\times T,\Lambda)$ such that $\overline{\mathbb P^1_{t'}}\to \mathbb P^1_{t'}$ factors over $T\to \mathbb P^1_{t'}$ 
and $\mathcal G|_{X\times \overline{\mathbb P^1_{t'}}}\simeq Rj_\ast(({\rm pr}_1^\ast\mathcal F\otimes^L \mathcal L_!(ft'))|_{X\times \xi})$.
Since $\overline{\mathbb P^1_{t'}}\to T$ is surjective, $\mathcal G$ is also of finite tor-dimension and thus $\mathcal G\in D_{\rm ctf}(X\times T,\Lambda)$.
Since ${\rm pr}_2: (X\times \mathbb P^1_{t'})\setminus(Z\times\infty)\to \mathbb P^1_{t'}$ is universally locally acyclic relatively to ${\rm pr}_1^\ast\mathcal F\otimes^L \mathcal L_!(ft')$, the canonical morphism \[
\phi: ({{\rm pr}_1^{\prime\ast}\mathcal F\otimes^L \mathcal L_!(ft')})|_{X\times (\mathbb A^1_{t'}\times_{\mathbb P_{t'}^1}\overline{\mathbb P^1_{t'}})}\to Rj_\ast(({\rm pr}_1^\ast\mathcal F\otimes^L \mathcal L_!(ft'))|_{X\times \xi})|_{X\times (\mathbb A^1_{t'}\times_{\mathbb P_{t'}^1}\overline{\mathbb P^1_{t'}})}
\simeq \mathcal G|_{X\times (\mathbb A^1_{t'}\times_{\mathbb P_{t'}^1}\overline{\mathbb P^1_{t'}})}\]
is an isomorphism. 
Then there is a finite surjective morphism $\mathbb P'\to T$ such that $\phi$ descends to an isomorphism ${{\rm pr}_1^{\prime\ast}\mathcal F\otimes^L \mathcal L_!(ft')}|_{X\times (\mathbb A^1_{t'}\times_{\mathbb P_{t'}^1}\mathbb P')}\to \mathcal G|_{X\times (\mathbb A^1_{t'}\times_{\mathbb P_{t'}^1}\mathbb P')}$. Put $\widetilde{\mathcal F}:=\mathcal G|_{X\times\mathbb P'}$.
Since the morphism ${\rm pr}_2: X\times \overline{\mathbb P^1_{t'}}\to \overline{\mathbb P^1_{t'}}$ is universally locally acyclic relatively to $Rj_\ast(({\rm pr}_1^\ast\mathcal F\otimes^L \mathcal L_!(ft'))|_{X\times \xi})$, the projection $X\times \mathbb P'\to \mathbb P'$ is universally locally acyclic relatively  to $\widetilde{\mathcal F}$ by \cite[Theorem 3.7]{HS}.
Thus $\mathbb P'$ and $\widetilde{\mathcal F}$ satisfy the required conditions.

We fix geometric points $0'$ and $\infty'$ of $\mathbb P'$ over $0$ and $\infty$ of $\mathbb P^1_{t'}$ respectively. 
Let $\eta'$ be a geometric generic point of $\mathbb P^\prime$. Let $\Psi_{{\rm pr}^{\prime}_2}:=\Psi_{{\rm pr}^\prime_{2},\infty'\leftarrow \eta'}$ (cf. \eqref{eq:NCFunctordef}).
We have $\widetilde{\mathcal F}|_{X\times\infty'}\simeq\Psi_{\rm pr'_2}(\widetilde{\mathcal F})\simeq \Psi_{\rm pr'_2}({{\rm pr}_1^{\prime\ast}\mathcal F\otimes^L \mathcal L_!(ft')})$, which is supported on $Z$. 
We also have  $\widetilde{\mathcal F}|_{X\times 0'}\simeq \mathcal F$.	
Now we apply Proposition \ref{prop:LocBC} and Proposition \ref{prop:spOfNA} to the following diagram
	\begin{align}
		\begin{gathered}
			\xymatrix@R=0.5cm{
				Z\times\mathbb P'\ar@{^(->}[r]&X\times \mathbb P'\ar[r]^-{f\times {\rm id}}\ar[d]&Y\times\mathbb P'.\ar[ld]\\
				&\mathbb P'
			}
		\end{gathered}
	\end{align}
	Let $V=\mathbb P'\times_{\mathbb P^1_{t'}}(\mathbb P^1_{t'}\setminus\{\infty\})$.
	Note that $(X\times V)\setminus(Z\times V)\to Y\times V$ is universally locally acyclic relatively to $\widetilde{\mathcal F}|_{X\times V}$ and $(X\times \infty')\setminus(Z\times \infty')\to Y\times \infty'$ is universally locally acyclic relatively to $\Psi_{\rm pr'_2}(\widetilde{\mathcal F})$.
	Consider the following pull-back morphism and specialization morphism
	\begin{align}
		H^0(Z\times 0',\mathcal K_{Z\times 0'/0'})\xleftarrow[\simeq]{{\rm res}_0}
		H^0(Z\times V,\mathcal K_{Z\times V/V})\xrightarrow[\simeq]{{\rm sp}_{\infty}}
		H^0(Z\times \infty',\mathcal K_{Z\times \infty'/\infty'}).
	\end{align}
These groups are canonically identified with $\Lambda$.
We have ${\rm res}_0(C^{Z\times V}_{X\times V/Y\times V/ V}(\widetilde{\mathcal F}))={\rm sp}_{\infty}(C^{Z\times V}_{X\times V/Y\times V/ V}(\widetilde{\mathcal F}))$. 
By Proposition \ref{prop:LocBC} and Proposition \ref{prop:spOfNA}, we have
	\begin{align}
		{\rm res}_0(C^{Z\times V}_{X\times V/Y\times V/ V}(\widetilde{\mathcal F}))&
		\overset{\eqref{eq:prop:LocBC1}}{=}C^{Z\times {0'}}_{X\times 0'/Y\times 0'/k}(\widetilde{\mathcal F}|_{X\times 0'})=C^{Z}_{X/Y/k}(\mathcal F),\\
		{\rm sp}_\infty(C^{Z\times V}_{X\times V/Y\times V/ V}(\widetilde{\mathcal F}))&
		\overset{\eqref{spOfNAinpropCor}}{=}C^{Z\times {\infty'}}_{X\times \infty'/Y\times \infty'/k}(\widetilde{\mathcal F}|_{X\times \infty'})=C^{Z}_{X/Y/k}(\Psi_{\rm pr'_2}({{\rm pr}_1^{\prime\ast}\mathcal F\otimes^L \mathcal L_!(ft')})).
	\end{align}
	Since $\Psi_{\rm pr'_2}({{\rm pr}_1^{\prime\ast}\mathcal F\otimes^L \mathcal L_!(ft')})$ is supported on $Z=\{x\}$,  we have
	\begin{align}\label{pf:CMdimtotT}
		\begin{split}
			C^{Z}_{X/Y/k}(\Psi_{\rm pr'_2}({{\rm pr}_1^{\prime\ast}\mathcal F\otimes^L \mathcal L_!(ft')}))=\dim \Psi_{\rm pr'_2}({{\rm pr}_1^{\prime\ast}\mathcal F\otimes^L \mathcal L_!(ft')})_{\bar{x}}=- {\rm dimtot}R\Phi_{\bar{x}}(\mathcal F,f).
		\end{split}
	\end{align}
	The last equality in \eqref{pf:CMdimtotT} follows from \cite[Proposition 6.5.2]{Abe21}. For convenience, we include the proof here. 
	Let $T=(\mathbb P^1_{t'})_{(\bar \infty)}$ be the strict henselization.  We consider the following commutative diagram
	\begin{align}
		\begin{split}
			\xymatrix@R=0.5cm@C=2.2cm{
				(X\times \mathbb P^1_{t'})_{(\bar x,\bar \infty)}\ar[d]\ar[r]^-{f'=(f\times {\rm id})_{(\bar x,\bar \infty)}}&(\mathbb A_k^1\times T)_{(f(\bar x),\bar\infty)}\ar[r]^-\pi \ar[d]^-q&T\\
				X_{(\bar x)}\ar[r]^-{f_{(\bar x)}}&(\mathbb A_k^1)_{f(\bar x)}.
			}
		\end{split}
	\end{align}
	Let $\eta$ be the generic point of $T$.
	Then we have
	\begin{align}\label{eq:CalPSIT}
		\begin{split}
			\Psi_{\rm pr'_2}({{\rm pr}_1^{\prime\ast}\mathcal F\otimes^L \mathcal L_!(ft')})_{\bar{x}}&\overset{(a)}{=}R\Gamma((X\times \mathbb P^1_{t'})_{(\bar x,\bar \infty)}\times_T\bar\eta,{{\rm pr}_1^{\prime\ast}\mathcal F\otimes^L \mathcal L_!(ft')})\\
			&\overset{(b)}{=}(R\pi_{\eta\ast} Rf'_{\eta\ast}({{\rm pr}_1^{\prime\ast}\mathcal F\otimes^L \mathcal L_!(ft')}))_{\bar\eta}\\
			&=(R\pi_{\eta\ast} Rf'_{\eta\ast}({{\rm pr}_1^{\prime\ast}\mathcal F\otimes^L f^{\prime\ast}_\eta\mathcal L_!(tt')}))_{\bar\eta} \\
			&\overset{(c)}{=}(R\pi_{\eta\ast} (Rf'_{\eta\ast}({\rm pr}_1^{\prime\ast}\mathcal F)\otimes^L \mathcal L_!(tt')))_{\bar\eta}\\
			&\overset{(d)}{=}(R\pi_{\eta\ast}(q^\ast Rf_{(\bar x)\ast}(\mathcal F)\otimes \mathcal L_!(tt')))_{\bar \eta}\\
			&\overset{(e)}{=}\Psi_{\pi,\bar\infty\leftarrow \bar \eta}(q^\ast Rf_{(\bar x)\ast}(\mathcal F)\otimes \mathcal L_!(tt')),
		\end{split}
	\end{align}
	where both $(a)$ and (e) follow from the definition of nearby cycles functor \eqref{eq:NCFunctordef}, (b) follows from the smooth base change, (c) follows from the fact that $f^{\prime\ast}_\eta\mathcal L_!(tt')$ is a locally constant constructible flat $\Lambda$-modules, (d) follows from the $\Psi$-goodness (cf. \cite[Appendix, A.2(1)]{Ill17}).
	Finally, we have
	\begin{align}\label{eq:dimtotForPhiPSIT}
		\begin{split}
			\dim \Psi_{\rm pr'_2}({{\rm pr}_1^{\prime\ast}\mathcal F\otimes^L \mathcal L_!(ft')})_{\bar{x}}&\overset{\eqref{eq:CalPSIT}}{=}\dim \Psi_{\pi,\bar\infty\leftarrow \bar \eta}(q^\ast Rf_{(\bar x)\ast}(\mathcal F)\otimes \mathcal L_!(tt'))\\
			&\overset{\eqref{eq:thm:Lau}}{=}-{\rm dimtot}R\Phi(Rf_{(\bar x)\ast}(\mathcal F),id)=-{\rm dimtot}R\Phi_{\bar x}(\mathcal F,f).
		\end{split}
	\end{align}
	This finishes the proof of Theorem \ref{thm:MF}.

\begin{theorem}\label{thm:SJ:MFforTclass}
	Let $Y$ be a smooth connected curve over  $k$.  
	Let $f:X\to Y$ be a separated morphism of finite type and $Z\subseteq |X|$ be a finite set of closed points. 
	Let $\mathcal F\in D_{\rm ctf}(X,\Lambda)$ such that $f|_{X\setminus Z}$ is universally locally acyclic relatively to $\mathcal F|_{X\setminus Z}$. 
	Then we have 
	\begin{align}\label{eq:NewT}
		C_{X/k}(\mathcal F)=c_1(f^\ast\Omega_{Y/k}^{1,\vee})\cap C_{X/Y}(\mathcal F)-\sum_{x\in Z}{\rm dimtot} R\Phi_{\bar{x}}(\mathcal F,f)\cdot [x]\quad{\rm in}\quad H^0(X,\mathcal K_{X/k}).
	\end{align}
\end{theorem}
\begin{proof}
By \eqref{eq:milclass1-1}, we have
	\begin{align}\label{pfeq:milclass1-1}
			C_{X/k}(\mathcal F)=c_1(f^\ast\Omega^{1,\vee}_{Y/k})\cap C_{X/Y}(\mathcal F)+
			C_{X/Y/k}^{Z}(\mathcal F)\quad{\rm in}\quad H^0(X,\mathcal K_{X/k}).
	\end{align}
Now the equality \eqref{eq:NewT} follows from \eqref{eq:thm:MF1}.
\end{proof}

The following corollary gives a cohomological version of the Grothendieck-Ogg-Shafarevich formula:
\begin{corollary}\label{cor:GOSend}
	Let $X$ be a smooth connected curve over $k$. 
	Let $\mathcal F\in D_{\rm ctf}(X,\Lambda)$ and $Z\subseteq X$ be a finite set of closed points such that the cohomology sheaves of $\mathcal F|_{X\setminus Z}$ are locally constant. Then we have
	\[
	C_{X/k}(\mathcal F)={\rm rank}\mathcal F\cdot c_1(\Omega_{X/k}^{1,\vee})-\sum_{x\in Z}{a}_x(\mathcal F)\cdot[x]\quad {\rm in}\quad H^0(X,\mathcal K_{X/k}),
	\]
	where $a_x(\mathcal F)={\rm rank}\mathcal F|_{\bar\eta}-{\rm rank}\mathcal F_{\bar x}+{\rm Sw}_x\mathcal F$ is the Artin conductor and $\eta$ is the generic point of $X$.
\end{corollary}
\begin{proof}
	This follows from Theorem \ref{thm:SJ:MFforTclass} together with the fact that
	${\rm dimtot}R\Phi_{\bar x}(\mathcal F,{\rm id})=a_x(\mathcal F)$.
\end{proof}

Now we apply the proper push-forward property of the non-acyclicity classes to prove a cohomological version of the conductor formula \cite[Theorem 2.2.3]{Sai21} (also of Bloch's conjecture on the conductor formula for constant sheaves \cite{Blo87}).
\begin{theorem}\label{cor:conductorF}
	Let $Y$ be a smooth connected curve over  $k$. 
	Let    $f: X\to Y$ be a proper morphism over $k$ and $y\in|Y|$ a closed point. Let $\mathcal F\in D_{\rm ctf}(X,\Lambda)$. Assume that $f|_{X\setminus f^{-1}(y)}$ is  universally locally acyclic relatively to $\mathcal F|_{X\setminus f^{-1}(y)}$.
	Then we have 
	\begin{align}\label{eq:thm:conductorf}
		f_\ast \widetilde{C}_{X/Y/k}^{f^{-1}(y)}(\mathcal F)=-{ a}_y(R f_\ast\mathcal F)\quad {\rm in}\quad  \Lambda=H^0_{y}(Y,\mathcal K_{Y/k}).
	\end{align}
\end{theorem}
\begin{proof}
	We apply
	Proposition \ref{thm:ppForNTclass} to the following diagram (take $s=f$ and $f'={\rm id}$)
	\begin{align}\footnotesize
		\begin{gathered}
			\xymatrix@R=0.5cm{
				Z\ar@{=}[r]\ar@{_(->}[rd]&f^{-1}(y)\ar@{_(->}[d]\ar[rr]^-{f}\ar@{}[rrd]|\Box&&\{y\}\ar@{_(->}[d]\\
				&X\ar[rr]^-f\ar[rd]^f\ar[rdd]_-h&&Y\ar[ld]_{\rm id}\ar[ldd]^-{g}&\\
				&&Y\ar[d]^-(0.2)g&\\
				&&{\rm Spec}k.&
			}
		\end{gathered}
	\end{align}
	Then we get
	\begin{align}\label{eq:pushtocurveHelp}
		f_\ast \widetilde{C}_{X/Y/k}^{f^{-1}(y)}(\mathcal F)=C^{\{y\}}_{Y/Y/k}(Rf_\ast\mathcal F)\quad {\rm in}\quad  \Lambda=H^0_{y}(Y,\mathcal K_{Y/k}).
	\end{align}
	Since $f|_{X\setminus f^{-1}(y)}$ is proper and universally locally acyclic relatively to $\mathcal F|_{X\setminus f^{-1}(y)}$,  the cohomology sheaves of $Rf_\ast\mathcal F$ are locally constant on $Y\setminus \{y\}$.
	By the curve case of Theorem \ref{thm:MF}, we have  $C^{\{y\}}_{Y/Y/k}(Rf_\ast\mathcal F)=-{\rm dimtot}R\Phi_y(Rf_\ast\mathcal F,{\rm id})=-a_y(Rf_\ast\mathcal F)$. This proves \eqref{eq:thm:conductorf}.
\end{proof}
\begin{remark}\label{remark:localizedChern}
Let $f: X\to Y$ be a proper flat morphism of smooth schemes over $k$. Let $y\in|Y|$. Assume that ${\rm dim}X=n$, ${\rm dim}Y=1$ and that $X\setminus f^{-1}(y)\to Y\setminus \{y\}$ is smooth. We expect that
\begin{align}
\widetilde{C}_{X/Y/k}^{f^{-1}(y)}(\Lambda)=(-1)^n c_{n,X_y}^{X}(\Omega^1_{X/Y})\cap [X] \quad{\rm in}\quad H^0(f^{-1}(y),\mathcal K_{f^{-1}(y)/k}),
\end{align}
where $c_{n,X_y}^{X}(\Omega^1_{X/Y})$ is the localized Chern class defined in \cite[Section 1]{Blo87}. We have the following weaker equality:
	\begin{align}\label{eq:remark:localizedChern}
		f_\ast  (\widetilde{C}_{X/Y/k}^{f^{-1}(y)}(\Lambda))=f_\ast  ((-1)^n c_{n,X_y}^{X}(\Omega^1_{X/Y})\cap [X])\quad{\rm in}\quad H^0(Y,\mathcal K_{Y/k}).
	\end{align}
Indeed,  we have
\begin{align}
	\begin{split}
	f_\ast  (\widetilde{C}_{X/Y/k}^{f^{-1}(y)}(\Lambda))&\overset{\eqref{eq:pushtocurveHelp}}{=}C^{\{y\}}_{Y/Y/k}(f_\ast\Lambda)\overset{\eqref{eq:milclass1-1}}{=}C_{Y/k}(f_\ast\Lambda)-c_1(\Omega_{Y/k}^{1,\vee})\cap C_{Y/Y}(f_\ast\Lambda)\\
&=f_\ast C_{X/k}(\Lambda)- {\rm rank}(f_\ast\Lambda|_{Y\setminus\{y\}})\cdot  c_1(\Omega_{Y/k}^{1,\vee})\cap C_{Y/Y}(\Lambda)\\
&=f_\ast c_n(\Omega_{X/k}^{1,\vee})-{\rm rank}(f_\ast\Lambda|_{Y\setminus\{y\}})\cdot  c_1(\Omega_{Y/k}^{1,\vee}).		
	\end{split}
\end{align}
Then \eqref{eq:remark:localizedChern} follows from the above formula together with \cite[Lemma 2.1.3.(2) and Lemma 2.1.4]{Sai21}.
\end{remark}

\subsection{}\label{subsec:conjMnon}
We expect that the non-acyclicity class on a smooth variety can be calculated in
terms of the characteristic cycle. Let $X$ be a smooth scheme over  $k$.
Let  $\mathcal F\in D_{\rm ctf}(X,\Lambda)$ and let $SS(\mathcal F)$ be the singular support of $\mathcal F$. Let $Z\subseteq X$ be a closed subscheme and $f:X\to Y={\rm Spec}k[t]$ a morphism such that $f$ is $SS(\mathcal F)$-transversal outside $Z$.
By assumption, $df^{-1}(SS(\mathcal F))$ is supported on $(T^\ast Y\times_Y Z)\bigcup (T^\ast_YY\times_YX)$, 
where $df: T^\ast Y\times_YX\to T^\ast X$ is the canonical morphism induced by $f$. 
Consider the section $f^\ast dt: X\xrightarrow{dt}T^\ast Y\times_YX\xrightarrow{df} T^\ast X$.
Then the refined Gysin pull-back of the characteristic cycle $CC(\mathcal F)$ by $f^\ast dt$
\begin{align}\label{eq:refMilNumforNonIsolated}
	cc^Z_{X/Y/k}(\mathcal F):=(f^\ast dt)^!(CC(\mathcal F))
\end{align}
defines a zero cycle class on $Z$.  We expect the following Milnor type formula for non-isolated singular/characteristic points holds.
\begin{conjecture} We have an equality
	\begin{align}\label{MFforNonIsolated}
		\widetilde{C}^Z_{X/Y/k}(\mathcal F)=\widetilde{{\rm cl}}(cc^Z_{X/Y/k}(\mathcal F))\qquad{\rm in}\qquad H^0_Z(X,\mathcal K_{X/Y/k}),
	\end{align}
	where $\widetilde{{\rm cl}}$ is the composition
	$CH_0(Z)\xrightarrow{\rm cl} H^0_Z(X,\mathcal K_{X/k})\xrightarrow{\eqref{eq:distriK}} H^0_Z(X,\mathcal K_{X/Y/k})$.
\end{conjecture}
When $Z$ is a finite set of closed points, then \eqref{MFforNonIsolated} follows from Theorem \ref{thm:MF} and \cite[Theorem 5.9]{Sai17a}.

\subsection{}\label{subsec:gf}
In the rest of this section, we prove quasi-projective case of Saito's conjecture \cite[Conjecture 6.8.1]{Sai17a}. For the definitions of singular supports and characteristic cycles, we refer to \cite{Bei16} and \cite{Sai17a}. 
Let $X$ be a smooth scheme purely of dimension $d$ over $k$
and $\mathcal F\in D_{\rm ctf}(X,\Lambda)$. 
Let $Y$ be a smooth connected and projective curve over $k$.
Assume that $f: X\to Y$ is a  good fibration with respect to the singular support 
$SS(\mathcal F)$ of $\mathcal F$, and $Z\coloneq\{x_v\}_{v\in\Sigma}$ is the set of isolated characteristic points of $f$ and $\Sigma\subseteq Y$, i.e., the following conditions hold (cf. \cite[Definition 3.4.2]{UYZ})
\begin{enumerate}
	\item There exist finitely many closed points $x_v (v\in\Sigma)$ of $X$ such that $f$ is 
	$SS(\mathcal F)$-transversal on $X\setminus Z$ (thus $X\setminus Z\to Y$ is  universally locally acyclic  relatively to $\mathcal F|_{X\setminus Z}$).  
	
	\item  For each $v\in\Sigma$, $f(x_v)=v$. If $x_v\neq x_u$, then $v\neq u$ (each fiber $X_v$ contains at most one isolated characteristic point).
	
\end{enumerate}
Since $\dim Z=0$, the relative cohomological characteristic class $C_{X/Y}(\mathcal F)$ is 
well-defined (cf. Definition \ref{def:rcccCodBig}).
By Theorem \ref{thm:SJ:MFforTclass}, 
we have 
\begin{align}\label{eq:fibforC1}
	C_{X/k}(\mathcal F)=c_1(f^\ast\Omega^{1,\vee}_{Y/k})\cap C_{X/Y}(\mathcal F)-\sum_{v\in \Sigma}{\rm dimtot} R\Phi_{\bar{x}_v}(\mathcal F,f)\cdot [x_v]\quad{\rm in}\quad H^0(X,\mathcal K_{X/k}).
\end{align}
Let $\omega$ be a non-zero rational 1-form on the curve $Y$ such that $\omega$ has 
neither zeros nor poles on $\Sigma$. 
Let $U=X\setminus Z$. Then
the image of $C_{X/Y}(\mathcal F)$ by the map $H^0(X,\mathcal K_{X/Y})\xrightarrow{\simeq} H^0(U,\mathcal K_{U/Y})$ equals
$C_{U/Y}(\mathcal F)$. 
We have an equality in $H^0(Y,\mathcal K_{Y/k})$
\begin{align}\label{eq:fibforC2}
	c_1(\Omega^{1,\vee}_{Y/k})=-\sum_{v\in |Y\setminus\Sigma|}{\rm ord}_v(\omega)\cdot [v].
\end{align}
By Proposition \ref{cor:pullbackccc}, we have
\begin{align}\label{eq:fibforC3}
	\begin{split}
		c_1(f^\ast\Omega^{1,\vee}_{Y/k})\cap C_{X/Y}(\mathcal F)=c_1(f^\ast\Omega^{1,\vee}_{Y/k})\cap C_{U/Y}(\mathcal F|_{U})
		=-\sum_{v\in |Y\setminus\Sigma|}{\rm ord}_v(\omega)\cdot
		C_{X_v/v}(\mathcal F|_{X_v}).
	\end{split}
\end{align}
By \eqref{eq:fibforC1} and \eqref{eq:fibforC3}, we get the following induction formula for the cohomological characteristic classes:
\begin{align}
	\label{eq:conj:Milnorclass2}C_{X/k}(\mathcal F)&=-\sum_{v\in |Y\setminus\Sigma|}{\rm ord}_v(\omega)\cdot
	C_{X_v/v}(\mathcal F|_{X_v})-\sum_{v\in \Sigma}{\rm dimtot} R\Phi_{\bar{x}_v}(\mathcal F,f)\cdot [x_v].
\end{align}
The formula \eqref{eq:conj:Milnorclass2} implies that the  characteristic class $C_{X/k}(\mathcal F)$  can be built from characteristic classes $C_{X_v/v}(\mathcal F|_{X_v})$  on schemes $X_v$ of  dimension smaller than $X$. Based on this idea, let us explain that Theorem \ref{thm:SJ:MFforTclass} implies the quasi-projective case of Saito's conjecture \cite[Conjecture 6.8.1]{Sai17a}\footnote{We refer to \cite{UYZ} for a weak version of  Saito's another conjecture \cite[Conjecture 6.8.2]{Sai17a}.}.
\begin{conjecture}[Saito, {\cite[Conjecture 6.8.1]{Sai17a}}]\label{conj:Scc}
	Let $X$ be a closed sub-scheme of a smooth scheme over $k$. Let $cc_{X}:=cc_{X,0}: K_0(X,\Lambda)\to {\rm CH}_0(X)$ be the morphism defined in \cite[Definition 6.7.2]{Sai17a}.
	Then we have a commutative diagram
	\begin{align}
		\begin{gathered}
			\xymatrix@R=0.5cm{
				K_0(X,\Lambda)\ar[r]^-{cc_X}\ar[rd]_-{C_{X/k}}&{\rm CH}_0(X)\ar[d]^-{\rm cl}\\
				&H^0(X,\mathcal K_{X/k}),
			}
		\end{gathered}
	\end{align}
	where ${\rm cl}: {\rm CH}_0(X)\to H^0(X,\mathcal K_{X/k})$ is the cycle class map.
\end{conjecture}
\begin{theorem}\label{thm:sConj}
	Conjecture \ref{conj:Scc} holds for  any smooth and quasi-projective scheme $X$ over $k$.
\end{theorem}
\begin{proof}
	We prove the theorem by induction on the dimension $d={\rm dim}(X)$.
	If $d=0$, then both $C_{X/k}$ and $cc_X$ equal the rank function.
	If $d=1$, the result follows from  Corollary \ref{cor:GOSend} and the  Grothendieck-Ogg-Shafarevich formula (cf. \cite[Lemma 5.11.3]{Sai17a}). Suppose that $d\geq 1$ and that $X$ has a good fibration $f:X\to \mathbb P^1$ with respect to $SS(\mathcal F)$. We use the notation in \ref{subsec:gf}. By \cite[Proposition 5.3.7]{UYZ}, the characteristic class $cc_X(\mathcal F)$ satisfies the following fibration formula
	\begin{align}
		\label{eq:fibforcc}cc_{X}(\mathcal F)&=-\sum_{v\in |Y\setminus\Sigma|}{\rm ord}_v(\omega)\cdot
		cc_{X_v}(\mathcal F|_{X_v})-\sum_{v\in \Sigma}{\rm dimtot} R\Phi_{\bar{x}_v}(\mathcal F,f)\cdot [x_v]\quad{\rm in}\quad {\rm CH}_0(X).
	\end{align}
	By the induction hypothesis and \eqref{eq:conj:Milnorclass2}, we get ${\rm cl}(cc_{X}(\mathcal F))=C_{X/k}(\mathcal F)$.  Thus the result holds when $X$ has a good fibration.
	
	In general, for a finite Galois extension $k'/k$ of degree invertible in $\Lambda$, assume the result holds for $(X_{k'},\mathcal F|_{X_{k'}})$. We show the result also holds for $(X,\mathcal F)$.
	Indeed, let $\sigma:X_{k'}\to X$ be the projection. Since $\sigma_\ast\sigma^\ast\mathcal F=\mathcal F^{\oplus{{\rm deg}(k'/k)}}$, we get
	\begin{align}
		\begin{split}
			{\rm deg}(k'/k)\cdot {\rm cl}(cc_X(\mathcal F))&={\rm cl}(cc_X(\sigma_\ast\sigma^\ast\mathcal F))\\
			&\overset{(a)}{=}\sigma_\ast({\rm cl}(cc_{X_{k'}}(\sigma^\ast\mathcal F)))\\
			&=\sigma_\ast(C_{X_{k'}/k}(\sigma^\ast\mathcal F))\overset{\ref{prop:pushccc}}{=}C_{X/k}(\sigma_\ast\sigma^\ast\mathcal F)={\rm deg}(k'/k)\cdot C_{X/k}(\mathcal F),
		\end{split}
	\end{align}
	where $(a)$ follows from \cite[Lemma 2]{Sai18}.
	Since ${\rm deg}(k'/k)$ is invertible in $\Lambda$, hence ${\rm cl}(cc_X(\mathcal F))=C_X(\mathcal F)$.
	
	By \cite[Lemma 4.2.7]{UYZ}, after taking a finite extension of $k$ if necessary, we may assume  
	there exist a blow-up $\pi:  \widetilde{X}\to X$ of $X$ along a closed subscheme of codimension 2 and a good fibration $f: \widetilde{X}\to \mathbb P^1$.
	Since $ \widetilde{X}$ has a good fibration, the result holds for $(\widetilde{X},\pi^\ast\mathcal F)$. By Lemma \ref{lem:Bup}, the result also holds for $(X,\mathcal F)$. This finishes the proof.
\end{proof}

\end{document}